\documentclass[10pt,reqno]{amsart}

\usepackage[T1]{fontenc}
\usepackage{lmodern}
\usepackage{amsmath,amssymb,mathtools}
\usepackage{microtype}
\usepackage[letterpaper,textwidth=5.7in,textheight=8.55in,centering]{geometry}
\usepackage{enumitem}
\usepackage{xcolor}
\usepackage[numbers,square]{natbib}
\usepackage[colorlinks=true,linkcolor=blue!55!black,citecolor=blue!55!black,urlcolor=blue!55!black]{hyperref}
\usepackage{comment}
\hypersetup{
pdftitle={Finiteness and exponential growth of full graph p-spectra},
pdfauthor={Matthew J. Colbrook},
pdfsubject={Full spectra of finite graph p-Laplacians}
}

\setlist[itemize]{leftmargin=2em,itemsep=0.25em,topsep=0.4em}
\setlist[enumerate]{leftmargin=2em,itemsep=0.25em,topsep=0.4em}

\newtheorem{theorem}{Theorem}[section]
\newtheorem{proposition}[theorem]{Proposition}
\newtheorem{corollary}[theorem]{Corollary}
\newtheorem{lemma}[theorem]{Lemma}
\theoremstyle{definition}
\newtheorem{example}[theorem]{Example}
\theoremstyle{remark}
\newtheorem{remark}[theorem]{Remark}

\newcommand{\R}{\mathbb{R}}
\newcommand{\Sph}{\mathbb{S}}
\newcommand{\Crit}{\operatorname{Crit}}
\newcommand{\Spec}{\operatorname{Spec}}
\newcommand{\sgn}{\operatorname{sgn}}
\newcommand{\od}{\mathbin{\odot}}
\newcommand{\cE}{\mathcal{E}}
\newcommand{\cM}{\mathcal{M}}
\newcommand{\cQ}{\mathcal{Q}}
\newcommand{\inner}[2]{\left\langle #1,#2\right\rangle}
\newcommand{\one}{\mathbf 1}
\newcommand{\Cau}{\operatorname{Cau}}
\newcommand{\Maxw}{\operatorname{Maxw}}
\newcommand{\dist}{\operatorname{dist}}
\newcommand{\indm}{\operatorname{ind}}
\newcommand{\Pj}{\mathbb P}
\newcommand{\Cn}{\mathcal C_n}
\newcommand{\codim}{\operatorname{codim}}

\title[Finiteness and growth of full graph $p$-spectra]
{Finiteness and exponential growth of\\
full graph $p$-spectra}

\author{Matthew J. Colbrook}
\address{Department of Applied Mathematics and Theoretical Physics,
University of Cambridge, Wilberforce Road, Cambridge CB3 0WA, United Kingdom}
\email{m.colbrook@damtp.cam.ac.uk}


\begin{document}

\begin{abstract}
We prove that the full spectrum of the graph $p$-Laplacian is finite for every finite graph and every real $p>1$. This resolves the problem of finiteness of the full graph spectrum posed by Amghibech (2003). More generally, for signed weighted graph $p$-Schr\"odinger operators with arbitrary real potentials, we obtain bounds in terms of the numbers of vertices and positive-weight edges, uniform in all coefficients and in $p$. On $n$ vertices with $m$ positive-weight edges, the logarithms of the spectral cardinality and of the total number of connected components of the normalised eigenvector sets are $O((n+m)^2)$; the sum over the full spectrum of their rational Betti numbers is at most $\exp(C(n+m)^2)$ for an absolute constant $C$. These topological bounds extend to homogeneous eigenproblems built from arbitrary finite families of linear forms, including generalised $p$-eigenvalue problems for matrix pairs.

For the uniformly weighted $K_n$ and $p\ne2$, we identify the nonconstant eigenlines with the barycentres of the coordinate-hyperplane cell decomposition of $\mathbb RP^{n-2}$ and determine their local Morse data on either side of $p=2$. A renormalised logarithmic limit describes the transition at $p=2$. Positive integer weights then separate these critical values, giving at least $(3^n-2^{n+1}+3)/2$ distinct spectral values for every fixed $p\ne2$ and every $n\geq2$. By contrast, the maximal spectral cardinality at $p=2$ is $n$; at $p=4$, the maximal spectral cardinalities in the unsigned and general classes both have exponential growth rate exactly $3$. For $n\geq3$ the same examples resolve Amghibech's extremal question. The proof combines o-minimal and Pfaffian geometry with projective $L^p$-duality, Morse theory, critical groups, and tensor eigenvalue bounds.
\end{abstract}

\maketitle
\vspace{-10mm}
\setcounter{tocdepth}{1}
\tableofcontents

\linespread{0.97}

\section{Introduction}

For a linear operator on an $n$-dimensional space, finiteness of the spectrum follows from the characteristic polynomial and the number of distinct eigenvalues is at most $n$. The full spectrum of a graph $p$-Laplacian is a different object: it is the entire critical-value set of a Rayleigh quotient, nonquadratic when $p\ne2$, and can contain values outside the $n$-term sequence produced by the Krasnosel'skii genus min--max construction. Its critical locus may be singular or positive-dimensional. Finite dimensionality and bounds for isolated roots therefore do not decide whether the full spectrum is finite.

The graph-specific counting problem was posed explicitly by Amghibech: the number of $p$-Laplacian eigenvalues was unknown for a general graph \cite[p.~299]{Amghibech2003}. The problem has since been repeatedly highlighted as a fundamental open question in nonlinear spectral graph theory \cite[Introduction]{Zhang2025}, \cite[pp.~6--7]{DeiddaThesis2023}, \cite[Introduction]{GeQin2025}, and \cite[Section~3.1]{DeiddaTudiscoZhang2025}. Theorem~A resolves it for every finite graph and every real $p>1$ and gives explicit bounds uniform in the graph coefficients and in $p$. Amghibech also asked whether an $n$-vertex graph can have more positive eigenvalues than $K_n$ \cite[Question~3, p.~300]{Amghibech2003}. For every $p>1$, $p\ne2$, and every $n\geq3$, Theorem~C answers this in the stronger form that the underlying graph can be complete and all its edge weights can be positive integers.

Theorems~B and~C reveal a sharp change of scale at $p=2$. The maximal cardinality of the full spectrum on $n$ vertices is $n$ at $p=2$, but for every fixed $p\ne2$ there are complete graphs with positive integer edge weights whose full spectra contain exponentially many distinct values; at $p=4$ this exponential rate is exactly $3$, both in the unsigned class and in the full signed weighted class.

The proofs bring together o-minimal and Pfaffian geometry, projective $L^p$-duality, the topology of a projectivised hyperplane arrangement, Morse theory and critical groups, and symmetric-tensor eigenvalue bounds.

\subsection{Main results}

We now state the three principal results. Let $n\geq1$ and let
$$
 a_{ij}=a_{ji}\geq0,\qquad
 \varepsilon_{ij}=\varepsilon_{ji}\in\{-1,1\}\quad(i\ne j),
 \qquad b_i\in\R,\qquad \nu_i>0,
$$
where the signature on a pair with $a_{ij}=0$ is immaterial. We call these choices admissible. Set $\Phi_p(t)=|t|^{p-2}t$, with $\Phi_p(0)=0$. Define
\begin{equation}\label{eq:graph-p-lap-intro}
 (\Delta_{p,a,\varepsilon,b}x)_i
 =\sum_{j\neq i}a_{ij}\Phi_p(x_i-\varepsilon_{ij}x_j)
+b_i\Phi_p(x_i).
\end{equation}
The generalised eigenproblem is
\begin{equation}\label{eq:eigen-intro}
 \Delta_{p,a,\varepsilon,b}x
 =\lambda\bigl(\nu_i\Phi_p(x_i)\bigr)_{i=1}^n,
 \qquad x\in\R^n\setminus\{0\},
\end{equation}
and $\sigma_p(a,\varepsilon,b,\nu)$ denotes the set of all real $\lambda$ for which it has a solution. Set
$$
 d=(a,\varepsilon,b,\nu),
 \qquad \beta(d)=\min_i\frac{b_i}{\nu_i},
$$
and for $\lambda\in\sigma_p(d)$, let
$$
 E(d,p,\lambda)=
 \left\{x\in\Sph^{n-1}:
 \Delta_{p,a,\varepsilon,b}x
 =\lambda\bigl(\nu_i\Phi_p(x_i)\bigr)_{i=1}^n\right\}.
$$
Let $\overline E(d,p,\lambda)$ be its image in $\mathbb RP^{n-1}$. We set $b(X;\mathbb Q)=\sum_{j\geq0}b_j(X;\mathbb Q)$. For $K\geq1$ set
$$
 \mathfrak B(K)
 =\sum_{q=1}^{K}\binom{K}{q}2^q
2^{2q^2+q+2}(4q+3)^{2q}.
$$
Also set
$
 N_n=\frac{n(n+1)}2$ and $\widehat{\mathfrak B}_n=\mathfrak B(N_n)$.

\begin{theorem}[Theorem A: full-spectrum finiteness and uniform quantitative bounds]
\label{thm:intro-main}
For every $n\geq1$ the following hold.
\begin{enumerate}[label=\textup{(\roman*)}]
\item\label{part:intro-finite} For every $p>1$ and every admissible choice of $a,\varepsilon,b,\nu$, the spectrum $\sigma_p(a,\varepsilon,b,\nu)$ is finite.

\item\label{part:intro-uniform} There exists $\mathfrak B_n^{\mathrm{unif}}<\infty$ such that
$$
 \#\sigma_p(a,\varepsilon,b,\nu)\leq\mathfrak B_n^{\mathrm{unif}}
$$
simultaneously for every $p>1$ and all admissible graph data.

\item\label{part:intro-effective} One may take $\mathfrak B_n^{\mathrm{unif}}=\widehat{\mathfrak B}_n$, and
\begin{equation}\label{eq:explicit-Bn}
 \#\sigma_p(a,\varepsilon,b,\nu)\leq\widehat{\mathfrak B}_n
 \leq3^{N_n}2^{2N_n^2+N_n+2}(4N_n+3)^{2N_n}.
\end{equation}
In particular, the displayed choice satisfies $\log\widehat{\mathfrak B}_n=O(n^4)$, and
$$
 \sum_{\lambda\in\sigma_p(d)}b_0(E(d,p,\lambda))
 \leq\widehat{\mathfrak B}_n.
$$
\end{enumerate}
\end{theorem}

Bounds in terms of the support graph of the edge weights, namely the graph consisting of the positive-weight edges, sharpen the universal estimate~\eqref{eq:explicit-Bn}.

\begin{corollary}[Quantitative and topological refinements]
\label{thm:intro-effective-refinements}
Let $d=(a,\varepsilon,b,\nu)$ be admissible, with $n$ vertices and $m$ positive-weight edges, and set $K=n+m$. For every real $p>1$, each of
$
 \#\sigma_p(d)$ and $
 \sum_{\lambda\in\sigma_p(d)}b_0(E(d,p,\lambda))
$
is at most $\mathfrak B(K)$; in particular, the logarithm of either quantity is $O(K^2)$. There is an absolute constant $C$ such that
$$
 \sum_{\lambda\in\sigma_p(d)}
 b(E(d,p,\lambda);\mathbb Q)\leq \exp(CK^2),
$$
and the same bounds hold with $E(d,p,\lambda)$ replaced by $\overline E(d,p,\lambda)$. If $p=A/B>1$, where $A>B\geq1$ are coprime integers, and $D=2\max\{B,A-B,1\}$, then
$$
 \#\sigma_p(d)\leq
 \sum_{\lambda\in\sigma_p(d)}
 b(E(d,p,\lambda);\mathbb Q)
 \leq (CK^2D)^{3K+1}.
$$
These estimates include singular and positive-dimensional eigenvector sets.
\end{corollary}

Theorem~A provides the universal finiteness and upper-bound theory. We now turn to the complete-graph geometry from which the exponential lower bounds arise.

Let $U_n(p)$ denote the largest spectral cardinality among unsigned, zero-potential, unit-measure weighted graphs on $n$ vertices, and let $B_n(p)$ be the corresponding maximum over all admissible signed weighted graph data with arbitrary real potentials and positive vertex measures. Set
$$
 R_n=\frac{3^n-2^{n+1}+3}{2}.
$$
For a connected unsigned graph with zero potential, the constant eigenline is the unique eigenline for eigenvalue zero and disappears under translation reduction. Thus Theorem~B counts $R_n-1$ reduced critical points, both for $p\ne2$ and for the logarithmic limiting functional at $p=2$, whereas Theorem~C includes the zero value and obtains at least $R_n$ values in the full spectrum.

Let $V_n=\mathbb R^n/\mathbb R\one$. For $0\ne\bar x\in V_n$, choose any representative $x\in\mathbb R^n$, let $\kappa_p(\bar x)$ be the unique representative satisfying $\sum_i\Phi_p(\kappa_p(\bar x)_i)=0$, and define
$$
 \mathfrak m_p(\bar x)=\sum_i|\kappa_p(\bar x)_i|^p,
 \qquad
 \Theta_p([\bar x])=[\Phi_p(\kappa_p(\bar x))],
$$
$$
 \Psi_p([\bar x])=
 \left(\frac{|x_i-x_j|^p}{\mathfrak m_p(\bar x)}\right)_{i<j},
 \qquad
 \mathcal R_{p,a}([\bar x])=a\cdot\Psi_p([\bar x]).
$$

\begin{theorem}[Theorem B: complete-graph cells and the transition at $p=2$]
\label{thm:intro-cellular-transition}
Let $n\geq2$. For every $p>1$, the projectivised $L^p$-duality map
$$
 \Theta_p:\Pj(\mathbb R^n/\mathbb R\one)
 \rightarrow\Pj(\one^\perp)
$$
is a definable homeomorphism. The projectivised coordinate-hyperplane arrangement on the target is a regular CW decomposition $\Cn$, and the nonconstant eigenlines of the uniformly weighted $K_n$ for $p\ne2$ are precisely the inverse images of its cell barycentres.

If $p>2$, the inverse image under $\Theta_p$ of the barycentre of a cell $C$ is a nondegenerate critical point of $\mathcal R_{p,\one}$, with Morse index $\dim C$; whenever $C$ is a proper face of a cell $D$, the critical value associated with $C$ is smaller than that associated with $D$.

Identify $V_n$ with $\one^\perp$ by choosing the Euclidean mean-zero representative. Under this identification, as $p\to2$,
$$
 \frac{\mathcal R_{p,\one}-n}{p-2}
 \rightarrow\mathcal H_n
 \quad\hbox{in }C^1(\mathbb RP^{n-2}),
$$
where, for $0\ne x\in\one^\perp$,
$$
 \mathcal H_n([x])=
 \frac{\sum_{i<j}(x_i-x_j)^2\log|x_i-x_j|
-n\sum_i x_i^2\log|x_i|}
{\sum_i x_i^2},
 \qquad 0^2\log0:=0.
$$
The critical points of $\mathcal H_n$ are the cell barycentres. For every coefficient field $\mathbb F$, the local critical group at the barycentre of $C$ is $\mathbb F$ in degree $\dim C$ and zero in every other degree; whenever $C$ is a proper face of a cell $D$, its value at the barycentre of $C$ is smaller than its value at the barycentre of $D$.

If $1<p<2$, every eigenline of the uniformly weighted $K_n$ corresponding to a cell $C$ is an isolated critical point. For every coefficient field $\mathbb F$, its local critical group is $\mathbb F$ in degree $\codim C$ and zero in every other degree. The local Brouwer degree of its gradient in a local chart is $(-1)^{\codim C}$, and whenever $C$ is a proper face of a cell $D$, its value at the eigenline associated with $C$ is larger than its value at the eigenline associated with $D$.

The polynomials recording these local critical groups are
\begin{align*}
 M_n^+(t)
 &=\sum_{C\in\Cn}t^{\dim C}
 =\frac{(1+2t)^n-2(1+t)^n+1}{2t^2},\\
 M_n^{\log}(t)&=M_n^+(t),\\
 M_n^-(t)
 &=\sum_{C\in\Cn}t^{\codim C}
 =t^{n-2}M_n^+(t^{-1})
 =\frac{(t+2)^n-2(t+1)^n+t^n}{2}.
\end{align*}
Evaluating any of these polynomials at $t=1$ gives $R_n-1$, the number of nonconstant critical points of the uniformly weighted $K_n$ for $p\ne2$ and of critical points of $\mathcal H_n$ at $p=2$.
\end{theorem}

Theorem~C gives exponential spectral growth and describes the parameter regions in which the critical-point count persists.

\begin{theorem}[Theorem C: exponential spectral growth and parameter regions]
\label{thm:intro-growth-chambers}
The following statements hold. For parts~\textup{(iv)}--\textup{(vi)}, fix $n\geq2$ and set $m=\binom n2$.
\begin{enumerate}[label=\textup{(\roman*)}]
\item For every $n\geq2$, $U_n(2)=B_n(2)=n$.

\item For every $n\geq2$ and every $p>1$ with $p\ne2$,
\begin{equation}\label{eq:intro-exponential-growth}
 U_n(p)\geq R_n=\frac{3^n-2^{n+1}+3}{2}.
\end{equation}
The graphs may be chosen to be connected complete graphs with positive integer edge weights. After ordering the $\binom n2$ edges as $e=0,\ldots,\binom n2-1$, for each fixed pair $(n,p)$ there is an integer $T=T(n,p)\geq2$ such that the weights may be chosen in the form
$
 A_e=Q+T^e
$
for every sufficiently large integer $Q$. Consequently, for every fixed $p>1$ with $p\ne2$,
$$
 \liminf_{n\to\infty}U_n(p)^{1/n}\geq3.
$$

\item At $p=4$,
\begin{equation}\label{eq:intro-fourth-order-rate}
 \lim_{n\to\infty}U_n(4)^{1/n}
 =\lim_{n\to\infty}B_n(4)^{1/n}=3.
\end{equation}

\item There are relatively closed definable caustic and Maxwell sets of codimension at least one in $(2,\infty)\times(0,\infty)^m$. Off the caustic, the critical-point count and index multiset are constant on each connected component; off its union with the Maxwell set, all critical values are distinct. The line $\{(p,\one):p>2\}$ lies in one component of the caustic complement. On every compact $J\Subset(2,\infty)$, there is a positive integer edge-weight vector $a$ for which $\mathcal R_{p,a}$ is Morse with Morse polynomial $M_n^+$ for every $p\in J$; its critical values are pairwise distinct except at finitely many $p\in J$.

\item For every compact $J\Subset(1,2)$, there is a positive integer edge-weight vector such that the resulting graph has at least $R_n$ eigenlines for every $p\in J$ and at least $R_n$ distinct spectral values for all but finitely many $p\in J$. For each fixed $p\in(1,2)$, there is a neighbourhood $V_p$ of $\one$ in $(0,\infty)^m$ such that the weights in $V_p$ for which the graph has fewer than $R_n$ distinct values in the full spectrum lie in a relatively closed definable subset of codimension at least one.

\item For every fixed $\eta\in\mathbb R^m$, under $a(p)=\one+(p-2)\eta$,
$$
 \frac{\mathcal R_{p,a(p)}-n}{p-2}
 \rightarrow\mathcal H_n+\eta\cdot\Psi_2
 \quad\hbox{in }C^1(\mathbb RP^{n-2}).
$$
There are an open neighbourhood $W\ni0$ and a relatively closed definable set $\mathfrak M_n^{\mathrm{tan}}\subset W$ of codimension at least one such that, for every $\eta\in W\setminus\mathfrak M_n^{\mathrm{tan}}$, $\mathcal H_n+\eta\cdot\Psi_2$ has at least $R_n-1$ distinct critical values, and there is $\delta=\delta(\eta)>0$ such that the path $a(p)=\one+(p-2)\eta$ has at least $R_n$ distinct values in the full spectrum whenever $0<|p-2|<\delta$.
\end{enumerate}
\end{theorem}

\subsection{Consequences and interpretation}

Beyond finiteness, Theorems~A--C give uniform quantitative and topological bounds for signed weighted $p$-Schr\"odinger operators, a complete cellular description of the nonconstant eigenlines of the uniformly weighted complete graph, and exponential spectral growth under positive integer perturbations. The consequences below connect these results to Amghibech's extremal question, the Krasnosel'skii genus min--max sequence, and symmetric tensor eigenvalue theory.

Theorem~\ref{thm:intro-growth-chambers}\textup{(ii)} has an immediate extremal consequence. The graphs constructed there are connected and hence have at least $R_n-1$ distinct positive spectral values. By contrast, the positive values of the uniformly weighted $K_n$ are indexed by unordered pairs of positive block sizes $(k,\ell)$ with $k+\ell\leq n$; hence there are at most $\lfloor n^2/4\rfloor$ of them. Since
$$
 R_n-1>\left\lfloor\frac{n^2}{4}\right\rfloor
 \qquad(n\geq3),
$$
the theorem answers Amghibech's Question~3 \cite[p.~300]{Amghibech2003} for every $p>1$, $p\ne2$, using graphs with positive integer edge weights.

At $p=2$ one has the identity
$$
 \sum_{i<j}(x_i-x_j)^2
 =n\min_{c\in\mathbb R}\sum_i|x_i+c|^2.
$$
The reduced Rayleigh quotient of the uniformly weighted $K_n$ is therefore constant at $p=2$, whereas a generic perturbation of the edge weights gives a quadratic quotient with exactly $n-1$ critical points, the minimum allowed by the topology of $\mathbb RP^{n-2}$. By contrast, $\mathcal H_n$ has the $R_n-1$ barycentres of $\Cn$ as its critical points. For $p>2$, their reduced Morse indices are the cell dimensions; for $1<p<2$, their local critical groups are concentrated in the cell codimensions. Krasnosel'skii's perturbation theory established stability of variational critical values of even functionals on the sphere \cite{Krasnoselskii1955}. Here every nonconstant projective critical point of the complete-graph model has nonzero local degree. Each therefore persists under sufficiently small rational perturbations of the edge weights, for both $1<p<2$ and $p>2$.

The cellular count also quantifies the number of spectral values outside the Krasnosel'skii genus min--max sequence. With $P_n(t)=\sum_{j=0}^{n-2}t^j$, the $\mathbb F_2$-Poincar\'e polynomial of $\mathbb RP^{n-2}$, the Morse relations are
$$
 M_n^\pm(t)=P_n(t)+(1+t)Q_n^\pm(t),
 \qquad
 Q_n^-(t)=t^{n-3}Q_n^+(t^{-1})\quad(n\geq3),
$$
where the $Q_n^\pm$ have nonnegative integer coefficients and $2Q_n^\pm(1)=R_n-n$. Consequently, for $p>2$, if $(p,a)$ lies in the component of the joint caustic complement containing the curve $\{(q,\one):q>2\}$ and the reduced critical values at $(p,a)$ are distinct, then the corresponding graph has at least $R_n-n$ spectral values outside the $n$-term Krasnosel'skii genus min--max sequence. Parts~\textup{(iv)} and \textup{(v)} of Theorem~\ref{thm:intro-growth-chambers} make this conclusion uniform on every compact exponent interval contained in either $(1,2)$ or $(2,\infty)$, apart from finitely many value collisions.

For even integral exponents, the problem also lies in symmetric tensor eigenvalue theory, introduced independently by Lim and Qi \cite{Lim2005,Qi2005}. With the Euclidean normalisation used for E-eigenvalues, Cartwright and Sturmfels proved that a symmetric tensor has only finitely many normalised eigenvalues and counted the eigenpairs of a generic tensor \cite{CartwrightSturmfels2013}; that normalisation differs from the generalised graph eigenproblem considered here. In the graph setting, the tensor-pencil reduction of \cite[Proposition~3.1]{GeQin2025} has a nonsingular positive diagonal mass tensor, so the pencil is regular and has no eigenvalue at infinity. Combining this reduction with \cite[Theorem~2.1]{DingWei2015} gives
\begin{equation}\label{eq:intro-even-tensor}
 \#\sigma_p(a,\varepsilon,b,\nu)\leq n(p-1)^{n-1}
 \qquad(p\geq2\text{ even}).
\end{equation}
The graph tensors are highly structured and may have singular or positive-dimensional eigenvector sets. The tensor estimate applies at each fixed even exponent, whereas the topological bounds proved here are uniform over all real $p>1$.

For $p=4$, \eqref{eq:intro-even-tensor} and Theorem~\ref{thm:intro-growth-chambers}\textup{(ii)} give
$$
 R_n\leq U_n(4)\leq B_n(4)\leq n3^{n-1}.
$$
Thus both extremal exponential growth constants are $3$; the displayed upper and lower bounds differ only by a factor of order $n$.

Beyond pointwise finiteness, the full spectrum and its eigenvector sets admit a uniform global description over coefficient and exponent space, including in the endpoint regimes.

\begin{corollary}[Further global consequences]
\label{cor:intro-global-consequences}
For each fixed $n$, the space of coefficients and exponents admits a finite definable partition. On each cell, the spectral cardinality is constant, the ordered spectral values are continuous definable functions, and the homeomorphism types of the corresponding ambient pairs in $\mathbb S^{n-1}$ and $\mathbb RP^{n-1}$ are constant. For fixed graph data, the compact sets of normalised eigenpairs have one-sided Hausdorff limits as $p\downarrow1$; under the reparametrisation $\lambda=\beta(d)+\rho^p$ (with $\rho>0$), they also have Hausdorff limits as $p\to\infty$. The Clarke critical-value sets of the endpoint quotients restricted to the sphere are finite, uniformly in fixed dimension.
\end{corollary}

Finally, the effective argument also applies to finite families of linear forms. In particular, for the generalised $p$-eigenvalue problem for the matrix pair $(A,C)$
\begin{equation}\label{eq:matrix-intro}
 A^{\mathsf T}\bigl(\mu\od\Phi_p(Ax)\bigr)
 =\lambda C^{\mathsf T}\bigl(\nu\od\Phi_p(Cx)\bigr),
\end{equation}
the numerator coefficients $\mu$ may be arbitrary, the denominator coefficients $\nu$ are positive, and $C$ is injective. Fixed matrix sizes again give explicit bounds, uniform in all data and $p>1$, for the number of distinct values and the total number of components of the eigenvector sets.

\subsection{Ideas of the proofs}

Two issues must be addressed: the full critical-value set may be singular or positive-dimensional, and the reduced quotient of the uniformly weighted $K_n$ becomes constant at $p=2$. For Theorem~A, we treat the eigenproblem, jointly in the exponent and all coefficients, as the critical-value problem for a definable $C^1$ Rayleigh quotient. O-minimal critical-value theory gives finiteness, and definable triviality gives a uniform bound. The effective estimates use the Pfaffian theory initiated by Khovanskii \cite{Khovanskii1991}. On each stratum determined by the signs and vanishing pattern of the linear forms, logarithmic coordinates give a Pfaffian system whose format depends only on the number of active linear forms, and hence only on $n+m$ in the graph problem. Quantitative component and Betti bounds for semi-Pfaffian sets \cite{GabrielovVorobjov2004}, together with a semialgebraic Betti-number estimate at rational exponents \cite{GabrielovVorobjov2005}, yield the explicit estimates.

The proofs of Theorems~B and~C begin with the projectivised $L^p$-duality map after minimisation over translations. Hessian calculations give the Morse indices for $p>2$. For $1<p<2$, and for the limiting logarithmic functional at $p=2$, critical-group calculations give the degrees in which the local groups are supported. For $p>2$, caustic and Maxwell sets describe degeneracy and critical-value collisions; for $1<p<2$, nonzero local degree gives persistence and definable collision sets give generic separation. Rational approximation and scaling produce positive integer weights. At $p=4$, a tensor-pencil estimate gives the upper bound with the matching exponential rate.

\subsection{Historical context and relation to earlier work}

The construction of variational eigenvalues and the problem of bounding the full nonlinear spectrum developed as distinct strands. Browder, Amann, Fu\v{c}\'{i}k and Ne\v{c}as developed extensions of Ljusternik--Schnirelmann theory for nonlinear eigenvalue problems \cite{Browder1965,Amann1972,FucikNecas1972}. These methods construct min--max critical levels; by themselves they do not bound the entire critical-value set. An early direct study of that set is Pokhozhaev's paper, which proves that the critical-value set has Lebesgue measure zero for sufficiently smooth Fredholm functionals on separable reflexive Banach spaces under a uniform finite-nullity hypothesis \cite[Theorem, p.~94]{Pokhozhaev1968}.

In the early 1970s, a systematic programme on upper bounds emerged. For a class of homogeneous nonlinear second-order Sturm--Liouville equations under explicit positivity hypotheses, Ne\v{c}as proved that the full spectrum is a discrete sequence tending to infinity \cite[Theorem, p.~1045]{Necas1971}; Kratochv\'{i}l and Ne\v{c}as obtained the corresponding fourth-order result \cite[Theorem, p.~648]{KratochvilNecas1971}. J.~Sou\v{c}ek and V.~Sou\v{c}ek proved that a real-analytic function has only finitely many critical values attained in any fixed compact subset of its domain \cite[Theorem~1]{SoucekSoucek1972}. At even integral exponents, the graph Rayleigh quotient is real analytic away from the origin, so this theorem already gives qualitative finiteness of the full graph spectrum in that regime. Fu\v{c}\'{i}k, Ne\v{c}as, J.~Sou\v{c}ek and V.~Sou\v{c}ek then developed an abstract upper-bound and countability theory for the full critical-level set of homogeneous nonlinear eigenproblems \cite{FucikNecasSoucekSoucek1972,FucikNecasSoucekSoucek1972Strengthening, FucikNecasSoucekSoucek1973}, consolidated in their 1973 monograph \cite{FucikNecasSoucekSoucek1973Book}. Together with Ku\v{c}era, they later obtained Hausdorff-measure bounds for full critical-level sets under finite-smoothness and Fredholm hypotheses \cite[Theorems~2.1 and~3.1, pp.~221, 225--226] {FucikKuceraNecasSoucekSoucek1974}. Their two-dimensional quartic example has the three eigenvalues $2$, $1$ and $2/3$, whereas the Ljusternik--Schnirelmann procedure gives only $2$ and $2/3$ \cite[p.~65]{FucikNecasSoucekSoucek1973}. Thus the distinction between a min--max sequence and the full nonlinear spectrum was explicit from the beginning of this theory.

A particularly close finite-dimensional antecedent is the study by J.~Sou\v{c}ek and V.~Sou\v{c}ek of the eigenvalues realised by real-analytic potential maps $A:\mathbb R^n\to\mathbb R^n$ through $A(x)=\lambda x$ on fixed Euclidean spheres. They asked whether the set of eigenvalues realised on a fixed sphere is finite, recalled finiteness in the homogeneous case, showed that this set may contain an interval in the nonhomogeneous case, and described the discrete and continuous parts that may occur in the general analytic setting \cite[pp.~614--619]{SoucekSoucek1974}. When $p$ is not an even integer, the graph energy is not real analytic on the coordinate and edge-form hyperplanes. Thus the graph problem lies outside the real-analytic setting and requires a different argument. Contemporaneously, Dancer analysed the global structure of solution sets for real-analytic nonlinear eigenvalue problems \cite{Dancer1973}.

For the continuum $p$-Laplacian, the distinction became a prominent open question. Lindqvist's Jyv\"askyl\"a lectures were published in 1995 \cite{Lindqvist1995}. In a later expanded version, he asked whether variational constructions exhaust the Dirichlet $p$-Laplacian spectrum, whether all its eigenvalues can be enumerated, and whether the spectrum is discrete even for a ball or a cube \cite[p.~194]{Lindqvist2008}. Binding and Rynne constructed $p$-Laplacian eigenvalues outside the Ljusternik--Schnirelmann sequence \cite{BindingRynne2008}, and Dr\'{a}bek described the general problem of characterising the full spectrum of a homogeneous nonlinear eigenproblem as open since the 1970s \cite[p.~1]{Drabek2012}. The finite-dimensional graph problem admits a different source of control: o-minimality applies at arbitrary real exponents, including those for which the energy is not analytic.

Van den Dries and Miller proved in 1994 that $\mathbb R_{\mathrm{an},\exp}$ is o-minimal \cite{vanDenDriesMiller1994}. Placing the graph family in this structure gives a qualitative route to finiteness and, in each fixed dimension, a non-effective coefficient- and exponent-uniform bound through o-minimal critical-value theory and definable uniformity. For each fixed exponent, the same qualitative conclusion follows within Miller's polynomially bounded expansion by restricted analytic functions and real power functions \cite{Miller1994}; $\mathbb R_{\mathrm{an},\exp}$ is the natural setting here because the exponent itself varies as a parameter. Theorem~A makes this mechanism quantitative: it supplies explicit bounds in terms of the numbers of vertices and positive-weight edges, together with uniform topological estimates for singular eigenvector sets.

The graph $p$-energy also appears in the classical literature on network potential theory. Duffin developed quadratic network extremal length \cite{Duffin1962}; Yamasaki studied extremum problems on infinite networks \cite{Yamasaki1975}; and Nakamura and Yamasaki introduced a generalised extremal length based on the nonlinear energy $p^{-1}\sum_{\{v,w\}\in E}|f(v)-f(w)|^p$ \cite{NakamuraYamasaki1976}. Holopainen and Soardi subsequently developed $p$-harmonic potential theory on graphs \cite{HolopainenSoardi1997}. These works concern extremal length, capacity, and $p$-harmonicity, rather than the cardinality of the full spectrum or the topology of its normalised eigenvector sets.

Within graph spectral theory, sharp results are known in several special regimes. On forests the variational spectrum exhausts the full spectrum \cite[Theorem~3.7]{DeiddaPuttiTudisco2023}. For $p\ne2$, every eigenfunction for a positive eigenvalue of the uniformly weighted $K_n$ takes one positive value, one negative value, and possibly the value zero. This classification is due to Amghibech \cite{Amghibech2003}; we use the formulation in \cite[Statement~D.3]{DeiddaTudiscoZhang2025}. Signed weighted graph $p$-Schr\"odinger operators with real potentials are treated in \cite[Section~2]{GeLiuZhang2023}, while Jost and Zhang develop a broader homogeneous-function framework, including zero-homogeneous critical-point constructions and vertex--edge correspondences \cite{JostZhang2021}. Tudisco and Zhang establish, for convex homogeneous pairs, a nonlinear spectral duality relating the nonzero spectrum and variational values \cite[Theorems~1 and~4]{TudiscoZhang2026}. In a related direction, Zhang's polar-duality theory preserves sublevel-set homotopy types, Morse critical points, and Rothe critical groups \cite{Zhang2026}. Theorem~A places signed weighted graph $p$-Schr\"odinger operators with arbitrary real potentials within a uniform full-spectrum theory, while Theorems~B and~C show that exponential extremal growth already occurs for complete graphs with positive integer edge weights.

The joint work of Deidda, Tudisco and Zhang provides a complete, self-contained account of nonlinear spectral graph theory \cite{DeiddaTudiscoZhang2025}. Deidda's thesis gives a broad treatment of the graph $p$-Laplacian eigenproblem, including its nodal, limiting and computational aspects \cite{DeiddaThesis2023}. Variational graph eigenvalues underpin spectral clustering \cite{BuhlerHein2009}. Tudisco and Hein proved a nodal-domain theorem and higher-order Cheeger inequalities for the graph $p$-Laplacian \cite{TudiscoHein2018}. Zhang introduced homological eigenvalues for graph $p$-Laplacians and used them to analyse non-variational eigenvalues and their dependence on $p$ \cite{Zhang2025}. For $p>2$, Deidda, Segala and Putti identify the index of an eigenpair with the Morse index of the Rayleigh quotient and show that differentiable saddle points of the $k$th spectral energy correspond to index-$k$ eigenpairs \cite{DeiddaSegalaPutti2025}. Theorem~B gives a complete model for all nonconstant eigenlines of the uniformly weighted $K_n$: it determines their Morse indices for $p>2$ and their local critical groups for $1<p<2$, including at values outside the Krasnosel'skii genus min--max sequence. Under the hypotheses of \cite[Theorem~2.7]{BerkolaikoHofmann2025}, graph surgery relates eigenpairs to stationary points of cut-graph eigenvalue branches. Theorem~A controls the entire critical-value set.

\subsection{Organisation}

Sections~\ref{sec:ominimal}--\ref{sec:effective-bounds} prove Theorem~A and its quantitative and topological refinements. Sections~\ref{sec:projective}--\ref{sec:growth} establish the complete-graph cell decomposition and prove Theorems~B and~C. The appendices give definable-family and topological results, endpoint limits, matrix extensions, small-graph results, and the exact algebraic verification of a nine-value triangle. For example, an exact computer-assisted proof, given in Appendix~\ref{app:nine-triangle}, constructs a signed weighted triangle with nonnegative integer potential and exactly nine distinct spectral values at $p=4$; Corollary~\ref{cor:nine-open-interval} shows that these values remain distinct throughout an open interval containing $4$.

\vspace{4mm}

\paragraph{\textbf{Acknowledgments.}} I thank the Isaac Newton Institute for Mathematical Sciences, Cambridge, for support and hospitality during the programme Geometric spectral theory and applications, where this work was undertaken. This work was supported by EPSRC grant EP/Z000580/1.

\section{O-minimal preliminaries}\label{sec:ominimal}

An expansion $\mathcal R$ of the ordered real field is specified by a Boolean algebra of definable subsets of $\R^q$ for each $q\geq1$. These algebras contain the semialgebraic sets and are closed under Cartesian products and coordinate projections. The expansion is \emph{o-minimal} if its definable subsets of $\R$ are precisely the finite unions of points and intervals. A map is definable if its graph is definable; throughout the paper, real parameters are allowed. Cell decomposition gives the following standard consequences \cite{Coste1999,vanDenDries1998,vanDenDriesMiller1996}.

\begin{proposition}[Basic finiteness properties]\label{prop:ominimal-facts}
Let $\mathcal{R}$ be an o-minimal expansion of the real field.
\begin{enumerate}[label=\textup{(\roman*)}]
\item Every definable set has finitely many definably connected components. Each such component is definably path-connected, and a definable path can be partitioned into finitely many $C^1$ arcs.
\item If $X\subset\R^q\times\R$ is definable and every fibre
$
 X_\theta=\{t\in\R:(\theta,t)\in X\}
$
is finite, then there exists $N<\infty$ such that $\#X_\theta\leq N$ for every $\theta\in\R^q$.
\item If $I\subset\R$ is an interval and $X\subset I\times\R$ is definable with finite vertical fibres, then there is a finite set $P\subset I$ such that, for every component $J$ of $I\setminus P$, the set $X\cap(J\times\R)$ is a disjoint union of graphs
$$
 \bigsqcup_{k=1}^r\{(t,f_k(t)):t\in J\},
$$
where $r\geq0$ and, when $r>0$, the continuous definable functions satisfy $f_1<\cdots<f_r$ on $J$.
\item Let $f:X\to Y$ be a continuous definable map and let $A_1,\ldots,A_s\subset X$ be definable. There is a finite definable partition $Y=\bigsqcup_\alpha Y_\alpha$ such that, for each $\alpha$ and $y_\alpha\in Y_\alpha$, there is a definable homeomorphism
$$
 f^{-1}(Y_\alpha)\rightarrow
 Y_\alpha\times f^{-1}(y_\alpha)
$$
over $Y_\alpha$ which carries each $A_j\cap f^{-1}(Y_\alpha)$ onto $Y_\alpha\times(A_j\cap f^{-1}(y_\alpha))$.
\end{enumerate}
\end{proposition}

Part~(ii) is usually called \emph{uniform finiteness}. Part~(iii) follows by applying cell decomposition to $X$: finite fibres exclude two-dimensional cells, and a finite refinement by the endpoints of the projected cells makes each remaining one-dimensional cell the graph of a function over an entire component $J$. Part~(iv) is the definable version of Hardt's triviality theorem, with compatibility for prescribed definable subsets; see \cite[Theorem~5.22]{Coste1999}. Apply the theorem to the coordinate projection of $\operatorname{graph}(f)$ and identify that graph with $X$. This gives the stated version for a continuous definable map, including compatibility with finitely many prescribed definable subsets.

The structure relevant here is
$$
 \R_{\exp}=(\R;<,+,\cdot,0,1,\exp).
$$
Its o-minimality already follows from the 1994 theorem of van den Dries and Miller, since $\R_{\exp}$ is a reduct of $\mathbb R_{\mathrm{an},\exp}$ \cite{vanDenDriesMiller1994}; Wilkie later proved o-minimality of $\R_{\exp}$ directly \cite{Wilkie1996}. The logarithm is definable as the inverse of the restriction of $\exp$ to $\R$, and hence the two-variable power map
\begin{equation}\label{eq:pow-def}
 \operatorname{pow}(t,p)
 =\begin{cases}
\exp(p\log|t|),&t\neq 0,\\
0,&t=0,
\end{cases}
 \qquad (p>0),
\end{equation}
is definable in $\R_{\exp}$. On the range $p>1$, the map
\begin{equation}\label{eq:phi-def}
 \Phi(t,p)
 =\begin{cases}
\sgn(t)\exp((p-1)\log|t|),&t\neq 0,\\
 0,&t=0,
 \end{cases}
\end{equation}
is also definable and equals $|t|^{p-2}t$.

The following critical-value principle will be used throughout.

\begin{lemma}[Finite critical values]\label{lem:finite-critical-values}
Let $M$ be a definable $C^1$ manifold and let $f:M\to\R$ be a definable $C^1$ function. Then the set of critical values
$
 f\bigl(\Crit(f)\bigr)
$
is finite.
\end{lemma}

\begin{proof}
Derivatives of definable differentiable maps are definable, so the critical set $C=\Crit(f)$ is definable. By Proposition~\ref{prop:ominimal-facts}(i), it has finitely many definably connected components. Fix one such component $C_0$ and two points $x,y\in C_0$. There is a definable path $\gamma:[0,1]\to C_0$ joining them, and a finite partition of $[0,1]$ on whose open subintervals $\gamma$ is $C^1$. On each such subinterval, the chain rule gives
$$
 \frac{d}{dt}(f\circ\gamma)(t)
 =df_{\gamma(t)}\bigl(\gamma'(t)\bigr)=0.
$$
Continuity across the finitely many partition points shows that $f(x)=f(y)$. Thus $f$ is constant on every connected component of $C$, and there are only finitely many such components.
\end{proof}

\begin{remark}\label{rem:def-sard}
Lemma~\ref{lem:finite-critical-values} is the scalar-valued consequence of the definable Sard theorem. Kurdyka's o-minimal gradient inequality and the resulting finite-length property for bounded gradient trajectories are proved in \cite{Kurdyka1998}.
\end{remark}

The next step is to place homogeneous nonlinear eigenproblems in this critical-value framework.

\section{Definable homogeneous Rayleigh quotients}\label{sec:rayleigh}

Let $F,G:\R^n\setminus\{0\}\to\R$ be $C^1$ functions. Assume that, for some $\delta>0$,
\begin{equation}\label{eq:homogeneity}
 F(tx)=t^\delta F(x),\qquad G(tx)=t^\delta G(x)
 \quad\text{for all }t>0,
\end{equation}
and that $G(x)>0$ whenever $x\neq 0$. Define the degree-zero Rayleigh quotient
$$
 Q(x)=\frac{F(x)}{G(x)}.
$$
We call
\begin{equation}\label{eq:abstract-spectrum}
 \Spec(F,G)
 =\{\lambda\in\R:\exists x\neq 0\text{ with }
\nabla F(x)=\lambda\nabla G(x)\}
\end{equation}
the generalised spectrum of the pair $(F,G)$.

\begin{proposition}[Eigenvalues are critical values]\label{prop:eigen-critical}
Under the preceding assumptions,
\begin{equation}\label{eq:spec-critical}
 \Spec(F,G)
 =Q\bigl(\Crit(Q)\bigr)
 =Q\bigl(\Crit(Q|_{\Sph^{n-1}})\bigr).
\end{equation}
\end{proposition}

\begin{proof}
Euler's identity for homogeneous $C^1$ functions gives
\begin{equation}\label{eq:euler}
 \inner{\nabla F(x)}{x}=\delta F(x),\qquad
 \inner{\nabla G(x)}{x}=\delta G(x).
\end{equation}
If $\nabla F(x)=\lambda\nabla G(x)$, then \eqref{eq:euler} and $G(x)>0$ show that
$$
 \lambda=\frac{F(x)}{G(x)}=Q(x).
$$
Moreover,
\begin{equation}\label{eq:grad-quotient}
 \nabla Q(x)
 =\frac{G(x)\nabla F(x)-F(x)\nabla G(x)}{G(x)^2},
\end{equation}
so $\nabla Q(x)=0$. Conversely, \eqref{eq:grad-quotient} shows that every critical point of $Q$ satisfies $\nabla F(x)=Q(x)\nabla G(x)$. This proves the first equality in \eqref{eq:spec-critical}.

Since $Q$ is homogeneous of degree zero, Euler's identity gives $\inner{\nabla Q(x)}{x}=0$, and differentiating $Q(tx)=Q(x)$ with respect to $x$ gives $\nabla Q(tx)=t^{-1}\nabla Q(x)$ for $t>0$. Hence every unconstrained critical ray meets $\Sph^{n-1}$ in a critical point. Conversely, if $x\in\Sph^{n-1}$ is a critical point of the restriction of $Q$ to the sphere, the Lagrange multiplier equation has the form $\nabla Q(x)=\alpha x$. Taking the inner product with $x$ yields $\alpha=0$. Thus the constrained and unconstrained critical points agree on the sphere, proving the second equality.
\end{proof}

\begin{theorem}[Finite spectrum for definable homogeneous pairs]
\label{thm:definable-homogeneous}
Suppose, in addition, that $F$ and $G$ are definable in an o-minimal expansion of the real field. Then $\Spec(F,G)$ is finite.
\end{theorem}

\begin{proof}
The quotient $Q|_{\Sph^{n-1}}$ is a definable $C^1$ function. Its set of critical values is finite by Lemma~\ref{lem:finite-critical-values}, and it equals $\Spec(F,G)$ by Proposition~\ref{prop:eigen-critical}.
\end{proof}

\begin{example}[The definability hypothesis is essential]
\label{ex:nondefinable-homogeneous}
Let $f:S^1\to\mathbb R$ be a smooth function with infinitely many distinct critical values. For example, take disjoint flat bumps whose supports accumulate at one point and whose amplitudes tend to zero faster than every power of their support diameters, and extend their sum by zero at the accumulation point. For any $\delta>0$, define, on $\mathbb R^2\setminus\{0\}$,
$$
 F(x)=\|x\|^\delta f(x/\|x\|),\qquad G(x)=\|x\|^\delta.
$$
These are smooth homogeneous functions of the same degree and $G>0$, but $F/G=f$ on the unit circle. Their generalised spectrum is therefore infinite.
\end{example}

\begin{remark}[Positive-dimensional eigenvector sets]
The conclusion of Theorem~\ref{thm:definable-homogeneous} includes critical sets of positive dimension. Such sets occur for repeated linear eigenvalues and zero-energy graph eigenfunctions. Definability ensures that the critical set has only finitely many connected components, even when some components are positive-dimensional.
\end{remark}

The argument is uniform in definable families, even when the homogeneity degree varies with the parameter.

\begin{theorem}[Uniform finiteness for definable Rayleigh families]
\label{thm:abstract-uniform}
Let $\Theta$ be a definable parameter set and let
$
 \cQ:\Theta\times(\R^n\setminus\{0\})\to\R
$
be definable. Assume that, for every $\theta\in\Theta$, the fibre $Q_\theta=\cQ(\theta,\cdot)$ is $C^1$ and homogeneous of degree zero, and that the map $(\theta,x)\mapsto\nabla_xQ_\theta(x)$ is definable. Set
$
 \Sigma_\theta
 =Q_\theta\bigl(\Crit(Q_\theta|_{\Sph^{n-1}})\bigr).
$
Then there exists $B<\infty$ such that $\#\Sigma_\theta\leq B$ for every $\theta\in\Theta$.
\end{theorem}

\begin{proof}
The total critical-value set
$$
 \Sigma
 =\{(\theta,\lambda):\exists x\in\Sph^{n-1},\
\nabla_xQ_\theta(x)
=\inner{\nabla_xQ_\theta(x)}{x}x,\
Q_\theta(x)=\lambda\}
 \subset\Theta\times\R
$$
is definable; the displayed vector equation says exactly that $x$ is a critical point of the restriction to the sphere. Each fibre $\Sigma_\theta$ is finite by Lemma~\ref{lem:finite-critical-values}. Uniform finiteness, Proposition~\ref{prop:ominimal-facts}(ii), gives a common cardinality bound.
\end{proof}

\begin{corollary}[One-parameter branch decomposition]
\label{cor:abstract-branches}
In the setting of Theorem~\ref{thm:abstract-uniform}, suppose that $\Theta=I\subset\R$ is an interval. There is a finite set $P\subset I$ such that, on every component $J$ of $I\setminus P$, $\Sigma_\theta=\{f_1(\theta),\ldots,f_r(\theta)\}$, where $f_1<\cdots<f_r$ are continuous definable functions on $J$.
\end{corollary}

\begin{proof}
Apply Proposition~\ref{prop:ominimal-facts}(iii) to the definable set $\Sigma\subset I\times\R$.
\end{proof}

The graph $p$-Laplacian enters this framework through its energy and mass.

\section{The graph \texorpdfstring{$p$}{p}-Laplacian}\label{sec:graph}

Fix $n\geq1$. Encode every finite simple graph on $\{1,\ldots,n\}$ by nonnegative symmetric edge coefficients, setting $a_{ij}=0$ on nonedges. Assign signs to the unordered pairs and let
$$
 a_{ij}=a_{ji}\geq0,\quad a_{ii}=0,\qquad
 \varepsilon_{ij}=\varepsilon_{ji}\in\{-1,1\},
 \qquad b_i\in\R,\qquad \nu_i>0.
$$
For $p>1$, define the energy and mass
\begin{align}
 \cE_{p,a,\varepsilon,b}(x)
 &=\sum_{1\leq i<j\leq n}a_{ij}|x_i-\varepsilon_{ij}x_j|^p
+\sum_{i=1}^n b_i|x_i|^p,\label{eq:energy}\\
 \cM_{p,\nu}(x)
 &=\sum_{i=1}^n\nu_i|x_i|^p.\label{eq:mass}
\end{align}
Both functions are $C^1$ and positively homogeneous of degree $p$, and $\cM_{p,\nu}(x)>0$ for $x\neq0$. Their homogeneous Rayleigh quotient is
\begin{equation}\label{eq:graph-rayleigh}
 \cQ_{p,a,\varepsilon,b,\nu}(x)
 =\frac{\cE_{p,a,\varepsilon,b}(x)}{\cM_{p,\nu}(x)}.
\end{equation}
When $b_i\geq0$, a common alternative normalisation is
$$
 \frac{\cE_{p,a,\varepsilon,b}(x)^{1/p}}
{\cM_{p,\nu}(x)^{1/p}}.
$$
For general real $b$, the energy may be negative, while the quotient in \eqref{eq:graph-rayleigh} remains $C^1$ on $\R^n\setminus\{0\}$. The arguments below require only positivity of the denominator.

Let $\Phi_p(t)=|t|^{p-2}t$, with $\Phi_p(0)=0$. Direct differentiation gives
\begin{align}
 \frac{1}{p}\frac{\partial\cE_{p,a,\varepsilon,b}}{\partial x_i}(x)
 &=\sum_{j\neq i}a_{ij}\Phi_p(x_i-\varepsilon_{ij}x_j)
+b_i\Phi_p(x_i),\label{eq:energy-gradient}\\
 \frac{1}{p}\frac{\partial\cM_{p,\nu}}{\partial x_i}(x)
 &=\nu_i\Phi_p(x_i).\label{eq:mass-gradient}
\end{align}
The contribution at the other endpoint has the stated form because $-\varepsilon_{ij}\Phi_p(x_i-\varepsilon_{ij}x_j) =\Phi_p(x_j-\varepsilon_{ij}x_i)$. Thus \eqref{eq:eigen-intro} is precisely
\begin{equation}\label{eq:gradient-eigen}
 \nabla\cE_{p,a,\varepsilon,b}(x)
 =\lambda\nabla\cM_{p,\nu}(x).
\end{equation}

\begin{proposition}[Variational identification]\label{prop:graph-critical}
For every $p>1$,
\begin{equation}\label{eq:graph-spec-critical}
 \sigma_p(a,\varepsilon,b,\nu)
 =\cQ_{p,a,\varepsilon,b,\nu}
\bigl(\Crit(\cQ_{p,a,\varepsilon,b,\nu}|_{\Sph^{n-1}})\bigr).
\end{equation}
In particular, if $(x,\lambda)$ is an eigenpair, then
\begin{equation}\label{eq:eigenvalue-quotient}
 \lambda
 =\frac{\cE_{p,a,\varepsilon,b}(x)}{\cM_{p,\nu}(x)}.
\end{equation}
Set $\beta=\min_i b_i/\nu_i$. Then every eigenvalue satisfies $\lambda\geq\beta$; in particular, it is nonnegative when $b_i\geq0$.
\end{proposition}

\begin{proof}
Proposition~\ref{prop:eigen-critical}, applied to $F=\cE_{p,a,\varepsilon,b}$ and $G=\cM_{p,\nu}$, gives \eqref{eq:graph-spec-critical} and \eqref{eq:eigenvalue-quotient}. The lower bound follows because the edge energy is nonnegative and $\sum_i b_i|x_i|^p\geq\beta\cM_{p,\nu}(x)$.
\end{proof}

\begin{proposition}[Spectral shift]\label{prop:potential-shift}
Set
$$
 \beta=\min_i\frac{b_i}{\nu_i},\qquad
 \widetilde b_i=b_i-\beta\nu_i\geq0.
$$
Then
$$
 \cQ_{p,a,\varepsilon,b,\nu}
 =\beta+\cQ_{p,a,\varepsilon,\widetilde b,\nu},\qquad
 \sigma_p(a,\varepsilon,b,\nu)
 =\beta+\sigma_p(a,\varepsilon,\widetilde b,\nu),
$$
and the normalised eigenvector sets agree after translating the eigenvalue.
\end{proposition}

\begin{proof}
The identity $b_i=\widetilde b_i+\beta\nu_i$ gives both assertions directly.
\end{proof}

\begin{proof}[Proof of Theorem~\ref{thm:intro-main}\ref{part:intro-finite}]
For fixed $p>1$, the maps $t\mapsto|t|^p$ are definable in $\R_{\exp}$ by \eqref{eq:pow-def}. Hence $\cE_{p,a,\varepsilon,b}$, $\cM_{p,\nu}$, and their quotient are definable. The result follows from Theorem~\ref{thm:definable-homogeneous} and Proposition~\ref{prop:graph-critical}.
\end{proof}

\begin{remark}[Standard graph conventions]\label{rem:conventions}
The formulation above includes several common choices.
\begin{itemize}
\item For the unsigned unnormalised graph $p$-Laplacian, take $\varepsilon_{ij}=1$, $b_i=0$, and $\nu_i=1$.
\item A positive vertex measure is incorporated through $\nu_i$. The usual normalised problem is obtained by taking $\nu_i$ to be the weighted degree, provided isolated zero-degree vertices are removed or treated separately.
\item If a subset of vertices is fixed to zero by a Dirichlet condition, then an edge from an interior vertex $i$ to the boundary contributes a term $a_{i\partial}|x_i|^p$; summing these contributions gives $b_i$.
\item The signed weighted $p$-Schr\"odinger convention of \cite{GeLiuZhang2023} is already in the form used here. In the convention of \cite{GeQin2025}, multiplication of the vertexwise eigen-equation by the positive vertex measure gives the present form. In both conventions the real potential is the coefficient $b_i$.
\item Switching by signs $\tau_i\in\{-1,1\}$ replaces $\varepsilon_{ij}$ by $\tau_i\varepsilon_{ij}\tau_j$ and $x_i$ by $\tau_i x_i$. It leaves the energy, mass, and full spectrum unchanged.
\item In the convention of \cite{DeiddaTudiscoZhang2025}, an edge weight $\omega_{ij}\geq 0$ enters through $a_{ij}=\omega_{ij}^p$. For a fixed $p$, this is exactly the preceding framework.
\end{itemize}
\end{remark}

\begin{remark}[Zero eigenvalues and disconnected graphs]
The theory applies to disconnected and noncoercive data. When $b_i\geq0$ and the energy has a nontrivial kernel, its zero eigenvectors form a definable critical set, possibly of positive dimension, contributing the single spectral value $0$. For a general real potential, zero may lie above the bottom of the spectrum; Proposition~\ref{prop:potential-shift} gives the precise reduction.
\end{remark}

\section{Uniform bounds and dependence on the exponent}\label{sec:uniform}

Joint definability of the power function in its base and exponent yields a bound uniform in both the graph data and $p$. Let $D_n$ be the data parameter space
$$
 d=(a,\varepsilon,b,\nu),\qquad a_{ij}=a_{ji}\geq0,
 \quad\varepsilon_{ij}=\varepsilon_{ji}\in\{-1,1\},
 \quad b_i\in\R,\quad\nu_i>0.
$$
Changing $\varepsilon_{ij}$ on a pair with $a_{ij}=0$ leaves the equations unchanged. The sets $D_n$ and $\Theta_n=D_n\times(1,\infty)$ are definable. We write $\theta=(d,p)$ for a point of $\Theta_n$.

For $1\leq i\leq n$, set
$$
 \mathcal{F}_i(\theta,x,\lambda)
 =\sum_{j\neq i}a_{ij}\Phi(x_i-\varepsilon_{ij}x_j,p)
+b_i\Phi(x_i,p)
-\lambda\nu_i\Phi(x_i,p).
$$
Define the total normalised eigenpair set
\begin{equation}\label{eq:total-eigenpairs}
 \mathcal{Z}_n
 =\bigl\{(\theta,x,\lambda):\theta\in\Theta_n,\ x\in\Sph^{n-1},\ \lambda\in\R,\
\mathcal{F}_i(\theta,x,\lambda)=0\ \text{for all }i\bigr\}.
\end{equation}
The function $\Phi$ is the definable map in \eqref{eq:phi-def}; therefore $\mathcal{Z}_n$ is definable in $\R_{\exp}$. Its projection
\begin{equation}\label{eq:total-spectrum}
 \mathcal{S}_n
 =\{(\theta,\lambda):\exists x\in\Sph^{n-1}\text{ with }(\theta,x,\lambda)\in\mathcal{Z}_n\}
\end{equation}
is definable as well.

\begin{proof}[Proof of Theorem~\ref{thm:intro-main}\ref{part:intro-uniform}]
For each $\theta=(d,p)$, the fibre $(\mathcal{S}_n)_\theta$ is exactly $\sigma_p(d)$ and is finite by Theorem~\ref{thm:intro-main}\ref{part:intro-finite}. Uniform finiteness applied to the definable set \eqref{eq:total-spectrum} gives a constant $\mathfrak B_n^{\mathrm{unif}}$ bounding every fibre. The parameter space includes every edge set, signature, real potential, and positive vertex measure, proving the asserted independence.
\end{proof}

The preceding o-minimal argument is qualitative. Section~\ref{sec:effective-bounds} obtains explicit bounds by working on strata determined by the signs and vanishing pattern of the linear forms. Results on spectral branches and definable triviality of the eigenvector sets are given in Appendix~\ref{app:families}; endpoint limits are treated in Appendix~\ref{app:endpoints}.

\section{Quantitative Pfaffian and semialgebraic bounds}
\label{sec:effective-bounds}

We derive the effective estimates in Theorem~\ref{thm:intro-main} and Corollary~\ref{thm:intro-effective-refinements}. The common structure is a quotient of $p$-homogeneous sums built from finitely many linear forms; coordinate forms and signed edge differences recover the graph problem.

\subsection{An explicit estimate}\label{subsec:effective}

A Pfaffian chain of length $r$ and degree $\alpha$ on an open domain $U\subset\mathbb R^d$ is a sequence of real-analytic functions $f_1,\ldots,f_r$ such that every partial derivative of $f_i$ is a polynomial of degree at most $\alpha$ in the base variables and $f_1,\ldots,f_i$. A function has degree at most $\beta$ in the chain if it is represented by a polynomial of total degree at most $\beta$ in those variables. We apply the following explicit consequence of the Pfaffian component theorem of Gabrielov and Vorobjov \cite[Corollary~3.3]{GabrielovVorobjov2004}. If finitely many Pfaffian functions on $U=\mathbb R^d$ share a chain of length $r$ and degree $\alpha$, and each has degree at most $\beta$ in that chain, then their common zero set has at most
\begin{equation}\label{eq:GV-component}
 2^{r(r-1)/2+1}\beta(\alpha+2\beta-1)^{d-1}
 \bigl((2d-1)(\alpha+\beta)-2d+2\bigr)^r
\end{equation}
connected components. Formula~\eqref{eq:GV-component} applies on all of $\mathbb R^d$, including unbounded domains and singular or positive-dimensional zero sets. The numerical parameters $p,c_r,\gamma_r$ enter as real coefficients and therefore do not affect the Pfaffian format $(d,r,\alpha,\beta)$.

\begin{theorem}[Effective bounds for spectra arising from finite families of linear forms]
\label{thm:linear-form-effective}
Let $L_1,\ldots,L_K\in(\R^n)^*$, let $c\in\R^K$ and $\gamma\in[0,\infty)^K$, and suppose that
$$
 \bigcap_{\{r:\gamma_r>0\}}\ker L_r=\{0\}.
$$
For each $p>1$, set
$$
 G_{\gamma,p}(x)=\sum_{r=1}^K \gamma_r|L_r(x)|^p,
 \qquad
 F_{c,p}(x)=\sum_{r=1}^K c_r|L_r(x)|^p,
$$
and let $\Sigma_{L,c,\gamma,p}$ be the set of real $\lambda$ for which
\begin{equation}\label{eq:linear-form-eigen}
 \sum_{r=1}^K c_r\Phi_p(L_r(x))L_r
 =\lambda\sum_{r=1}^K \gamma_r\Phi_p(L_r(x))L_r
\end{equation}
for some $x\ne0$. Then, for every $p>1$,
$$
 \#\Sigma_{L,c,\gamma,p}\leq\mathfrak B(K).
$$
If $E_L(p,\lambda)\subset\Sph^{n-1}$ denotes the normalised solution set of \eqref{eq:linear-form-eigen}, then, more strongly,
\begin{equation}\label{eq:linear-form-component-bound}
 \sum_{\lambda\in\Sigma_{L,c,\gamma,p}}
 b_0\bigl(E_L(p,\lambda)/\{\pm1\}\bigr)
 \leq
 \sum_{\lambda\in\Sigma_{L,c,\gamma,p}}
 b_0\bigl(E_L(p,\lambda)\bigr)
 \leq\mathfrak B(K).
\end{equation}
The coefficients $c_r$ may be arbitrary real numbers; repeated and zero forms may be retained in the indexed family.
\end{theorem}

\begin{proof}
Let $\mathcal L:\R^n\to\R^K$ be $\mathcal Lx=(L_1(x),\ldots,L_K(x))$. The kernel hypothesis makes $G_{\gamma,p}$ positive away from the origin and implies that $\mathcal L$ is injective. Fix a linear left inverse $T:\R^K\to\R^n$, so $T\mathcal L=I$.

Stratify the unit sphere by the signs of the $K$ forms. Fix a sign pattern $\sigma\in\{-1,0,1\}^K$ with $q$ nonzero entries, introduce one variable $y_r$ for each nonzero entry, and set
$$
 z_r(y)=
 \begin{cases}
 0,&\sigma_r=0,\\
\sigma_r e^{y_r},&\sigma_r\ne0.
 \end{cases}
$$
The equations
$
 \mathcal LTz(y)=z(y)$ and $\|Tz(y)\|_2^2=1$ state respectively that $z(y)\in\operatorname{im}\mathcal L$ and that $Tz(y)$ is normalised. They reconstruct $x=Tz(y)$ uniquely.

On $\R^{q+1}$, with variables $(y,\lambda)$, use the Pfaffian chain
$$
 U_r=e^{y_r},\qquad V_r=e^{(p-1)y_r}
 \qquad(\sigma_r\ne0).
$$
It has length $2q$ and degree one. The consistency equations have degree one in the chain, the normalisation equation has degree two, and \eqref{eq:linear-form-eigen} has degree at most two because its only products are $\lambda V_r$. The map $(x,\lambda)\mapsto(y,\lambda)$, where $y_r=\log|L_r(x)|$ for the nonzero forms, is a homeomorphism from the normalised eigenpair stratum onto this common zero set. Applying \eqref{eq:GV-component} with
$
 d_{\rm Pf}=q+1, r=2q,\alpha=1,\beta=2
$
therefore bounds the number of its connected components by
$
 2^{2q^2+q+2}(4q+3)^{2q}.
$

The spectrum is finite by Theorem~\ref{thm:definable-homogeneous}. Hence the projection of a connected eigenpair component to the $\lambda$-axis is a point. There are at most $\binom Kq2^q$ sign patterns with $q$ nonzero forms, and $q=0$ is excluded by injectivity and normalisation. Summing over $1\leq q\leq K$ gives $\mathfrak B(K)$. The quotient map by the antipodal action is continuous and surjective, so it cannot increase the number of connected components. This proves \eqref{eq:linear-form-component-bound}.
\end{proof}

\begin{proof}[Proof of Theorem~\ref{thm:intro-main}\ref{part:intro-effective}]
Consider the $N_n$ linear forms
\begin{equation}\label{eq:linear-form-family}
 x_i\quad(1\leq i\leq n),\qquad
 x_i-\varepsilon_{ij}x_j\quad(1\leq i<j\leq n).
\end{equation}
Assign coefficient $b_i$ to the form $x_i$ and coefficient $a_{ij}$ to the form $x_i-\varepsilon_{ij}x_j$ in the numerator. In the denominator assign coefficient $\nu_i$ to $x_i$ and zero to every edge form. The denominator is positive away from the origin, so Theorem~\ref{thm:linear-form-effective} gives the first inequality in \eqref{eq:explicit-Bn}, and its component estimate gives Corollary~\ref{cor:effective-eigenset-components}. Bounding each summand corresponding to support size $q$ by its value at $q=N_n$ and using $\sum_q\binom{N_n}{q}2^q\leq3^{N_n}$ gives the second inequality.
\end{proof}

The same stratification controls eigenvector components, and this control descends to projective eigenvectors. The antipodal quotient has a concrete semialgebraic model. Let
$$
 \mathcal P_n=\{P\in\operatorname{Sym}_n(\R):P^2=P,\ \operatorname{tr}P=1\}.
$$
The map $\pi_{\rm proj}:\Sph^{n-1}\to\mathcal P_n$, $\pi_{\rm proj}(x)=xx^{\mathsf T}$, has fibres exactly $\{x,-x\}$. The eigen-equation is odd in $x$, so every eigenvector set is antipodally invariant. We set $\overline E(d,p,\lambda)=\pi_{\rm proj}(E(d,p,\lambda))$. Since $E(d,p,\lambda)$ is closed in the sphere, it is compact. The induced continuous bijection $E(d,p,\lambda)/\{\pm1\}\to\overline E(d,p,\lambda)$ is therefore a homoemorphism.

\begin{corollary}[An effective bound on the total number of components]
\label{cor:effective-eigenset-components}
For every admissible datum $d$ on $n$ vertices and every $p>1$,
\begin{equation}\label{eq:effective-eigenset-components}
 \#\sigma_p(d)
 \leq \sum_{\lambda\in\sigma_p(d)}
b_0\bigl(\overline E(d,p,\lambda)\bigr)
 \leq \sum_{\lambda\in\sigma_p(d)}
 b_0\bigl(E(d,p,\lambda)\bigr)
 \leq \widehat{\mathfrak B}_n.
\end{equation}
Here $b_0$ is the number of connected components. Thus the same explicit constant that bounds the number of spectral values also bounds the total number of connected components of all normalised eigenvector sets, including singular and positive-dimensional ones.
\end{corollary}

\begin{proof}
Apply Theorem~\ref{thm:linear-form-effective} to the coordinate and signed edge forms in \eqref{eq:linear-form-family}, with the numerator and denominator coefficients used in the proof of Theorem~\ref{thm:intro-main}\ref{part:intro-effective}. This gives the two component bounds in \eqref{eq:effective-eigenset-components}. Every eigenvector set is non-empty, and the antipodal quotient of a non-empty eigenvector set has at least one connected component, which gives the spectral-cardinality inequality.
\end{proof}

\begin{remark}
If the graph has $m$ positive-weight edges, the same proof can omit the zero-weight edge forms and replace $N_n$ by $n+m$. The complete-graph value is used in \eqref{eq:explicit-Bn} to obtain a bound independent of the graph.
\end{remark}

\begin{corollary}[Bounds in terms of the support graph]
\label{cor:sparse-real-exponent}
For every real $p>1$ and every signed graph datum with $n$ vertices and $m$ positive-weight edges,
$$
 \#\sigma_p(d)
 \leq \sum_{\lambda\in\sigma_p(d)}b_0(E(d,p,\lambda))
 \leq \mathfrak B(n+m).
$$
In particular,
$$
 \log\#\sigma_p(d)=O((n+m)^2)
$$
uniformly in $p$ and all coefficients. If the graph formed by the positive-weight edges has maximum degree at most $\Delta$, then the logarithm is $O_\Delta(n^2)$.
\end{corollary}

\begin{proof}
Use only the $n$ coordinate forms and the $m$ positive edge forms in Theorem~\ref{thm:linear-form-effective}. The definition of $\mathfrak B$ and the inequality $m\leq\Delta n/2$ give the stated asymptotics.
\end{proof}

The preceding argument does not require coefficientwise nonnegativity of the denominator once positivity of the resulting homogeneous form is assumed.

\begin{proposition}[Signed denominator coefficients]
\label{prop:signed-denominator}
In Theorem~\ref{thm:linear-form-effective}, for a fixed $p>1$, the conditions $\gamma_r\geq0$ and $\bigcap_{\{r:\gamma_r>0\}}\ker L_r=\{0\}$ may be replaced by the intrinsic condition
\begin{equation}\label{eq:denominator-positive}
 G_{\gamma,p}(x)=\sum_{r=1}^K\gamma_r|L_r(x)|^p>0
 \qquad(x\ne0),
\end{equation}
with arbitrary real $\gamma_r$. The same effective bound holds. For a varying-exponent family, it is enough to impose \eqref{eq:denominator-positive} at every exponent in the parameter interval.
\end{proposition}

\begin{proof}
Condition \eqref{eq:denominator-positive} implies $\bigcap_r\ker L_r=\{0\}$, so the joint linear-form map is injective. This is the only consequence of coefficientwise nonnegativity used in the logarithmic-coordinate proof. The Pfaffian equations in that proof allow arbitrary real coefficients. Positivity makes $G_{\gamma,p}$ a nonvanishing denominator on the unit sphere; equivalently, Euler's identity gives
$
 \langle\nabla G_{\gamma,p}(x),x\rangle=pG_{\gamma,p}(x)>0$ for $x\ne0$.
\end{proof}

The component estimate can be strengthened to a uniform bound for the total rational homology of the eigenvector sets. For a locally compact space $X$, let
$$
 b^{\rm BM}(X;\mathbb Q)
 =\sum_{j\geq0}\dim_{\mathbb Q}H_j^{\rm BM}(X;\mathbb Q),
 \qquad
 b(X;\mathbb Q)=\sum_{j\geq0}b_j(X;\mathbb Q).
$$

\begin{theorem}[Total Betti-number bound at real exponents]
\label{thm:real-exponent-betti}
There is an absolute constant $C>0$ such that, for every real $p>1$, every homogeneous pair formed from $K$ linear forms as in Theorem~\ref{thm:linear-form-effective} whose denominator satisfies \eqref{eq:denominator-positive} obeys
\begin{equation}\label{eq:real-exponent-betti}
 \sum_{\lambda\in\Sigma_{L,c,\gamma,p}}
 b(E_L(p,\lambda);\mathbb Q)
 \leq\exp(CK^2).
\end{equation}
The same bound holds for the sum of the rational Betti numbers of the projective eigenvector sets. It is uniform in $p$, in the linear forms, and in all coefficients.
\end{theorem}

\begin{proof}
Let
$
 \mathcal Lx=(L_1(x),\ldots,L_K(x)).
$
Condition \eqref{eq:denominator-positive} implies that $\mathcal L$ is injective, so $n\leq K$; fix a linear left inverse $T$. Let $\mathcal Z$ be the normalised eigenpair set in $\mathbb S^{n-1}\times\mathbb R$. Pairing the eigen-equation with $x$ gives
$
 \lambda=F_{c,p}(x)/G_{\gamma,p}(x).
$
The denominator is positive on the sphere, so this quotient is bounded. Thus $\mathcal Z$ is a closed subset of a compact cylinder and is compact.

Fix a sign pattern $\sigma\in\{-1,0,1\}^K$, and let $S=\{r:\sigma_r\ne0\}$, with $q=|S|$. The case $q=0$ is empty by the normalisation equation. For $r\in S$ set
$
 z_r(y)=\sigma_re^{y_r},
 w_r(y)=\sigma_re^{(p-1)y_r},
$
and set $z_r=w_r=0$ otherwise. Let $Y_\sigma\subset\mathbb R^{q+1}$ be the common zero set
$$
 \mathcal LTz-z=0,\qquad
 \|Tz\|_2^2=1,\qquad
 \mathcal L^{\mathsf T}\bigl((c-\lambda\gamma)\odot w\bigr)=0.
$$
It is closed. The map
$
 (y,\lambda)\mapsto(Tz(y),\lambda)
$
is a homeomorphism from $Y_\sigma$ onto $\mathcal Z_\sigma$, the eigenpair stratum with sign pattern $\sigma$, with inverse
$
 (x,\lambda)\mapsto
 \bigl((\log|L_r(x)|)_{r\in S},\lambda\bigr).
$
Indeed, $\mathcal LTz=z$ gives $L_r(Tz)=z_r$, which has precisely the prescribed signs and support.

On all of $\mathbb R^{q+1}$ use the Pfaffian chain
$$
 U_r=e^{y_r},\qquad V_r=e^{(p-1)y_r}
 \qquad(r\in S).
$$
It has length $2q$ and degree one. The three displayed families of scalar equations have degrees at most $1,2,2$ in this chain. If $h_{\sigma,1},\ldots,h_{\sigma,s_\sigma}$ are all their scalar members, set
$$
 \mathcal F_\sigma=\sum_{j=1}^{s_\sigma}h_{\sigma,j}^2.
$$
Then $Y_\sigma=\{\mathcal F_\sigma=0\}$ and
$
 \mathbb R^{q+1}\setminus Y_\sigma=\{\mathcal F_\sigma>0\}.
$
This single strict inequality has Pfaffian format
$$
 N=q+1,\qquad \ell=2q,\qquad
 \alpha=1,\qquad\beta=4,\qquad s=1.
$$
Euclidean space is a simple Pfaffian domain. Therefore the general semi-Pfaffian Betti estimate \cite[Theorem~3.4 and Remark~2.12]{GabrielovVorobjov2004}, applied with $N=q+1$, chain length $2q$, $\alpha=1$, $\beta=4$, and one strict inequality, gives an absolute constant $C_1>0$ such that
\begin{equation}\label{eq:log-complement-betti}
 b(\mathbb R^{q+1}\setminus Y_\sigma;\mathbb Q)
 \leq\exp(C_1q^2).
\end{equation}
Substituting these format parameters into the quoted estimate, the exponential factor in the standard format bound is $2^{O(q^2)}$, while its remaining factors have the form $4^{O(q)}(5q)^{O(q)}$; their product is $\exp(O(q^2))$. The quantities $p-1$, the entries of $T$, and all coefficients occur only as real coefficients and do not alter this format.

Since $Y_\sigma$ is closed and definable, it is locally compact. Under stereographic compactification,
$
 A_\sigma=Y_\sigma\cup\{\infty\}\subset\Sph^{q+1}
$
is a compact definable set, and hence is triangulable. It is the one-point compactification of $Y_\sigma$, so
$
 H_i^{\rm BM}(Y_\sigma;\mathbb Q)
 \cong\widetilde H_i(A_\sigma;\mathbb Q).
$
Its singular and \v Cech cohomology agree. The universal coefficient theorem over $\mathbb Q$ and Alexander duality \cite[Chapter~VI, Corollary~8.7]{Bredon1993} give, degree by degree,
$$
 \dim_{\mathbb Q}\widetilde H_i(A_\sigma;\mathbb Q)
 =\dim_{\mathbb Q}\widetilde H^i(A_\sigma;\mathbb Q)
 =\dim_{\mathbb Q}\widetilde H_{q-i}
(\Sph^{q+1}\setminus A_\sigma;\mathbb Q).
$$
Here
$
 \Sph^{q+1}\setminus A_\sigma
 =\mathbb R^{q+1}\setminus Y_\sigma.
$
Consequently
\begin{equation}\label{eq:pattern-BM-bound}
 b^{\rm BM}(\mathcal Z_\sigma;\mathbb Q)
 =b^{\rm BM}(Y_\sigma;\mathbb Q)
 \leq1+b(\mathbb R^{q+1}\setminus Y_\sigma;\mathbb Q)
 \leq\exp(C_1q^2).
\end{equation}
The stratumwise estimates must now be assembled across sign patterns whose closures meet. A filtration by the number of active forms provides the required control.

For $0\leq q\leq K$, let
$$
 \mathcal Z_{\leq q}
 =\{(x,\lambda)\in\mathcal Z:\#\{r:L_r(x)\ne0\}\leq q\}.
$$
These are closed subsets of the compact set $\mathcal Z$, and $\mathcal Z_{\leq0}=\varnothing$. The layer $\mathcal Z_{\leq q}\setminus\mathcal Z_{\leq q-1}$ is the finite disjoint union of the sign-pattern strata with support size $q$. Each stratum is clopen in the layer: near one of its points the $q$ active forms retain their signs, while within the support-$q$ layer no inactive form can become nonzero. Thus Borel--Moore homology is the direct sum over these strata. The Borel--Moore localization sequence for the closed filtration now gives
$$
 b(\mathcal Z;\mathbb Q)
 =b^{\rm BM}(\mathcal Z;\mathbb Q)
 \leq
 \sum_{\sigma\in\{-1,0,1\}^K}
 b^{\rm BM}(\mathcal Z_\sigma;\mathbb Q)
 \leq3^K\exp(C_1K^2)
 \leq\exp(CK^2).
$$

The spectrum is finite, so $\mathcal Z$ is the finite disjoint union of the fibres $E_L(p,\lambda)\times\{\lambda\}$. Its total Betti number is therefore the sum in \eqref{eq:real-exponent-betti}. The antipodal quotient of each normalised eigenvector set is a two-sheeted covering. For a finite covering, the transfer map over $\mathbb Q$ shows that pullback in cohomology is injective, and hence
$$
 b(E_L(p,\lambda)/\{\pm1\};\mathbb Q)
 \leq b(E_L(p,\lambda);\mathbb Q).
$$
Summing over the spectrum proves the projective assertion.
\end{proof}

At rational exponents the absolute powers admit an algebraic lift, yielding a sharper degree-sensitive estimate. Now let $p=A/B$ in lowest terms, where $A>B\geq1$, and set
$
 D=2\max\{B,A-B,1\}.
$

\begin{theorem}[Semialgebraic bounds at rational exponents]
\label{thm:rational-bound}
There is an absolute constant $C>0$ with the following property. At $p=A/B$, let a homogeneous pair formed from $K$ linear forms as in Theorem~\ref{thm:linear-form-effective} have denominator satisfying \eqref{eq:denominator-positive}, and set $F=F_{c,p}$ and $G=G_{\gamma,p}$.
Then its normalised eigenpair set obeys
\begin{equation}\label{eq:rational-betti}
 b\!\left(
\{(x,\lambda)\in\mathbb S^{n-1}\times\mathbb R:\nabla F(x)=\lambda\nabla G(x)\};\mathbb Q
\right)
\leq(CK^2D)^{3K+1}.
\end{equation}
For $\lambda\in\operatorname{Spec}(F,G)$, set
$$
 E(\lambda)=\{x\in\mathbb S^{n-1}:\nabla F(x)=\lambda\nabla G(x)\}.
$$
Consequently,
\begin{equation}\label{eq:rational-spectrum}
 \#\operatorname{Spec}(F,G)
 \leq\sum_{\lambda\in\operatorname{Spec}(F,G)}b(E(\lambda);\mathbb Q)
 \leq(CK^2D)^{3K+1}.
\end{equation}
The same estimate holds for the total rational Betti number of the projective eigenvector sets, after enlarging $C$ if necessary. Thus
$$
 \log\#\operatorname{Spec}(F,G)=O(K\log(KD)).
$$
\end{theorem}

\begin{proof}
Denominator positivity makes the joint linear map
$
 x\mapsto(L_1(x),\ldots,L_K(x))
$
injective, so $n\leq K$. Introduce $u_r\geq0$ and $v_r\in\mathbb R$ and impose
\begin{equation}\label{eq:rational-lift}
 L_r(x)^2=u_r^{2B},\qquad
 v_r^2=u_r^{2(A-B)},\qquad
 L_r(x)v_r\geq0.
\end{equation}
These conditions uniquely force
$$
 u_r=|L_r(x)|^{1/B},\qquad
 v_r=\Phi_p(L_r(x));
$$
if $L_r(x)=0$, the first two equations force $u_r=v_r=0$. Add $\|x\|_2^2=1$ and the $n$ coordinate equations in
\begin{equation}\label{eq:rational-eigen-lift}
 \sum_{r=1}^K(c_r-\lambda\gamma_r)v_rL_r=0
 \quad\hbox{in }(\mathbb R^n)^*.
\end{equation}
The lift uses
$$
 N=n+2K+1\leq3K+1
$$
variables and at most
$$
 s=4K+n+1\leq5K+1
$$
distinct polynomials, all of degree at most $D$. Here each $L_r$ in \eqref{eq:rational-eigen-lift} is a fixed covector, so the mixed coordinate terms $\lambda v_r(L_r)_i$ have degree two. Projection to $(x,\lambda)$ is a semialgebraic homeomorphism, with the displayed auxiliary variables as its continuous inverse.

The bound of Gabrielov and Vorobjov for the Betti numbers of semialgebraic sets defined by quantifier-free formulae \cite[Theorem~1]{GabrielovVorobjov2005} applies to polynomials of degree strictly less than its parameter $d_{\rm alg}$. Apply it with $d_{\rm alg}=D+1$. Since $D\geq2$, its estimate $b(X)\leq(C_0s^2d_{\rm alg})^N$ gives \eqref{eq:rational-betti} after changing the absolute constant.

The normalised eigenpair set is the finite disjoint union of $E(\lambda)\times\{\lambda\}$ over the distinct spectral values $\lambda$, yielding \eqref{eq:rational-spectrum}. The transfer argument for the antipodal two-sheeted cover gives the projective estimate.
\end{proof}

\begin{corollary}[Graph bounds at rational exponents]
\label{cor:rational-graph}
For a graph with $n$ vertices and $m$ positive-weight edges, take $K=n+m$ in Theorem~\ref{thm:rational-bound}. Hence
$$
 \log\#\sigma_{A/B}(d)
 =O\bigl((n+m)\log((n+m)D)\bigr).
$$
For dense graphs and fixed rational $p$ this is $O(n^2\log n)$; for bounded-degree graphs it is $O_p(n\log n)$.
\end{corollary}

\begin{proof}[Proof of Corollary~\ref{thm:intro-effective-refinements}]
Combine Corollary~\ref{cor:sparse-real-exponent}, Theorem~\ref{thm:real-exponent-betti}, and Corollary~\ref{cor:rational-graph}. The projective estimates follow from the transfer argument for the antipodal double cover.
\end{proof}

\begin{remark}[Comparison with the even-integer tensor bound]
\label{rem:even-tensor-bound}
Suppose that $p\geq2$ is even. The tensor-pencil realisation \cite[Proposition~3.1]{GeQin2025} identifies \eqref{eq:eigen-intro} with the generalised eigenproblem for a pair of symmetric tensors of order $p$ and dimension $n$. The denominator is the diagonal mass tensor
$
 (B^{(p)}x^{p-1})_i=\nu_i x_i^{p-1}.
$
It is nonsingular over $\mathbb C$: the equation $B^{(p)}x^{p-1}=0$ has only the zero solution because every $\nu_i$ is positive. Hence the tensor pair is regular and has no eigenvalue at infinity. The degree formula for regular tensor pairs \cite[Theorem~2.1]{DingWei2015} therefore gives exactly $n(p-1)^{n-1}$ complex generalised eigenvalues, counted with algebraic multiplicity. Every distinct real graph eigenvalue is among them, and hence
\begin{equation}\label{eq:combined-even-bound}
 \#\sigma_p(a,\varepsilon,b,\nu)
 \leq\min\bigl\{\widehat{\mathfrak B}_n,\,n(p-1)^{n-1}\bigr\}
 \qquad(p\geq2\text{ even}).
\end{equation}
For fixed even $p$, the tensor estimate has logarithm $O(n)$, compared with $O(n^4)$ for $\widehat{\mathfrak B}_n$. For fixed $n$, it grows as $p^{n-1}$, whereas $\widehat{\mathfrak B}_n$ is uniform over all real $p>1$. The Pfaffian estimate also controls the total number of connected components of the normalised and projective eigenvector sets. The tensor estimate reflects algebraic degree at a fixed even exponent, whereas the Pfaffian estimate is uniform in $p>1$ and also controls the topology of the real eigenvector sets.
\end{remark}

This completes the upper-bound part of the argument. We now turn to the complete-graph geometry that produces the exponential lower bounds.

\section{Complete graphs in projective coordinates}
\label{sec:projective}

Throughout Sections~\ref{sec:projective}--\ref{sec:morse-chambers} we assume $n\geq2$ and work, unless explicitly stated otherwise, with the unsigned complete graph having zero potential, unit vertex measure, and strictly positive edge weights:
$$
 \varepsilon_{ij}=1,\qquad b_i=0,\qquad \nu_i=1,\qquad a_{ij}>0.
$$

We reduce its Rayleigh quotient by translations and introduce the two projective maps used below. We distinguish three notions. A \emph{spectral value} is a number $\lambda$ in the full spectrum. A \emph{projective eigenline} is a point $[x]\in\mathbb{R}P^{n-1}$ for which a nonzero representative satisfies the eigen-equation for some $\lambda$; after normalisation it is equivalently an antipodal pair. A \emph{reduced critical point} is a critical point of the reduced Rayleigh quotient on
$$
 \mathbb P(\mathbb R^n/\mathbb R\one)\cong\mathbb{R}P^{n-2}.
$$
The constant eigenline is the unique eigenline with eigenvalue zero and is absent from the reduced projective space. For $p>2$, Proposition~\ref{prop:index-shift} below shows that, at a regular nonconstant eigenline, the negative index of the full projective Hessian is one greater than the reduced Morse index. A property of a definable parameter manifold is called \emph{generic} if it holds outside a relatively closed definable subset of codimension at least one. Its locus of validity contains an open dense definable set whose complement has measure zero with respect to every smooth volume form on the parameter manifold.

Throughout the section, set
$$
 [n]=\{1,\ldots,n\},\qquad V_n=\mathbb R^n/\mathbb R\one,\qquad m=\binom n2,
$$
and write $\bar x$ for the translation class of $x\in\mathbb R^n$.
For $S\subseteq[n]$, let $\one_S$ denote its indicator vector.

\subsection{Optimal translation and \texorpdfstring{normalised $p$-th powers of pairwise distances}{normalised p-th powers of pairwise distances}}

Identify $V_n$ with the Euclidean mean-zero hyperplane $H=\one^\perp$. For a nonzero class $\bar u\in V_n$ and any representative $u\in\mathbb R^n$, let $c_p(u)$ minimise $c\mapsto\sum_i|u_i+c|^p$, and define
$$
 \kappa_p(\bar u)=u+c_p(u)\one
$$
to be the $p$-centred representative of $\bar u$. Replacing $u$ by $u+t\one$ replaces $c_p(u)$ by $c_p(u)-t$, so $\kappa_p$ is well defined. Set $c_p(0)=0$.

\begin{lemma}[The minimising translation]\label{lem:p-centre}
Let $p>1$ and $0\ne u\in H$. The function
$$
 c\mapsto\sum_{i=1}^n|u_i+c|^p
$$
has a unique minimiser $c_p(u)$, characterised by
\begin{equation}\label{eq:p-centre}
 \sum_{i=1}^n\Phi_p(u_i+c_p(u))=0.
\end{equation}
The intrinsic quantity
\begin{equation}\label{eq:mp-def}
 \mathfrak m_p(\bar u)
 =\sum_i|\kappa_p(\bar u)_i|^p
 =\min_{c\in\mathbb R}\sum_{i=1}^n|u_i+c|^p
\end{equation}
is positive and homogeneous of degree $p$ on $V_n\setminus\{0\}$. It is $C^1$ for every $p>1$. If $p>2$, then
$$
 c_p:H\setminus\{0\}\to\mathbb R,\qquad
 \kappa_p:V_n\setminus\{0\}\to\mathbb R^n
$$
are $C^1$, while $\mathfrak m_p$ is $C^2$. The maps $(p,u)\mapsto c_p(u)$, $(p,\bar u)\mapsto\kappa_p(\bar u)$, and $(p,\bar u)\mapsto\mathfrak m_p(\bar u)$ are definable. For $p>2$, $Dc_p(u)$ and $D^2\mathfrak m_p(\bar u)$ depend jointly continuously on $(p,u)$ and $(p,\bar u)$, respectively.
\end{lemma}

\begin{proof}
Strict convexity and coercivity give a unique minimiser, and differentiation in $c$ gives \eqref{eq:p-centre}. The left side is strictly increasing in $c$. For each $(p,u)$, the left side has a zero, and
$$
 -\max_i u_i\leq c_p(u)\leq -\min_i u_i.
$$
The minimiser is jointly continuous in $(p,u)$ on $(1,\infty)\times(H\setminus\{0\})$. Indeed, if $(p_j,u_j)\to(p,u)$, the displayed bound places the corresponding minimisers in one compact interval; every subsequential limit minimises the limiting strictly convex function, and uniqueness identifies it with $c_p(u)$. Definability follows from definable choice and uniqueness.

For fixed $p>1$, set $\kappa=\kappa_p(\bar u)$. Differentiation of the minimum gives
\begin{equation}\label{eq:mp-envelope}
 D\mathfrak m_p(\bar u)[\bar h]
 =p\sum_i\Phi_p(\kappa_i)h_i.
\end{equation}
Evaluating the difference quotient at $c_p(u)$ gives one inequality, while evaluating it at the minimiser $c_p(u+th)$ gives the reverse inequality; continuity of the minimiser shows that the two limits agree. The formula is independent of the representative of $\bar h$ by \eqref{eq:p-centre}. Its right-hand side is linear in $\bar h$ and continuous in $\bar u$; in finite dimension, the directional derivative is therefore the Fr\'echet derivative, and $\mathfrak m_p$ is $C^1$ for every $p>1$. Positivity and homogeneity follow immediately from strict convexity and the change of variables $c\mapsto tc$.

Assume now that $p>2$. The implicit-function theorem applies to
\begin{equation}\label{eq:mp-IFT}
 H_0(p,u,c)=\sum_i\Phi_p(u_i+c),
\qquad
 \partial_cH_0=(p-1)\sum_i|u_i+c|^{p-2}>0.
\end{equation}
Writing $w_i=|\kappa_i|^{p-2}$ and $S=\sum_iw_i$, one obtains
\begin{equation}\label{eq:mp-centre-derivative}
 Dc_p(u)[h]= -\frac{\sum_iw_ih_i}{S}
\end{equation}
and
\begin{equation}\label{eq:mp-hessian}
 D^2\mathfrak m_p(\bar u)[\bar h,\bar k]
 =p(p-1)\left(
\sum_iw_ih_ik_i
-\frac{(\sum_iw_ih_i)(\sum_iw_ik_i)}{S}
 \right).
\end{equation}
The bilinear expression on the right-hand side is unchanged if constants are added to $h$ or $k$, so it is well defined on $V_n$. The implicit-function theorem in the variables $(p,u,c)$ also gives
\begin{align}
 \partial_pc_p(u)
 &=-\frac{\sum_i\Phi_p(\kappa_i)\log|\kappa_i|}
{(p-1)\sum_i|\kappa_i|^{p-2}},\label{eq:p-centre-p-derivative}\\
 \partial_p\mathfrak m_p(\bar u)
 &=\sum_i|\kappa_i|^p\log|\kappa_i|,
 \label{eq:mp-p-derivative}
\end{align}
where expressions of the form $|t|^\alpha\log|t|$ at $t=0$ are assigned their continuous value zero. The functions $|t|^\alpha\log|t|$ and $|t|^\alpha(\log|t|)^2$ have continuous zero extensions for $\alpha>0$. Thus \eqref{eq:mp-centre-derivative}, \eqref{eq:mp-hessian}, and \eqref{eq:p-centre-p-derivative} are jointly continuous for $p>2$. Differentiating the envelope formula \eqref{eq:mp-envelope} in $p$ and in $\bar u$ proves the asserted joint regularity, including the $C^1$ dependence of $D\mathfrak m_p$.
\end{proof}

\subsection{The projectivised \texorpdfstring{$L^p$}{Lp}-duality map and the coordinate-hyperplane cells}
\label{subsec:sign-cells}

Set $p'=p/(p-1)$. The coordinatewise maps $\Phi_p$ and $\Phi_{p'}$ are mutually inverse. Projectivising the coordinatewise $L^p$-duality map after minimisation over translations gives
\begin{equation}\label{eq:Theta}
 \Theta_p\colon\Pj(V_n)\rightarrow\Pj(H),
 \qquad
 \Theta_p([\bar x])=[\Phi_p(\kappa_p(\bar x))].
\end{equation}

\begin{proposition}[The projectivised $L^p$-duality map]
\label{prop:Theta}
For every $p>1$, the map $\Theta_p$ is a definable homeomorphism. Its inverse is
\begin{equation}\label{eq:Theta-inverse}
 [y]\mapsto[\overline{\Phi_{p'}(y)}],
 \qquad y\in H\setminus\{0\}.
\end{equation}
Before projectivisation, the corresponding maps are odd and positively homogeneous of degrees $p-1$ and $p'-1=1/(p-1)$, respectively. Their restrictions to the sign strata are definable diffeomorphisms between the corresponding strata.
\end{proposition}

\begin{proof}
The centring identity \eqref{eq:p-centre} gives $\sum_i\Phi_p(\kappa_p(\bar x)_i)=0$, so the image lies in $H$. Conversely, if $y\in H$, then $x=\Phi_{p'}(y)$ satisfies
$$
 \sum_i\Phi_p(x_i)=\sum_i y_i=0.
$$
Thus $x$ is already $p$-centred, and the coordinatewise inverse identities together with Lemma~\ref{lem:p-centre} show that the displayed maps are continuous definable inverses. Their homogeneity and oddness make them projectively well defined. For the assertion on sign strata, fix disjoint nonempty sets $P,N$ and impose $x_z=0$ for $z\notin P\cup N$. Choosing the orientation in which the coordinates on $P$ are positive and those on $N$ are negative, the source stratum is
$$
 \left\{x:\sum_i\Phi_p(x_i)=0,\quad
 x_i>0\ (i\in P),\quad x_j<0\ (j\in N),\quad
 x_z=0\ (z\notin P\cup N)\right\}\big/\mathbb R_{>0}.
$$
On every nonzero coordinate, $\Phi_p$ and $\Phi_{p'}$ are smooth mutually inverse power maps, and they carry the centring equation exactly to the linear equation $\sum_i y_i=0$ on the target. Their restrictions are therefore diffeomorphisms between the corresponding sign strata for every $p>1$; globally the maps are definable homeomorphisms.
\end{proof}

If $[\bar x]\in\Pj(V_n)$ is represented by a $p$-centred vector $x$, then
$\Theta_p([\bar x])=[\Phi_p(x)]$, so its coordinate signs and zero coordinates
agree with those of $x$, up to simultaneous reversal of all signs. The coordinate
hyperplanes in $\Pj(H)$ therefore determine the sign strata described below, and
their inverse images under $\Theta_p$ determine the corresponding strata of
$\Pj(V_n)$. For disjoint non-empty subsets $P,N\subset[n]$, set $Z=[n]\setminus(P\cup N)$ and
\begin{equation}\label{eq:sign-cell}
 C_{P,N}=
 \left\{[y]\in\Pj(H):
 y_i>0\ (i\in P),\quad y_j<0\ (j\in N),\quad
 y_z=0\ (z\in Z)\right\}.
\end{equation}
Since $C_{P,N}=C_{N,P}$ in projective space, $[P\mid N]$ denotes the unordered pair $\{(P,N),(N,P)\}$.

\begin{proposition}[Projectivised coordinate arrangement]
\label{prop:sign-cells}
The sets $C_{P,N}$ are the open cells of a regular CW decomposition $\Cn$ of
$
 \Pj(H)\cong\mathbb RP^{n-2}.
$
If $k=|P|$, $\ell=|N|$, and $r=|Z|$, then
\begin{equation}\label{eq:cell-product-dimension}
\begin{aligned}
 C_{P,N}
 &\cong \operatorname{Int}\Delta^{k-1}\times\operatorname{Int}\Delta^{\ell-1},
 &\overline C_{P,N}
 &\cong \Delta^{k-1}\times\Delta^{\ell-1},\\
 \dim C_{P,N}&=n-r-2.
\end{aligned}
\end{equation}
A cell $C_{P',N'}$ is a face of $\overline C_{P,N}$ precisely when, after possibly interchanging its two labels,
$$
 \varnothing\ne P'\subseteq P,\qquad
 \varnothing\ne N'\subseteq N.
$$
Equivalently, faces are obtained by moving indices from $P$ or $N$ into $Z$.
\end{proposition}

\begin{proof}
Every nonzero vector in $H$ has at least one positive and one negative coordinate, so the cells partition $\Pj(H)$. Choose one of the two ordered representatives $(P,N)$ of $[P\mid N]$. Each projective class in $C_{P,N}$ has a unique representative satisfying
\begin{equation}\label{eq:cell-normalization}
 \sum_{i\in P}y_i=1=-\sum_{j\in N}y_j.
\end{equation}
The positive and negative coordinate lists are then points in $\operatorname{Int}\Delta^{k-1}$ and $\operatorname{Int}\Delta^{\ell-1}$, respectively. Allowing some of these coordinates to vanish defines a continuous characteristic map
$$
 \Delta^{k-1}\times\Delta^{\ell-1}\rightarrow\Pj(H),
 \qquad
 (u,v)\mapsto[(u,-v,0_Z)].
$$
It is injective: a positive projective rescaling is fixed by \eqref{eq:cell-normalization}, while a negative rescaling interchanges $P$ and $N$ and therefore cannot identify two points in the chosen parameter domain. Its image is exactly $\overline C_{P,N}$: allowing zero coordinates gives every face in the closure, and every limit of points in $C_{P,N}$ has such a normalised representative. Compactness and the Hausdorff property therefore make the characteristic map a homeomorphism onto $\overline C_{P,N}$. Its boundary is exactly the union obtained by deleting at least one positive or negative coordinate, without deleting an entire sign block. These are precisely the cells described in the face criterion. The characteristic maps are therefore regular, and the dimension is $(k-1)+(\ell-1)=n-r-2$.
\end{proof}

In the normalisation \eqref{eq:cell-normalization}, the cell barycentre is represented by
\begin{equation}\label{eq:cell-barycentre}
 b_{P,N,i}=
 \begin{cases}1/k,&i\in P,\\
-1/\ell,&i\in N,\\
0,&i\in Z.
 \end{cases}
\end{equation}
For $C=C_{P,N}$, let $b_C=[b_{P,N}]\in\Pj(H)$ be its projective barycentre. The duality map organises the cells; a second projective map retains the normalised $p$-th powers of pairwise distances and makes the edge weights linear parameters. For $[\bar x]\in\mathbb P(V_n)$ define
\begin{equation}\label{eq:Psi-def}
 \Psi_p([\bar x])_{ij}
 =\frac{|x_i-x_j|^p}{\mathfrak m_p(\bar x)},
 \qquad 1\leq i<j\leq n.
\end{equation}
Translation invariance and homogeneity make this well defined.

\begin{theorem}[Embedding by normalised $p$-th powers of pairwise distances]
\label{thm:projective-embedding}
For every $p>1$, the map
$$
 \Psi_p:\mathbb{R}P^{n-2}\rightarrow\mathbb R^m
$$
is a definable $C^1$ map and a topological embedding. At $p=2$ it is a real-analytic embedding, and for every $p>2$ it is a $C^2$ definable embedding. The family $(p,[\bar x])\mapsto\Psi_p([\bar x])$ is definable and $C^1$ jointly on $(2,\infty)\times\mathbb{R}P^{n-2}$. In every smooth projective chart, $(p,z)\mapsto D_z\Psi_p(z)$ is locally $C^1$, and $D_z^2\Psi_p(z)$ is jointly continuous. For every compact $I\Subset(1,\infty)$, moreover, the family is continuous as a map
$$
 p\mapsto\Psi_p\in
 C^1\bigl(\mathbb RP^{n-2},\mathbb R^{\binom n2}\bigr).
$$
\end{theorem}

\begin{proof}
Choose $p$-centred representatives $\widehat x,\widehat y$ normalised by
$$
 \sum_i|\widehat x_i|^p=\sum_i|\widehat y_i|^p=1.
$$
If their images under $\Psi_p$ agree, then taking the $2/p$ power coordinatewise gives a common squared-distance matrix $D$. Let $\Pi=I-n^{-1}\one\one^{\mathsf T}$. Double centring gives
$$
 -\frac12\Pi D\Pi
 =(\Pi\widehat x)(\Pi\widehat x)^{\mathsf T}
 =(\Pi\widehat y)(\Pi\widehat y)^{\mathsf T}.
$$
The configurations are nonconstant, so $\widehat y=\epsilon\widehat x+c\one$ for some $\epsilon\in\{-1,1\}$ and $c\in\mathbb R$. Both vectors are $p$-centred. The strictly increasing function $c\mapsto\sum_i\Phi_p(\epsilon\widehat x_i+c)$ vanishes at zero, hence $c=0$. Thus $\Psi_p$ is injective. Lemma~\ref{lem:p-centre} and the $C^1$ regularity of $|t|^p$ give $C^1$ regularity for every $p>1$; definability is immediate. Since the source is compact and the target is Hausdorff, $\Psi_p$ is a definable topological embedding.

At $p=2$, one has $c_2(u)=0$ on $H$ and $\mathfrak m_2(\bar u)=\|u\|_2^2$, so $\Psi_2$ is real analytic. Let $h$ be tangent to the unit sphere of $H$, so the derivative of the denominator vanishes. If $D\Psi_2(x)[h]=0$, then $h_i=h_j$ whenever $x_i\ne x_j$. These pairs form a connected complete multipartite graph, and hence $h=c\one$; since $h\in H$, one has $h=0$. Thus $\Psi_2$ is an immersion and hence a real-analytic embedding.

Assume now that $p>2$. For immersion, work on the normalised $p$-centred slice
\begin{equation}\label{eq:p-slice}
 \mathcal S_p=
 \left\{x:\sum_i\Phi_p(x_i)=0,\quad\sum_i|x_i|^p=1\right\}.
\end{equation}
Its two constraint gradients are independent at a nonconstant point. Indeed, if they were proportional, then every nonzero $x_i$ would have the same value, contradicting $p$-centring. The projection $\mathcal S_p\to\mathbb{R}P^{n-2}$ is therefore a $C^1$ two-sheeted covering map. Let $h$ be tangent to $\mathcal S_p$. The denominator is constant on this slice, so its derivative along $h$ vanishes. If $D\Psi_p(x)[h]=0$, then, whenever $x_i\ne x_j$,
$$
 0=D|x_i-x_j|^p[h]
 =p\Phi_p(x_i-x_j)(h_i-h_j),
$$
so $h_i=h_j$. The edges $\{i,j\}$ for which $x_i\ne x_j$ form the connected complete multipartite graph determined by the coordinate levels of $x$, and hence $h=c\one$. Tangency to the first constraint gives
$$
 0=(p-1)c\sum_i|x_i|^{p-2},
$$
so $c=0$. The differential is injective. An injective immersion from a compact manifold into a Hausdorff space is an embedding. The regularity claims follow from Lemma~\ref{lem:p-centre} and the regularity of $|t|^p$ for $p>2$. The assertion for compact $I$ is proved in Lemma~\ref{lem:Psi-C1-family} below.
\end{proof}

\begin{lemma}[Compact $C^1$ dependence on the exponent]
\label{lem:Psi-C1-family}
Let $I\Subset(1,\infty)$ be compact. Choose a finite precompact definable projective atlas whose chart closures cover $\mathbb RP^{n-2}$, together with definable smooth local sections whose images lie in compact subsets of $H\setminus\{0\}$. In these charts, the maps
$$
 (p,z)\mapsto\Psi_p(z),
 \qquad
 (p,z)\mapsto D_z\Psi_p(z)
$$
are continuous and definable on the product of $I$ with each chart closure. Equivalently, $p\mapsto\Psi_p$ is continuous from $I$ into $C^1(\mathbb RP^{n-2},\mathbb R^{\binom n2})$.
\end{lemma}

\begin{proof}
Work in the fixed mean-zero model $H=\one^\perp$ and use the atlas and local sections from the statement. By Lemma~\ref{lem:p-centre}, $(p,u)\mapsto c_p(u)$ is jointly continuous on the resulting compact sets.

The maps
$$
 (p,t)\mapsto |t|^p,
 \qquad
 (p,t)\mapsto p\Phi_p(t)
$$
are continuous on $(1,\infty)\times\mathbb R$, including at $t=0$. The envelope formula
$$
 D\mathfrak m_p(\bar u)[\bar h]
 =p\sum_i\Phi_p\bigl(\kappa_p(\bar u)_i\bigr)h_i
$$
therefore shows that both $\mathfrak m_p$ and its differential with respect to $\bar u$ are jointly continuous. The denominator is uniformly bounded away from zero on the compact chart section. Applying the quotient rule to the finitely many functions $|u_i-u_j|^p$ occurring in the numerator proves the assertion in that chart closure. The finitely many chart closures cover projective space, so these estimates give continuity in the global $C^1$ topology. Definability follows from definable choice and uniqueness.
\end{proof}

\subsection{Reduction by translations and Morse indices}

For positive edge weights $a=(a_{ij})$ on $K_n$, set
$$
 E_{p,a}(x)=\sum_{i<j}a_{ij}|x_i-x_j|^p,\qquad
 \mathcal R_{p,a}([\bar x])
 =\frac{E_{p,a}(x)}{\mathfrak m_p(\bar x)}.
$$
We write $(\Delta_{p,a}x)_i=\sum_{j\ne i}a_{ij}\Phi_p(x_i-x_j)$ for the corresponding unsigned zero-potential $p$-Laplacian.
Thus $\mathcal R_{p,a}$ is the pullback under $\Psi_p$ of the linear functional $z\mapsto a\cdot z$:
\begin{equation}\label{eq:height-identity}
 \mathcal R_{p,a}=a\cdot\Psi_p.
\end{equation}

\begin{proposition}[Reduced critical-point correspondence]
\label{prop:reduced-correspondence}
Let $p>1$ and $a_{ij}>0$. The nonconstant eigenlines of the unsigned, zero-potential, unit-vertex-measure complete graph are in bijection with the critical points of $\mathcal R_{p,a}$ on $\mathbb{R}P^{n-2}$. Each of the two $p$-centred representatives with $\sum_i|x_i|^p=1$ satisfies
$
 \Delta_{p,a}x=\mathcal R_{p,a}([\bar x])\Phi_p(x).
$
The constant eigenline is the unique eigenline with eigenvalue zero and is not represented in the reduced space.
\end{proposition}

\begin{proof}
Both $E_{p,a}$ and $\mathfrak m_p$ are $C^1$ on $V_n\setminus\{0\}$ for every $p>1$. If $[\bar x]$ is critical and $\lambda=\mathcal R_{p,a}([\bar x])$, the differential vanishes on a complement to the radial direction, while degree-zero homogeneity makes it vanish in the radial direction as well. Hence $D\mathcal R_{p,a}(\bar x)=0$ on $V_n$, and the quotient rule gives
$
 D(E_{p,a}-\lambda\mathfrak m_p)(\bar x)=0.
$
Lift this identity to $\mathbb R^n$. Each covector annihilates the translation direction, and \eqref{eq:mp-envelope}, which remains valid for $p>1$, gives
$
 p\Delta_{p,a}x=p\lambda\Phi_p(x)
$
at either normalised $p$-centred representative.

Conversely, let $(x,\lambda)$ be a nonconstant eigenpair. Pairing the eigen-equation with $x$ gives
$$
 \lambda=\frac{E_{p,a}(x)}{\sum_i|x_i|^p}>0.
$$
Summing its vertex equations gives $\sum_i\Phi_p(x_i)=0$, so $x$ is $p$-centred. The differential identity above then vanishes on $V_n$, and homogeneity identifies $\lambda$ with $\mathcal R_{p,a}([\bar x])$; hence $[\bar x]$ is critical.
\end{proof}

The height-function representation also makes the variation of a critical value with respect to the edge weights transparent. Fix $p>1$, and let $\Omega$ be an open subset of the positive weight space. Along any $C^1$ branch $a\mapsto[\bar x(a)]$, $a\in\Omega$, of reduced critical points, set $\lambda(a)=a\cdot\Psi_p([\bar x(a)])$. For every weight variation $\dot a$, the critical-point equation and the chain rule give
$
 D\lambda(a)[\dot a]=\dot a\cdot\Psi_p([\bar x(a)]).
$
Hence
$
 \nabla_a\lambda(a)=\Psi_p([\bar x(a)]).
$
For a $p$-centred representative this gives the Hellmann--Feynman formula
\begin{equation}\label{eq:edge-hellmann-feynman}
 \frac{\partial\lambda}{\partial a_{uv}}
 =\frac{|x_u-x_v|^p}{\mathfrak m_p(\bar x)}
 =\frac{|x_u-x_v|^p}{\sum_i|x_i|^p}.
\end{equation}
When the linearisation has only its scaling kernel, this is also the edge-weight specialisation of \cite[Theorem~2.4]{BerkolaikoHofmann2025}; this regularity condition is characterised in Proposition~\ref{prop:index-shift}.

\begin{proposition}[Morse index before and after reduction]\label{prop:index-shift}
Let $p>2$, and let $(x,\lambda)$ be a nonconstant normalised $p$-centred eigenpair for an unsigned complete graph with positive edge weights. Set
$$
 L=D_x\bigl(\Delta_{p,a}x-\lambda\Phi_p(x)\bigr),\qquad
 w=(|x_1|^{p-2},\ldots,|x_n|^{p-2}),
$$
and $W=\{h:w\cdot h=0\}$. Homogeneity and $p$-centring give $Lx=0$ and $x\in W$. Since $pL$ is the Hessian of $E_{p,a}-\lambda\sum_i|x_i|^p$, the operator $L$ is symmetric. Let
$$
 \mathcal Q_{p,a}(y)=\frac{E_{p,a}(y)}{\sum_i|y_i|^p}.
$$
On $x^\perp\simeq T_{[x]}\mathbb RP^{n-1}$, its Hessian is
\begin{equation}\label{eq:full-hessian-L}
 \operatorname{Hess}\mathcal Q_{p,a}([x])(h,k)
 =\frac{p}{\sum_i|x_i|^p}\langle Lh,k\rangle.
\end{equation}
On $W/\mathbb Rx\simeq T_{[\bar x]}\mathbb P(V_n)$,
\begin{equation}\label{eq:reduced-hessian-L}
 \operatorname{Hess}\mathcal R_{p,a}([\bar x])([h],[k])
 =\frac{p}{\mathfrak m_p(\bar x)}\langle Lh,k\rangle.
\end{equation}
The reduced critical point is nondegenerate if and only if $\ker L=\mathbb Rx$. We call such an eigenline \emph{regular} and define its Morse index on $\mathbb RP^{n-1}$ by
$
 \indm_{\mathrm{full}}(x,\lambda)
 =\indm_-\bigl(L|_{x^\perp}\bigr).
$
For a regular eigenline,
\begin{equation}\label{eq:index-shift}
 \indm_{\mathrm{full}}(x,\lambda)
 =1+\indm_{\mathrm{red}}([\bar x]).
\end{equation}
\end{proposition}

\begin{proof}
At a critical point,
$$
 D^2\mathcal R_{p,a}
 =\frac{D^2E_{p,a}-\lambda D^2\mathfrak m_p}{\mathfrak m_p}.
$$
For $h,k\in W$, the rank-one correction in \eqref{eq:mp-hessian} vanishes. Combining the remaining term in $D^2\mathfrak m_p$ with $D^2E_{p,a}$ gives \eqref{eq:reduced-hessian-L}.

Translation invariance of the edge term gives $L\one=-\lambda(p-1)w$. By symmetry, for $h\in W$, $\langle Lh,\one\rangle=\langle h,L\one\rangle=0$; hence $L(W)\subset\one^\perp$. If $h\in W$ belongs to the radical of the restricted form, then $Lh\in W^\perp=\mathbb Rw$. Since also $Lh\in\one^\perp$ and $w\cdot\one>0$, it follows that $Lh=0$. Conversely, every element of $\ker L\cap W$ lies in this radical.

If $y\in\ker L$, write $y=\alpha\one+h$ with $h\in W$. Pairing $Ly=0$ with $\one$ gives
$
 0=\alpha\langle L\one,\one\rangle,
$
because $\langle Lh,\one\rangle=0$. Since the nonconstant eigenvalue is positive,
$$
 \langle L\one,\one\rangle
 =-\lambda(p-1)\sum_i|x_i|^{p-2}<0,
$$
so $\alpha=0$. Thus $\ker L\subset W$, and the radical of the restricted form is $\ker L$. This proves the nondegeneracy criterion.

The decomposition $\mathbb R^n=\mathbb R\one\oplus W$ is $L$-orthogonal, and the translation summand is negative. After quotienting by the scaling kernel, the remaining form is the reduced Hessian up to the positive factor in \eqref{eq:reduced-hessian-L}. This proves \eqref{eq:index-shift}.
\end{proof}

\section{\texorpdfstring{The uniformly weighted $K_n$}{The uniformly weighted complete graph}}
\label{sec:uniform-cellular}

For uniform weights and $p\ne2$, the nonconstant eigenlines can be classified exactly. Under $\Theta_p$, they correspond to the cell barycentres; for $p>2$, this description also determines their Morse indices.

\begin{lemma}[Coordinate levels of eigenvectors of the uniformly weighted $K_n$]
\label{lem:uniform-level-collapse}
Let $p>1$, $p\ne2$, and let $x$ be a nonconstant eigenvector of the uniformly weighted $K_n$. Then its positive coordinates all have one value, its negative coordinates all have one value, and its remaining coordinates vanish.
\end{lemma}

\begin{proof}
Set $q=p-1$, so $q>0$ and $q\ne1$. Taking the scalar product of the eigen-equation with $x$ shows that a nonconstant vector has
$$
 \lambda\sum_i|x_i|^p
 =\sum_{i<j}|x_i-x_j|^p>0.
$$
Summing the component equations and using the oddness of $\Phi_p$ gives
\begin{equation}\label{eq:uniform-centre}
 \sum_i\Phi_p(x_i)=0.
\end{equation}
Thus both signs occur.

Write the negative coordinates as $-u_i$, $i\in I_-$, and the positive coordinates as $v_j$, $j\in I_+$, with $u_i,v_j>0$. Set
$$
 U=\max_{i\in I_-}u_i,\qquad
 u=\min_{i\in I_-}u_i,\qquad
 z=\#\{i:x_i=0\}.
$$
Equation \eqref{eq:uniform-centre} is
\begin{equation}\label{eq:uniform-q-balance}
 \sum_{i\in I_-}u_i^q=\sum_{j\in I_+}v_j^q.
\end{equation}
Evaluate the eigen-equation at a coordinate $-U$ and at a coordinate $-u$, and divide respectively by $-U^q$ and $-u^q$. This gives
\begin{align}
 \lambda
 &=\sum_{i\in I_-}\left(1-\frac{u_i}{U}\right)^q
+z+\sum_{j\in I_+}\left(1+\frac{v_j}{U}\right)^q,
\label{eq:uniform-at-Umax}\\
 \lambda
 &=-\sum_{i\in I_-}\left(\frac{u_i}{u}-1\right)^q
+z+\sum_{j\in I_+}\left(1+\frac{v_j}{u}\right)^q.
\label{eq:uniform-at-umin}
\end{align}

Suppose that $U>u$ and set $\delta=u^{-1}-U^{-1}>0$. Subtracting \eqref{eq:uniform-at-Umax} from \eqref{eq:uniform-at-umin} yields
\begin{align}
 0
 =\sum_{j\in I_+}
\left[
\left(1+\frac{v_j}{u}\right)^q
-\left(1+\frac{v_j}{U}\right)^q
\right]-\sum_{i\in I_-}
\left[
\left(\frac{u_i}{u}-1\right)^q
+\left(1-\frac{u_i}{U}\right)^q
\right].
 \label{eq:uniform-level-difference}
\end{align}
For $A,B\geq0$,
$$
 (A+B)^q
 \begin{cases}
 \geq A^q+B^q,&q>1,\\
 \leq A^q+B^q,&0<q<1,
 \end{cases}
$$
strictly when $A,B>0$. Since
$$
 1+\frac{v_j}{u}
 =1+\frac{v_j}{U}+v_j\delta
$$
and both summands on the right are positive,
\begin{equation}\label{eq:uniform-positive-comparison}
 \left(1+\frac{v_j}{u}\right)^q
 -\left(1+\frac{v_j}{U}\right)^q
 \begin{cases}
 >(v_j\delta)^q,&q>1,\\
 <(v_j\delta)^q,&0<q<1.
 \end{cases}
\end{equation}
Also
$$
 \left(\frac{u_i}{u}-1\right)
 +\left(1-\frac{u_i}{U}\right)=u_i\delta,
$$
with nonnegative summands, and hence
\begin{equation}\label{eq:uniform-negative-comparison}
 \left(\frac{u_i}{u}-1\right)^q
 +\left(1-\frac{u_i}{U}\right)^q
 \begin{cases}
 \leq(u_i\delta)^q,&q>1,\\
 \geq(u_i\delta)^q,&0<q<1.
 \end{cases}
\end{equation}
Equality here occurs exactly when $u_i\in\{u,U\}$, whereas \eqref{eq:uniform-positive-comparison} is strict for every $j\in I_+$.

If $q>1$, the right side of \eqref{eq:uniform-level-difference} is therefore strictly greater than
$$
 \delta^q\left(\sum_{j\in I_+}v_j^q-\sum_{i\in I_-}u_i^q\right)=0
$$
by \eqref{eq:uniform-q-balance}; if $0<q<1$, it is strictly less than the same zero quantity. Both conclusions contradict \eqref{eq:uniform-level-difference}. Hence $U=u$, so every negative coordinate has the same value. Applying the argument to $-x$ proves the same assertion for the positive coordinates. All remaining coordinates are zero.
\end{proof}

\begin{proposition}[Eigenlines of the uniformly weighted $K_n$, regularity, and inertia]
\label{prop:uniform-rays-inertia}
Let $p>1$, $p\ne2$. Every nonconstant eigenline of the uniformly weighted $K_n$ is indexed by a partition into labelled blocks
$$
 \{1,\ldots,n\}=P\sqcup N\sqcup Z,\qquad
 k=|P|\geq1,\quad\ell=|N|\geq1,\quad r=|Z|,
$$
with $(P,N,Z)$ identified with $(N,P,Z)$, and has a representative
\begin{equation}\label{eq:uniform-ray}
 x_i=
 \begin{cases}
 \ell^{1/(p-1)},&i\in P,\\
 -k^{1/(p-1)},&i\in N,\\
 0,&i\in Z.
 \end{cases}
\end{equation}
Its eigenvalue is
\begin{equation}\label{eq:uniform-ray-value}
 \lambda_{k,\ell,r}
 =\bigl(k^{1/(p-1)}+\ell^{1/(p-1)}\bigr)^{p-1}+r.
\end{equation}
There are
$$
 R_n-1=\frac{3^n-2^{n+1}+1}{2}
$$
nonconstant eigenlines.

If $p>2$, every one of these eigenlines is regular. The quadratic form induced by the full linearisation after quotienting by its scaling kernel has negative index $n-r-1$, while the reduced Morse index is $n-r-2$.
\end{proposition}

\begin{proof}
Lemma~\ref{lem:uniform-level-collapse} gives the three-level form. Equation \eqref{eq:uniform-centre} fixes the ratio of the two nonzero levels, and direct substitution gives \eqref{eq:uniform-ray} and \eqref{eq:uniform-ray-value}.

Assume now that $p>2$. We compute the regularity and inertia directly. Set
$$
 \alpha=\ell^{1/(p-1)},\quad
 \beta=k^{1/(p-1)},\quad
 A=\alpha^{p-2},\quad B=\beta^{p-2},\quad
 C=(\alpha+\beta)^{p-2}.
$$
Let $L=D_x(\Delta_{p,\one}x-\lambda_{k,\ell,r}\Phi_p(x))$. Apart from the positive factor $p-1$, the quadratic form of the linearisation is
\begin{equation}\label{eq:uniform-linearized-form}
 \frac{\langle Ly,y\rangle}{p-1}
 =\sum_{i<j}|x_i-x_j|^{p-2}(y_i-y_j)^2
 -\lambda_{k,\ell,r}\sum_i|x_i|^{p-2}y_i^2.
\end{equation}
Decompose $\mathbb R^n$ into the zero-sum subspaces supported on $P$ and $N$, together with the zero-sum subspace supported on $Z$ when $r>0$, and the subspace of vectors constant on each non-empty block. The first two subspaces have dimensions $k-1$ and $\ell-1$; when $r>0$, the third has dimension $r-1$. On them $L/(p-1)$ acts respectively by
$$
 -A\beta C,\qquad -B\alpha C,\qquad kA+\ell B,
$$
where the third expression is present only when $r>0$.

On the block-constant subspace, set $y=u_P\one_P+u_N\one_N+u_Z\one_Z$, omitting the last term when $r=0$. Direct substitution in \eqref{eq:uniform-linearized-form} shows that $Ly=0$ precisely when
$$
 \beta u_P+\alpha u_N=0,\qquad u_Z=0\quad(r>0),
$$
so the kernel on this subspace is $\mathbb Rx$. The all-ones vector is a negative direction, while, if $r>0$, $\one_Z$ is a positive direction. Since the block-constant subspace has dimension two for $r=0$ and three for $r>0$, its inertia is thereby determined. Together with the zero-sum block subspaces above, this gives
$$
 \indm_-(L)=n-r-1,\qquad
 \indm_+(L)=r,\qquad\dim\ker L=1.
$$
Proposition~\ref{prop:index-shift} gives the reduced index.

For the eigenline count, which applies to every $p\ne2$, fix $r$, choose $Z$, and then a nontrivial sign partition of its complement, modulo global sign. This gives $\binom nr(2^{n-r-1}-1)$ eigenlines. Summing over $0\leq r\leq n-2$ gives the stated total.
\end{proof}

\begin{corollary}[Cell-barycentre parametrisation of the eigenlines of the uniformly weighted $K_n$]
\label{cor:uniform-barycentres}
For every $p>1$ with $p\ne2$, the nonconstant eigenlines of the uniformly weighted $K_n$ are exactly
$$
 \Theta_p^{-1}([b_{P,N}]),\qquad C_{P,N}\in\Cn.
$$
For $p>2$, the reduced Morse index of the eigenline associated with $C_{P,N}$ is $\dim C_{P,N}$.
\end{corollary}

\begin{proof}
For the representative in \eqref{eq:uniform-ray}, $\Phi_p(x_i)=\ell$ on $P$, $\Phi_p(x_j)=-k$ on $N$, and $\Phi_p(x_z)=0$ on $Z$. Division by $k\ell$ gives \eqref{eq:cell-barycentre}. Proposition~\ref{prop:uniform-rays-inertia} gives the index $n-r-2=\dim C_{P,N}$ when $p>2$.
\end{proof}

For $p\ne2$, the classification associates an eigenvalue branch with each cell, and the resulting formula extends through $p=2$. Its dependence on $p$ is encoded by the two nonzero block sizes. For a cell $C=C_{P,N}$ with block sizes $(k,\ell,r)$ and for $p>1$, define
$$
 \lambda_p(C)=r+
 \bigl(k^{1/(p-1)}+\ell^{1/(p-1)}\bigr)^{p-1}.
$$
For $p\ne2$, this is the eigenvalue in \eqref{eq:uniform-ray-value}. Set
$$
 h(C)=(k+\ell)\log(k+\ell)-k\log k-\ell\log\ell.
$$

\begin{proposition}[Dependence on $p$ of the block values for the uniformly weighted $K_n$]
\label{prop:uniform-entropy}
Let $p>1$, and let a cell $C$ have block sizes $(k,\ell,r)$. Set $q=p-1$ and define
$$
 F_{k,\ell}(q)=\left(k^{1/q}+\ell^{1/q}\right)^q,
 \qquad
 \theta_{k,\ell}(q)=
 \frac{k^{1/q}}{k^{1/q}+\ell^{1/q}}.
$$
Let
$$
 \mathsf H(\theta)
 =-\theta\log\theta-(1-\theta)\log(1-\theta)
$$
be the binary entropy with natural logarithms. Then
\begin{align}
 \lambda_p(C)&=r+F_{k,\ell}(p-1),\label{eq:entropy-value}\\
 \frac{d}{dp}\log(\lambda_p(C)-r)
 &=\mathsf H\!\left(\theta_{k,\ell}(p-1)\right)>0,
 \label{eq:entropy-first}\\
 \frac{d^2}{dp^2}\log(\lambda_p(C)-r)
 &=\frac{\theta_{k,\ell}(q)(1-\theta_{k,\ell}(q))}{q^3}
\left(\log\frac{k}{\ell}\right)^2\geq0.
 \label{eq:entropy-second}
\end{align}
Moreover,
\begin{equation}\label{eq:entropy-linear-jet}
 \lambda_2(C)=n,
 \qquad
 \left.\frac{d}{dp}\lambda_p(C)\right|_{p=2}
 =h(C),
\end{equation}
and, for every $p>1$,
\begin{equation}\label{eq:entropy-speed-bound}
 0<\lambda_p'(C)
 \leq(\log2)\bigl(\lambda_p(C)-r\bigr),
\end{equation}
with equality in the upper bound if and only if $k=\ell$. Consequently, $\lambda_p(C)$ is strictly increasing in $p$, the logarithm of $\lambda_p(C)-r$ is convex (strictly unless $k=\ell$), and
$$
 \lambda_p(C)<n\quad(1<p<2),
 \qquad
 \lambda_p(C)>n\quad(p>2).
$$
After interchanging $k$ and $\ell$ if necessary so that $k\leq\ell$,
\begin{equation}\label{eq:entropy-endpoint-limits}
 \lim_{p\downarrow1}\lambda_p(C)=n-k
 =n-\min\{k,\ell\},
 \qquad
 \lim_{p\to\infty}2^{1-p}\lambda_p(C)=\sqrt{k\ell}.
\end{equation}
\end{proposition}

\begin{proof}
Let
$$
 A=k^{1/q},\quad B=\ell^{1/q},\quad
 S=A+B,
 \qquad \theta=A/S,
$$
so that $F_{k,\ell}(q)=S^q$. Differentiating gives
\begin{align*}
 \frac{d}{dq}\log F_{k,\ell}(q)
 &=\log S-\frac{A\log k+B\log\ell}{qS}\\
 &=\log S-\theta\log A-(1-\theta)\log B\\
 &=-\theta\log\theta-(1-\theta)\log(1-\theta)
 =\mathsf H(\theta).
\end{align*}
Since $dq/dp=1$, this proves \eqref{eq:entropy-first}. Moreover,
$$
 \log\frac{\theta}{1-\theta}
 =\frac{\log k-\log\ell}{q},
$$
so
$$
 \theta'
 =\frac{\theta(1-\theta)}{q^2}(\log\ell-\log k),
 \qquad
 \mathsf H'(\theta)
 =\log\frac{1-\theta}{\theta}
 =\frac{\log\ell-\log k}{q}.
$$
Their product is \eqref{eq:entropy-second}.

At $q=1$, one has $F_{k,\ell}(1)=k+\ell$ and $\theta=k/(k+\ell)$. Hence
$$
 F_{k,\ell}'(1)
 =(k+\ell)\mathsf H\left(\frac{k}{k+\ell}\right)
 =h(C),
$$
which proves \eqref{eq:entropy-linear-jet}. The standard bound $\mathsf H(\theta)\leq\log2$, with equality precisely at $\theta=1/2$, gives \eqref{eq:entropy-speed-bound}. Strict positivity of binary entropy proves monotonicity and the inequalities on the two sides of $2$.

If $k\leq\ell$, then
$$
 F_{k,\ell}(q)
 =\ell\left(1+(k/\ell)^{1/q}\right)^q
 \rightarrow\ell
 \qquad(q\downarrow0),
$$
with the same conclusion when $k=\ell$, since then $F_{k,k}(q)=2^qk$. This proves the first endpoint limit. For the second,
$$
 2^{-q}F_{k,\ell}(q)
 =\left(\frac{k^{1/q}+\ell^{1/q}}2\right)^q
 \rightarrow\sqrt{k\ell},
$$
and the additive term $r$ disappears after multiplication by $2^{-q}$.
\end{proof}
\begin{corollary}[Second-order expansion at $p=2$]
\label{cor:uniform-second-order}
For a fixed cell $C$ with block sizes $(k,\ell,r)$, set $s=k+\ell$. As $p\to2$,
\begin{align}
 \lambda_p(C)
 =n+h(C)(p-2)
 +\frac{h(C)^2+k\ell\bigl(\log(k/\ell)\bigr)^2}{2s}(p-2)^2+O((p-2)^3).
 \label{eq:uniform-second-order}
\end{align}
In particular,
$$
 \left.\frac{d^2}{dp^2}\lambda_p(C)\right|_{p=2}
 =\frac{h(C)^2+k\ell\bigl(\log(k/\ell)\bigr)^2}{k+\ell}.
$$
\end{corollary}

\begin{proof}
At $q=1$, equation \eqref{eq:entropy-second} gives
$$
 \left(\log F_{k,\ell}\right)''(1)
 =\frac{k\ell}{(k+\ell)^2}
 \left(\log\frac{k}{\ell}\right)^2,
$$
while \eqref{eq:entropy-linear-jet} gives $(\log F_{k,\ell})'(1)=h(C)/(k+\ell)$. Hence
$$
 F_{k,\ell}''(1)
 =\frac{h(C)^2+k\ell(\log(k/\ell))^2}{k+\ell}.
$$
Taylor's theorem and $q-1=p-2$ prove the assertion.
\end{proof}
\begin{proposition}[Ordering along the face poset]
\label{prop:face-order}
Under a proper face inclusion $C'\subset\overline C$,
\begin{equation}\label{eq:face-order}
\begin{aligned}
 \lambda_p(C')&<\lambda_p(C) &&(p>2),&
 \lambda_p(C')&>\lambda_p(C) &&(1<p<2),\\
 h(C')&<h(C).&&&
\end{aligned}
\end{equation}
\end{proposition}

\begin{proof}
It is enough to treat a codimension-one face, since every proper face inclusion is a chain of such inclusions. Set $s=p-1$ and
$
 F_s(k,\ell)=\bigl(k^{1/s}+\ell^{1/s}\bigr)^s.
$
Moving one index from the positive block to the zero block (necessarily $k\geq2$) changes $(k,\ell,r)$ to $(k-1,\ell,r+1)$, and
$$
 \lambda_p(C)-\lambda_p(C')
 =F_s(k,\ell)-F_s(k-1,\ell)-1.
$$
For real $k>0$,
\begin{equation}\label{eq:F-derivative}
 \frac{\partial F_s}{\partial k}
 =\left(1+\left(\frac{\ell}{k}\right)^{1/s}\right)^{s-1}.
\end{equation}
This derivative is strictly greater than $1$ when $s>1$ and strictly less than $1$ when $0<s<1$. Integration over $[k-1,k]$ proves the first two inequalities. The same argument applies when an index is moved from the negative block. For the logarithmic value,
$$
 \frac{\partial}{\partial k}
 \bigl((k+\ell)\log(k+\ell)-k\log k-\ell\log\ell\bigr)
 =\log\left(1+\frac{\ell}{k}\right)>0,
$$
and symmetrically in $\ell$, which proves the logarithmic inequality.
\end{proof}

\begin{corollary}[Extremal positive values of the uniformly weighted $K_n$]
\label{cor:uniform-extremes}
Let $q=p-1$, and set
$$
 a_n=\left\lfloor\frac n2\right\rfloor,\qquad
 b_n=\left\lceil\frac n2\right\rceil,
$$
$$
 \lambda_{\mathrm{edge}}(n,p)=n-2+2^q,\qquad
 \lambda_{\mathrm{bal}}(n,p)
 =\left(a_n^{1/q}+b_n^{1/q}\right)^q.
$$
For $p>2$, the smallest and largest positive eigenvalues of the uniformly weighted $K_n$ are respectively $\lambda_{\mathrm{edge}}(n,p)$ and $\lambda_{\mathrm{bal}}(n,p)$. For $1<p<2$, the order is reversed. At $p=2$, both expressions equal $n$.

For $n\geq3$ and $p\ne2$, the value $\lambda_{\mathrm{edge}}(n,p)$ is attained precisely on the permutation orbit of block type $(1,1,n-2)$, while $\lambda_{\mathrm{bal}}(n,p)$ is attained precisely on the permutation orbit of the balanced block type with $r=0$. The first contains $\binom n2$ eigenlines; the second contains
$$
 \begin{cases}
 \displaystyle\binom n{\lfloor n/2\rfloor},&n\ \hbox{odd},\\[6pt]
 \displaystyle\frac12\binom n{n/2},&n\ \hbox{even}.
 \end{cases}
$$
For $n=2$, the two orbit types coincide.
\end{corollary}

\begin{proof}
Every cell contains a vertex in its closure and is a face of a top-dimensional cell. Proposition~\ref{prop:face-order} therefore reduces the global extrema to vertices and top-dimensional cells. Every vertex has block sizes $(1,1,n-2)$, giving $n-2+2^q$.

On a top-dimensional cell, $r=0$ and $k+\ell=n$. If $q>1$, the function
$
 k\mapsto k^{1/q}+(n-k)^{1/q}
$
is strictly concave and symmetric, hence is maximised at the split $(\lfloor n/2\rfloor,\lceil n/2\rceil)$. If $0<q<1$, it is strictly convex and is minimised at that split. Together with the reversal of face order for $1<p<2$, this proves the assertions about the extrema. The strict inequalities show that equality occurs only for the stated block types; their orbit sizes follow from the labelled-partition count modulo global sign.
\end{proof}

We write $\sigma_p(K_n)$ for the spectrum of the uniformly weighted, unsigned, zero-potential, unit-measure complete graph.
\begin{corollary}[The positive eigenvalue nearest to $n$]
\label{cor:uniform-opening-gap}
For every $n\geq2$ and every $p>1$,
\begin{equation}\label{eq:uniform-opening-gap}
 \operatorname{dist}\bigl(n,\sigma_p(K_n)\setminus\{0\}\bigr)
 =\bigl|2^{p-1}-2\bigr|.
\end{equation}
For $n\geq3$ and $p\ne2$, this distance is attained precisely by the $(1,1,n-2)$ orbit. In particular,
\begin{equation}\label{eq:uniform-opening-rate}
 \operatorname{dist}\bigl(n,\sigma_p(K_n)\setminus\{0\}\bigr)
 =2\log2\,|p-2|+O((p-2)^2)
 \qquad(p\to2).
\end{equation}
\end{corollary}

\begin{proof}
Proposition~\ref{prop:uniform-entropy} places every positive spectral value of the uniformly weighted $K_n$ below $n$ when $1<p<2$ and above $n$ when $p>2$. Corollary~\ref{cor:uniform-extremes} says that the positive value nearest to $n$ is therefore
$
 \lambda_{\mathrm{edge}}(n,p)=n-2+2^{p-1}.
$
This proves \eqref{eq:uniform-opening-gap} and the equality statement. The expansion \eqref{eq:uniform-opening-rate} is the Taylor expansion of $2^{p-1}$ at $p=2$.
\end{proof}

The preceding results locate the extreme values; we next determine when distinct block types produce the same spectral value. For
$
 \mathcal I_n=
 \{(k,\ell)\in\mathbb N^2:1\leq k\leq\ell,\ k+\ell\leq n\},
$
define
\begin{equation}\label{eq:uniform-block-value}
 \Lambda^{(n)}_{k,\ell}(p)
 =n-k-\ell+
 \left(k^{1/(p-1)}+\ell^{1/(p-1)}\right)^{p-1}.
\end{equation}

\begin{proposition}[Coincidences among the eigenvalues of the uniformly weighted $K_n$ occur at finitely many exponents]
\label{prop:uniform-exceptional}
Define the collision set
$$
 \mathfrak E_n=
 \left\{p>1:
 \begin{array}{l}
 \Lambda^{(n)}_{k,\ell}(p)=
 \Lambda^{(n)}_{k',\ell'}(p)
 \text{ for distinct}\\[-2pt]
 (k,\ell),(k',\ell')\in\mathcal I_n
 \end{array}
 \right\}.
$$
Then:
\begin{enumerate}[label=\textup{(\roman*)}]
\item For every $p>1$,
\begin{equation}\label{eq:uniform-value-set}
 \sigma_p(K_n)=\{0\}\cup
 \{\Lambda^{(n)}_{k,\ell}(p):(k,\ell)\in\mathcal I_n\}.
\end{equation}
At $p=2$, all displayed nonzero values equal $n$.

\item The set $\mathfrak E_n$ is finite. If $p\notin\mathfrak E_n$, then
\begin{equation}\label{eq:generic-uniform-count}
 \#\sigma_p(K_n)
 =1+|\mathcal I_n|
 =\left\lfloor\frac{n^2}{4}\right\rfloor+1.
\end{equation}
Moreover, $2\in\mathfrak E_n$ if and only if $n\geq3$.

\item At $p=3$,
\begin{equation}\label{eq:p3-block-value}
 \Lambda^{(n)}_{k,\ell}(3)=n+2\sqrt{k\ell},
 \qquad
 \#\sigma_3(K_n)
 =1+\#\{k\ell:(k,\ell)\in\mathcal I_n\}.
\end{equation}
Thus $3\in\mathfrak E_n$ if and only if $n\geq5$. In particular,
$$
 \Lambda^{(n)}_{1,4}(3)=\Lambda^{(n)}_{2,2}(3)=n+4,
 \qquad
 \#\sigma_3(K_5)=6.
$$

\item For $(k,\ell)\in\mathcal I_n$, set $r=n-k-\ell$ and
$$
 \mu_{k,\ell,r}=
 \begin{cases}
 \displaystyle\frac{n!}{k!\,\ell!\,r!},&k<\ell,\\[7pt]
 \displaystyle\frac{n!}{2(k!)^2r!},&k=\ell.
 \end{cases}
$$
For every $p>1$ with $p\ne2$ and every positive spectral value $\lambda$, the number of eigenlines with eigenvalue $\lambda$ is
\begin{equation}\label{eq:uniform-value-multiplicity}
 \sum_{\substack{(k,\ell)\in\mathcal I_n\\\Lambda^{(n)}_{k,\ell}(p)=\lambda}}
 \mu_{k,\ell,n-k-\ell}.
\end{equation}
Outside $\mathfrak E_n$, this sum has one term.
\end{enumerate}
\end{proposition}

\begin{proof}
The complete-graph classification gives \eqref{eq:uniform-value-set} for $p\ne2$, and the quadratic identity on $\one^\perp$ gives it at $p=2$.

Set $q=p-1$. For $k\leq\ell$,
$
 \lim_{q\downarrow0}
 (k^{1/q}+\ell^{1/q})^q=\ell,
$
including $k=\ell$, where the expression is $2^qk$. Therefore
\begin{equation}\label{eq:uniform-q0-limit}
 \lim_{q\downarrow0}\Lambda^{(n)}_{k,\ell}(1+q)=n-k.
\end{equation}
If $\Lambda^{(n)}_{k,\ell}\equiv\Lambda^{(n)}_{k',\ell'}$, then \eqref{eq:uniform-q0-limit} gives $k=k'$. At the other endpoint,
\begin{equation}\label{eq:uniform-qinf-limit}
 \lim_{q\to\infty}
 2^{-q}\left(k^{1/q}+\ell^{1/q}\right)^q=\sqrt{k\ell}.
\end{equation}
Indeed, after taking logarithms,
$$
 q\log\left(
 \frac{e^{(\log k)/q}+e^{(\log\ell)/q}}2
 \right)
 \rightarrow\frac{\log k+\log\ell}{2}.
$$
The additive term $n-k-\ell$ disappears after multiplication by $2^{-q}$. Equation~\eqref{eq:uniform-qinf-limit} then gives $k\ell=k'\ell'$, and hence $\ell=\ell'$.

Every difference $\Lambda^{(n)}_{k,\ell}-\Lambda^{(n)}_{k',\ell'}$ associated with distinct pairs $(k,\ell)$ and $(k',\ell')$ is consequently a nonzero real-analytic function on $(1,\infty)$ and is definable in $\mathbb R_{\exp}$. Its zero set contains no interval and hence is finite by o-minimality. The union over finitely many pairs is $\mathfrak E_n$. Moreover,
$$
 |\mathcal I_n|
 =\sum_{s=2}^n\left\lfloor\frac s2\right\rfloor
 =\left\lfloor\frac{n^2}{4}\right\rfloor.
$$
At $p=2$, all the functions $\Lambda^{(n)}_{k,\ell}$ have the same value whenever $|\mathcal I_n|\geq2$, which is equivalent to $n\geq3$.

Equation~\eqref{eq:p3-block-value} follows directly from \eqref{eq:uniform-block-value}. The admissible products are distinct for $n\leq4$, whereas the $(1,4)$ and $(2,2)$ products collide for every $n\geq5$. The number $\mu_{k,\ell,r}$ is the multinomial count of labelled partitions with block sizes $(k,\ell,r)$, divided by two exactly when $k=\ell$. Summing over every block pair producing the same value proves \eqref{eq:uniform-value-multiplicity}.
\end{proof}

\begin{corollary}[The spectrum of the uniformly weighted $K_n$ at $p=3$ and the multiplication-table problem]
\label{cor:p3-multiplication}
Let $ \delta=1-(1+\log\log2)/\log2=0.086071332\ldots $. Then
\begin{equation}\label{eq:p3-multiplication-order}
 \#\sigma_3(K_n)
 \asymp
 \frac{n^2}
 {(\log n)^\delta(\log\log n)^{3/2}}
 \qquad(n\to\infty).
\end{equation}
\end{corollary}

\begin{proof}
Let
$$
 T(n)=\#\{k\ell:(k,\ell)\in\mathcal I_n\},\qquad
 M(N)=\#\{ab:1\leq a,b\leq N\}.
$$
After sorting the factors,
$
 M(\lfloor n/2\rfloor)\leq T(n)\leq M(n).
$
Ford's theorem on the number of distinct products $ab$ with $1\leq a,b\leq N$ gives
$$
 M(N)\asymp
 \frac{N^2}
 {(\log N)^\delta(\log\log N)^{3/2}}
$$
\cite[Corollary~3]{Ford2008}. The bounds above and \eqref{eq:p3-block-value} prove the result.
\end{proof}

Consequently, the number of spectral values of the uniformly weighted $K_n$ at $p=3$ is $o(n^2)$ but $n^{2-o(1)}$, whereas the number of eigenlines grows as $3^n/2$.

\begin{remark}[The $(1,4)$--$(2,2)$ collision at $p=3$]
For every $n\geq5$, the common value $n+4$ is attained on exactly
$$
 5\binom n5+3\binom n4
$$
eigenlines: the first term comes from $(1,4,n-5)$, and the second from $(2,2,n-4)$. Their reduced Morse indices are respectively $3$ and $2$. Thus the quotient $\mathcal R_{3,\one}$ is Morse but has critical points of different indices sharing a critical value.
\end{remark}

For $p>2$, collisions of spectral values do not affect the regularity of the individual eigenlines. Counting critical points by index retains the cellular information lost when equal values are identified.

\begin{theorem}[Morse polynomial and indices for the uniformly weighted $K_n$]
\label{thm:exact-morse-polynomial}
Fix $p>2$. The reduced quotient $\mathcal R_{p,\one}$ for the uniformly weighted $K_n$ is Morse on $\mathbb{R}P^{n-2}$. For $0\leq q\leq n-2$, the number of critical points of reduced Morse index $q$ is
\begin{equation}\label{eq:index-count-reduced}
 c_{n,q}=\binom n{q+2}(2^{q+1}-1).
\end{equation}
Its Morse polynomial is
\begin{align}
 M_n^+(t)
 &=\sum_{q=0}^{n-2}\binom n{q+2}(2^{q+1}-1)t^q
 \label{eq:M-sum}\\
 &=\frac{(1+2t)^n-2(1+t)^n+1}{2t^2}.
 \label{eq:M-closed}
\end{align}
In particular,
\begin{equation}\label{eq:M-checks}
 M_n^+(1)=\frac{3^n-2^{n+1}+1}{2}=R_n-1,\qquad
 M_n^+(-1)=\frac{1+(-1)^n}{2}
 =\chi(\mathbb{R}P^{n-2}).
\end{equation}
\end{theorem}

\begin{proof}
Propositions~\ref{prop:reduced-correspondence} and~\ref{prop:uniform-rays-inertia} show that every critical point is nondegenerate. An eigenline with $r$ zero coordinates has reduced index $q=n-r-2$. Choosing its zero set and its nontrivial sign partition gives
$$
 \binom nr(2^{n-r-1}-1)
 =\binom n{q+2}(2^{q+1}-1).
$$
Equivalently, a $q$-cell is obtained by choosing a support of size $q+2$ and a nonconstant sign vector on that support modulo global sign. Thus $M_n^+$ is the face polynomial of $\Cn$. The binomial theorem gives \eqref{eq:M-closed}, and the two specialisations follow.
\end{proof}

Outside $\mathfrak E_n$, the $R_n-1$ eigenlines give exactly $\lfloor n^2/4\rfloor$ nonzero spectral values. At $p=3$ and $n\geq5$, the additional coincidences are determined by equal products $k\ell$ subject to $k+\ell\leq n$. For each fixed $p>2$, sufficiently small generic perturbations preserve these $R_n-1$ critical points and separate their values.

The quadratic case is different: the reduced uniform quotient is constant, whereas generic weights give a perfect Morse function with the minimum number of critical points permitted by the $\mathbb F_2$-Betti numbers.

\begin{proposition}[The quadratic case: generic weights give an $\mathbb F_2$-perfect Morse function]
\label{prop:p2-perfect}
At $p=2$ the reduced quotient for the uniformly weighted $K_n$ on $\mathbb{R}P^{n-2}$ is identically $n$. For generic positive edge weights on $K_n$, the reduced quadratic quotient has exactly $n-1$ critical points, one in each Morse index $0,\ldots,n-2$, and Morse polynomial $1+t+\cdots+t^{n-2}$. It is perfect over $\mathbb F_2$.
\end{proposition}

\begin{proof}
For $x\perp\one$,
$
 \sum_{i<j}(x_i-x_j)^2=n\sum_i x_i^2.
$
The discriminant of the characteristic polynomial of the weighted Laplacian restricted to $\one^\perp$ is a polynomial in the edge weights. This polynomial is not identically zero: a path has simple Laplacian spectrum, and assigning sufficiently small positive weights to its missing edges preserves simplicity. Its zero set is therefore a relatively closed definable subset of codimension at least one. Outside this set, the restriction has simple spectrum. Its $n-1$ ordered eigenlines are the critical points of the quadratic Rayleigh quotient on projective space and have successive indices $0,\ldots,n-2$. Since $b_j(\mathbb RP^{n-2};\mathbb F_2)=1$ for $0\leq j\leq n-2$, the Morse function is perfect over $\mathbb F_2$.
\end{proof}

\section{The singular transition at \texorpdfstring{$p=2$}{p=2}}
\label{sec:singular-transition}

For uniform weights, the collapse at $p=2$ conceals the $R_n-1$ critical points present on either side. Rescaling the quotient by $p-2$ reveals the logarithmic functional governing the transition.

\begin{lemma}[Logarithmic divided differences near $p=2$]
\label{lem:scalar-log}
For every $M>0$, as $p\to2$,
\begin{equation}\label{eq:scalar-log-C1}
 \frac{|u|^p-u^2}{p-2}\rightarrow u^2\log|u|
 \quad\hbox{in }C^1([-M,M]),
\end{equation}
where $0^2\log0=0$ and the derivative at zero is zero. Moreover, uniformly on $[-M,M]$,
\begin{equation}\label{eq:Phi-log-uniform}
 \Phi_p(u)\rightarrow u,\qquad
 \frac{\Phi_p(u)-u}{p-2}\rightarrow u\log|u|.
\end{equation}
\end{lemma}

\begin{proof}
Let $\epsilon=p-2$ and, for $\epsilon\ne0$, set
$$
 \gamma_\epsilon(u)=
 \frac{\Phi_{2+\epsilon}(u)-u}{\epsilon},
 \qquad
 F_\epsilon(u)=
 \frac{|u|^{2+\epsilon}-u^2}{\epsilon}.
$$
At $\epsilon=0$ set $\gamma_0(u)=u\log|u|$ and $F_0(u)=u^2\log|u|$, with value zero at the origin. For $u\ne0$,
\begin{equation}\label{eq:scalar-mean-exponent}
 \gamma_\epsilon(u)
 =u\log|u|\int_0^1|u|^{\theta\epsilon}\,d\theta,
 \qquad F_\epsilon(u)=u\gamma_\epsilon(u).
\end{equation}
Fix $0<\epsilon_0<1$. If $0<|u|\leq1$ and $|\epsilon|\leq\epsilon_0$, then
$
 |\gamma_\epsilon(u)|
 \leq |u|^{1-\epsilon_0}|\log|u||.
$
The right-hand side tends to zero as $u\to0$. On $\delta\leq|u|\leq M$, the expression in \eqref{eq:scalar-mean-exponent} converges uniformly to $u\log|u|$. Together with the preceding estimate on $|u|\leq\delta$, this proves $\gamma_\epsilon\to\gamma_0$ uniformly. Also,
$
 F_\epsilon'(u)=u+(2+\epsilon)\gamma_\epsilon(u),
$
including at $u=0$ by continuous extension. This proves \eqref{eq:scalar-log-C1}; the two assertions in \eqref{eq:Phi-log-uniform} follow at once.
\end{proof}

The scalar expansion must next be propagated through the $p$-dependent minimising translation in the denominator.

\begin{lemma}[Derivative of the $p$-centre at $p=2$]
\label{lem:centre-derivative}
Let $\Sph(H)$ be the Euclidean unit sphere in $H=\one^\perp$. Uniformly for $x\in\Sph(H)$,
\begin{equation}\label{eq:centre-derivative}
 c_p(x)=O(|p-2|),\qquad
 \frac{c_p(x)}{p-2}\rightarrow
 \dot c(x):=-\frac1n\sum_i x_i\log|x_i|.
\end{equation}
If
$$
 S(x)=\sum_i x_i^2,\qquad
 B(x)=\sum_i x_i^2\log|x_i|,
$$
then
\begin{equation}\label{eq:mass-C1-expansion}
 \frac{\mathfrak m_p-S}{p-2}\rightarrow B
 \quad\hbox{in }C^1(\Sph(H)).
\end{equation}
\end{lemma}

\begin{proof}
Set $\epsilon=p-2$ and $c_\epsilon=c_{2+\epsilon}(x)$. A minimising translation lies between $-\max_i x_i$ and $-\min_i x_i$, so all $x_i+c_\epsilon$ remain in a common bounded interval. The centring equation and Lemma~\ref{lem:scalar-log} give the exact identity
$$
 0=\sum_i\Phi_{2+\epsilon}(x_i+c_\epsilon)
 =n c_\epsilon+\epsilon\sum_i\gamma_\epsilon(x_i+c_\epsilon).
$$
Uniform boundedness of $\gamma_\epsilon$ first gives $c_\epsilon=O(|\epsilon|)$; uniform convergence in \eqref{eq:scalar-mean-exponent} then gives \eqref{eq:centre-derivative}.

Since $\sum_i x_i=0$,
\begin{equation}\label{eq:mass-divided-value}
 \frac{\mathfrak m_{2+\epsilon}(x)-S(x)}{\epsilon}
 =\sum_iF_\epsilon(x_i+c_\epsilon)
 +\frac{n c_\epsilon^2}{\epsilon},
\end{equation}
and the right side converges uniformly to $B(x)$. For $h\in H$, the envelope identity \eqref{eq:mp-envelope} gives
\begin{align}
 \frac{D\mathfrak m_{2+\epsilon}(x)[h] -D\mathfrak m_2(x)[h]}{\epsilon}
 =\sum_i\left[x_i+(2+\epsilon)\gamma_\epsilon(x_i+c_\epsilon)
 \right]h_i.
 \label{eq:mass-divided-differential}
\end{align}
The term involving the common translation vanishes because $(2+\epsilon)c_\epsilon\sum_i h_i/\epsilon=0$. Uniform convergence of $\gamma_\epsilon$, together with $c_\epsilon\to0$, proves convergence of the differentials in operator norm and hence \eqref{eq:mass-C1-expansion}.
\end{proof}

\begin{theorem}[The $C^1$ first-order logarithmic limit at $p=2$]
\label{thm:log-blowup}
For $0\ne x\in\one^\perp$, define
\begin{equation}\label{eq:log-functional}
 \mathcal H_n([x])=
 \frac{\displaystyle
 \sum_{i<j}(x_i-x_j)^2\log|x_i-x_j|
 -n\sum_i x_i^2\log|x_i|}
 {\displaystyle\sum_i x_i^2},
\end{equation}
where $0^2\log0=0$. For $x\in\one^\perp$, the identity $\sum_{i<j}(x_i-x_j)^2=n\sum_i x_i^2$ shows that the logarithmic terms introduced by rescaling $x$ cancel. Hence $\mathcal H_n([x])$ is well defined on $\mathbb P(\one^\perp)$. On this projective space,
\begin{equation}\label{eq:p2-expansion}
 \frac{\mathcal R_{p,\one}-n}{p-2}
 \rightarrow\mathcal H_n
 \quad\hbox{in }C^1
 \qquad(p\to2).
\end{equation}
At the eigenline represented by a three-level vector with positive block size $k$ and negative block size $\ell$,
\begin{equation}\label{eq:H-value}
 \mathcal H_n=(k+\ell)\log(k+\ell)-k\log k-\ell\log\ell.
\end{equation}
\end{theorem}

\begin{proof}
Work first on the Euclidean unit sphere in $H$. Set
$$
 E_p(x)=\sum_{i<j}|x_i-x_j|^p,\qquad
 A(x)=\sum_{i<j}(x_i-x_j)^2\log|x_i-x_j|
$$
and, with $\epsilon=p-2$, define
$$
 A_\epsilon(x)=\frac{E_{2+\epsilon}(x)-E_2(x)}{\epsilon},
 \qquad
 B_\epsilon(x)=
 \frac{\mathfrak m_{2+\epsilon}(x)-S(x)}{\epsilon}.
$$
Lemma~\ref{lem:scalar-log}, applied to the finitely many difference forms, gives $A_\epsilon\to A$ in $C^1(\Sph(H))$, while Lemma~\ref{lem:centre-derivative} gives $B_\epsilon\to B$ there. Since $E_2=nS$, the exact identity
\begin{equation}\label{eq:exact-log-quotient}
 \frac{\mathcal R_{2+\epsilon,\one}-n}{\epsilon}
 =
 \frac{A_\epsilon-nB_\epsilon}{S+\epsilon B_\epsilon}
\end{equation}
holds. Its denominator is uniformly bounded away from zero on $\Sph(H)$, and hence the right side converges in $C^1$ to $(A-nB)/S=\mathcal H_n$. All the functions involved are antipodally invariant and homogeneous of degree zero, so the convergence descends to projective space. Substitution of the values $\ell,-k,0$ on blocks of sizes $k,\ell,r$ gives \eqref{eq:H-value}; all terms involving $r$ cancel.
\end{proof}

\subsection{Critical set and local form of the logarithmic model}
\label{subsec:logarithmic-model}

We first determine the critical set of the limiting functional and then its local critical groups. Because the logarithmic terms are only $C^1$ at vanishing coordinate and edge forms, the local analysis requires a substitute for the ordinary Morse lemma.

\begin{theorem}[Critical points of the logarithmic model]
\label{thm:H-critical-set}
The critical points of $\mathcal H_n$ are exactly the cell barycentres $b_C$, $C\in\Cn$. At the barycentre with block sizes $(k,\ell,r)$,
\begin{equation}\label{eq:H-cell-value}
 \mathcal H_n(b_C)
 =h(C)
 =(k+\ell)\log(k+\ell)-k\log k-\ell\log\ell.
\end{equation}
Consequently $\mathcal H_n$ has exactly $R_n-1$ critical points. Moreover, $\mathcal H_n(b_C)<\mathcal H_n(b_D)$ whenever $D$ is a proper coface of $C$.
\end{theorem}

\begin{proof}
Let $[x]\in\Pj(H)$ be critical, choose a representative $x\in H\setminus\{0\}$, and set $h=\mathcal H_n([x])$. Define
$$
 A_i=\sum_{j\ne i}(x_i-x_j)\log|x_i-x_j|.
$$
The derivative of $u^2\log|u|$ is $2u\log|u|+u$. In the derivative of the numerator of \eqref{eq:log-functional}, the extra linear edge terms sum to $n x_i$ because $\sum_jx_j=0$, and cancel the corresponding extra mass term. Thus the derivative of the numerator has $i$th component
$
 2\bigl(A_i-nx_i\log|x_i|\bigr).
$
Let $\mathcal N$ denote the numerator of \eqref{eq:log-functional} and set $S(x)=\sum_i x_i^2$. Projective criticality gives $d(\mathcal N-hS)_x=0$ on $x^\perp\cap H$. Since the restriction of $\mathcal N-hS$ to $H$ is homogeneous of degree two and vanishes at $x$, its radial derivative also vanishes. Hence $d(\mathcal N-hS)_x$ annihilates $H$, so there is a constant $\gamma$ such that
\begin{equation}\label{eq:H-critical-equation}
 A_i=h x_i+n x_i\log|x_i|+\gamma.
\end{equation}
Summing in $i$ and using $\sum_iA_i=\sum_i x_i=0$ gives
\begin{equation}\label{eq:H-multiplier}
 \gamma=-\sum_i x_i\log|x_i|.
\end{equation}

Write the negative coordinates as $-u_i$, the positive coordinates as $v_j$, and set
$$
 L_-=\sum_i u_i\log u_i,\qquad
 L_+=\sum_jv_j\log v_j.
$$
Then
$$
 \gamma=L_--L_+,\qquad \sum_i u_i=\sum_jv_j.
$$
Suppose that the negative magnitudes are not all equal. Set
$$
 U=\max_i u_i,\qquad u=\min_i u_i,\qquad
 \delta=u^{-1}-U^{-1}>0
$$
and
\begin{align*}
 G_U&=\sum_i\left(1-\frac{u_i}{U}\right)\log\left(1-\frac{u_i}{U}\right)
 +\sum_j\left(1+\frac{v_j}{U}\right)\log\left(1+\frac{v_j}{U}\right),\\
 G_u&=-\sum_i\left(\frac{u_i}{u}-1\right)\log\left(\frac{u_i}{u}-1\right)
 +\sum_j\left(1+\frac{v_j}{u}\right)\log\left(1+\frac{v_j}{u}\right),
\end{align*}
with $0\log0=0$. Dividing \eqref{eq:H-critical-equation} by the coordinate values at $x_{i_U}=-U$ and $x_{i_u}=-u$ gives
$$
 \frac{A_{i_U}}{-U}=n\log U+G_U,\qquad
 \frac{A_{i_u}}{-u}=n\log u+G_u,
$$
and hence
\begin{equation}\label{eq:H-extreme-equality}
 G_U=h-\frac \gamma U,\qquad G_u=h-\frac \gamma u,\qquad
 G_u-G_U=(L_+-L_-)\delta.
\end{equation}

For $\varphi(s)=s\log s$, strict superadditivity gives, for every positive $v_j$,
$$
 \varphi\left(1+\frac{v_j}{u}\right)
 -\varphi\left(1+\frac{v_j}{U}\right)
 >\varphi(v_j\delta),
$$
whereas superadditivity gives
$$
 \varphi\left(\frac{u_i}{u}-1\right)
 +\varphi\left(1-\frac{u_i}{U}\right)
 \leq\varphi(u_i\delta).
$$
After summation,
\begin{align*}
 G_u-G_U
 &>\sum_j\varphi(v_j\delta)-\sum_i\varphi(u_i\delta)\\
 &=\delta(L_+-L_-)+\delta\log\delta\left(\sum_jv_j-\sum_i u_i\right)
 =\delta(L_+-L_-),
\end{align*}
contradicting \eqref{eq:H-extreme-equality}. All negative magnitudes are therefore equal. Applying the argument to $-x$ shows that all positive magnitudes are equal. If a sign block is a singleton, the corresponding conclusion is automatic, so the argument covers every block cardinality.

Conversely, let $x^0$ take the values $\ell,-k,0$ on $P,N,Z$. A direct calculation gives
$$
 A_i=
 \begin{cases}
 \ell\bigl((k+\ell)\log(k+\ell)+r\log\ell\bigr),&i\in P,\\
 -k\bigl((k+\ell)\log(k+\ell)+r\log k\bigr),&i\in N,\\
 k\ell(\log k-\log\ell),&i\in Z,
 \end{cases}
$$
while
$$
 \gamma=k\ell(\log k-\log\ell),\qquad
 h=(k+\ell)\log(k+\ell)-k\log k-\ell\log\ell.
$$
Thus \eqref{eq:H-critical-equation} holds, proving the converse and the value formula. The cell count is $M_n^+(1)=R_n-1$, and strict face ordering is Proposition~\ref{prop:face-order}.
\end{proof}

At a point $x^0$ with values $\ell,-k,0$ on $P,N,Z$, set $s=k+\ell$. When $r\geq1$, use the mean-zero projective chart
\begin{align}
 x_i&=\ell+\xi_i-\frac rs t &&(i\in P),\notag\\
 x_j&=-k+\eta_j-\frac rs t &&(j\in N),\notag\\
 x_a&=t+\zeta_a &&(a\in Z),
 \label{eq:log-chart}
\end{align}
where $\sum_P\xi_i=\sum_N\eta_j=\sum_Z\zeta_a=0$. When $r=0$, omit $\zeta$ and $t$. Let
$$
 \rho^2=\|\xi\|^2+\|\eta\|^2+\|\zeta\|^2+t^2
$$
with the omitted terms understood, and define
\begin{equation}\label{eq:log-Q}
 Q_{k,\ell,r}
 =k\|\xi\|^2+\ell\|\eta\|^2
 -(k+\ell)\|\zeta\|^2-nr t^2.
\end{equation}

\begin{lemma}[Logarithmic splitting]
\label{lem:log-splitting}
As $\rho\downarrow0$ in the chart \eqref{eq:log-chart},
\begin{equation}\label{eq:log-splitting}
 \mathcal H_n([x])-\mathcal H_n([x^0])
 =\frac{\log\rho}{k\ell(k+\ell)}Q_{k,\ell,r}
 +R(\xi,\eta,\zeta,t),
\end{equation}
where, on a sufficiently small chart,
\begin{equation}\label{eq:log-remainder}
 |R|\leq C\rho^2,\qquad \|DR\|\leq C\rho.
\end{equation}
\end{lemma}

\begin{proof}
The denominator $S(x)=\sum_i x_i^2$ is exactly
\begin{equation}\label{eq:log-chart-denominator}
 k\ell(k+\ell)+\|\xi\|^2+\|\eta\|^2+\|\zeta\|^2
 +\frac{nr}{k+\ell}t^2,
\end{equation}
so it has no linear term. For a linear form $L$ vanishing at the base point,
$$
 L^2\log|L|
 =L^2\log\rho+
 L^2\log\frac{|L|}{\rho}.
$$
The second term is $O(\rho^2)$ with derivative $O(\rho)$. Indeed, $z\log|z|$ is bounded on bounded intervals and
$$
 D\left(L^2\log\frac{|L|}{\rho}\right)
 =\left(2L\log\frac{|L|}{\rho}+L\right)DL
 -\frac{L^2}{\rho^2}u,
$$
where $u=(\xi,\eta,\zeta,t)$.

The linear forms vanishing at $x^0$ are precisely the differences within each constant block and the coordinate forms $x_a$ with $a\in Z$. The coefficient of $\log\rho$ in the numerator of \eqref{eq:log-functional} is therefore
\begin{align*}
 \sum_{i<i'\in P}(\xi_i-\xi_{i'})^2
 +\sum_{j<j'\in N}(\eta_j-\eta_{j'})^2
 +\sum_{a<a'\in Z}(\zeta_a-\zeta_{a'})^2
 -n\sum_{a\in Z}(t+\zeta_a)^2
 =Q_{k,\ell,r},
\end{align*}
where we used $\sum_{i<i'\in P}(\xi_i-\xi_{i'})^2=k\|\xi\|^2$ and its analogues. Every other logarithmic term is smooth. Its total linear term cancels because $[x^0]$ is critical. Dividing by $S(x^0)=k\ell(k+\ell)$ and using \eqref{eq:log-chart-denominator} proves both \eqref{eq:log-splitting} and \eqref{eq:log-remainder}.
\end{proof}

The expansion has a nondegenerate quadratic form multiplied by $\log\rho$, but it is not a smooth Morse normal form. Critical groups and local degree provide the appropriate invariants. For an isolated critical point $z$ of a $C^1$ function $f$ and a sufficiently small isolating neighbourhood $U$, define the local critical groups
$$
 C_j(f,z;\mathbb F)
 =H_j\bigl(f^{\leq f(z)}\cap U,\,
 (f^{\leq f(z)}\cap U)\setminus\{z\};\mathbb F\bigr).
$$
By excision, these groups are independent of $U$.
For an isolated critical point $z$, $\deg(\nabla f,z)$ denotes the local Brouwer degree of the gradient in any local chart centred at $z$.

\begin{lemma}[Local continuation of critical groups]
\label{lem:local-critical-group-continuation}
Let $U\subset\mathbb R^d$ be a bounded open neighbourhood of the origin, and suppose that each $G_\tau$, $0\leq\tau\leq1$, is $C^1$ on an open neighbourhood of $\overline U$. Assume that
$$
 [0,1]\rightarrow C^1(\overline U),\qquad
 \tau\mapsto G_\tau,
$$
is continuous, that $G_\tau(0)=0$, and that
$$
 \operatorname{Crit}(G_\tau)\cap\overline U=\{0\}
 \qquad(0\leq\tau\leq1).
$$
Then, for every coefficient field $\mathbb F$,
$$
 C_j(G_0,0;\mathbb F)\cong C_j(G_1,0;\mathbb F)
 \qquad(j\geq0).
$$
Moreover,
$
 \deg(\nabla G_0,U,0)=\deg(\nabla G_1,U,0).
$
\end{lemma}

\begin{proof}
Choose an open neighbourhood $U_0$ of the origin with $\overline U_0\Subset U$, and a $C^1$ function $\chi:\mathbb R^d\to[0,1]$ supported in $U$ and equal to one near $\overline U_0$. Extending $\chi G_\tau$ by zero and setting
$$
 \widetilde G_\tau(x)
 =\chi(x)G_\tau(x)+(1-\chi(x))\|x\|^2
$$
yields a coercive family, continuous in $C^1_{\mathrm{loc}}(\mathbb R^d)$. The functions $\widetilde G_\tau$ and $G_\tau$ agree on a neighbourhood of $\overline U_0$. Since $\operatorname{Crit}(G_\tau)\cap\overline U=\{0\}$, the origin is the only critical point of $\widetilde G_\tau$ in $\overline U_0$ for every $\tau$; the two families therefore have the same local critical groups at the origin. In finite dimensions the coercive family $\widetilde G_\tau$ satisfies the Palais--Smale condition. The continuation theorem \cite[Theorem~1.1.4]{PereraSchechter2013} therefore gives the critical-group isomorphisms. The degree identity follows directly from homotopy invariance, since $\nabla G_\tau$ has no zero on $\partial U$.
\end{proof}

\begin{lemma}[Finite-dimensional logarithmic Morse lemma]
\label{lem:log-Morse}
Suppose that $f$ and $R$ are $C^1$ near the origin and that, in a Euclidean chart, for $0<\|u\|$ sufficiently small,
$$
 f(u)=f(0)+\log\|u\|\,Q(u)+R(u),
$$
where $Q$ is a nondegenerate quadratic form and
$$
 |R(u)|\leq C\|u\|^2,\qquad
 \|DR(u)\|\leq C\|u\|.
$$
Then the origin is an isolated critical point, and its critical groups are those of $-Q$. If $m=\indm(-Q)$, then, for all sufficiently small $u\ne0$,
\begin{equation}\label{eq:log-gradient-degree}
 \|\nabla f(u)\|
 \geq c\|u\|\,|\log\|u\||,
 \qquad
 \deg(\nabla f,B_\rho,0)=(-1)^m
\end{equation}
for every sufficiently small $\rho>0$.
\end{lemma}

\begin{proof}
Set $r=\|u\|$ and $L=-\log r$. For $r<e^{-1/2}$ make the radial change of variables
$
 v=\sqrt L\,u.
$
The function $r\mapsto r\sqrt{-\log r}$ is strictly increasing there, so this is a homeomorphism of neighbourhoods of the origin and a smooth diffeomorphism away from the origin. Since $Q(v)=LQ(u)$, the transformed function has the form
$$
 f(u(v))-f(0)=-Q(v)+\widetilde R(v),
 \qquad \widetilde R(v)=R(u(v)).
$$
Moreover,
$$
 \widetilde R(v)=o(\|v\|^2),\qquad
 D\widetilde R(v)=o(\|v\|).
$$
Indeed, $\|v\|=r\sqrt L$ and $\|D_vu\|=O(L^{-1/2})$; more explicitly, $|\widetilde R(v)|=O(\|v\|^2/L)$ and $\|D\widetilde R(v)\|=O(\|v\|/L)$. Thus $\widetilde R$ extends as a $C^1$ function at the origin with $D\widetilde R(0)=0$.

Write $Q(v)=\langle Av,v\rangle$, where $A$ is invertible, and set
$$
 G_\tau(v)=-Q(v)+\tau\widetilde R(v),\qquad0\leq\tau\leq1.
$$
After shrinking the ball, uniformly in $\tau$ and $v\ne0$,
$$
 \|DG_\tau(v)\|
 \geq2\|Av\|-\|D\widetilde R(v)\|
 \geq c\|v\|.
$$
Hence the origin is the only critical point of every $G_\tau$ in the closed ball. Lemma~\ref{lem:local-critical-group-continuation} gives
$$
 C_*(G_1,0;\mathbb F)\cong C_*(-Q,0;\mathbb F).
$$
Since the radial map is a diffeomorphism off the origin, the chain rule identifies the nonzero critical points of $f-f(0)$ with those of $G_1$. The gradient estimate therefore also proves that the origin is isolated. The radial homeomorphism identifies the local sublevel pair of $f-f(0)$ with that of $G_1$, so its critical groups are those of $-Q$.

For the quantitative gradient estimate, return to the original $u$-coordinates. Direct differentiation gives
$$
 \nabla\bigl(\log r\,Q(u)\bigr)
 =2\log r\,Au+Q(u)\frac{u}{r^2},
 \qquad r=\|u\|.
$$
Hence, for small $r>0$,
$$
 \|\nabla f(u)\|
 \geq2\sigma_{\min}(A)r|\log r|-\|A\|r-Cr
 \geq c r|\log r|.
$$
On $\partial B_\rho$, the straight-line homotopy from $\nabla f$ to $2\log\rho\,Au$ has no zero when $\rho$ is small. If $d$ is the dimension of the chart, then
$$
 \deg(\nabla f,B_\rho,0)
 =\sgn\det(2\log\rho\,A)
 =(-1)^d\sgn\det A
 =(-1)^{\indm(-Q)}.
$$
\end{proof}

\begin{corollary}[Critical groups and local degrees of the logarithmic model]
\label{cor:log-indices}
Every cell barycentre is an isolated critical point of $\mathcal H_n$. For a cell with $r$ zero coordinates and every coefficient field $\mathbb F$,
\begin{equation}\label{eq:log-critical-groups}
 C_j(\mathcal H_n,b_C;\mathbb F)
 \cong
 \begin{cases}
 \mathbb F,&j=n-r-2=\dim C,\\
 0,&j\ne n-r-2.
 \end{cases}
\end{equation}
The quadratic form $Q_{k,\ell,r}$ in \eqref{eq:log-Q} is nondegenerate. Consequently, the polynomial obtained by summing the local critical-group ranks is
$
 M_n^{\log}(t)=M_n^+(t).
$
Moreover, in any local chart,
\begin{equation}\label{eq:H-local-degree}
 \deg(\nabla\mathcal H_n,b_C)=(-1)^{\dim C}.
\end{equation}
\end{corollary}

\begin{proof}
For $n=2$, the reduced space is a point, so the unique logarithmic critical group is $\mathbb F$ in degree zero and its zero-dimensional local degree is $1=(-1)^0$; the assertion is immediate. Assume henceforth that $n\geq3$. The positive subspace of $Q_{k,\ell,r}$ has dimension $(k-1)+(\ell-1)=n-r-2$, and its negative subspace has dimension $r$. Thus $-Q_{k,\ell,r}$ has Morse index $n-r-2$. Apply Lemmas~\ref{lem:log-splitting} and~\ref{lem:log-Morse}, and sum over the cells using Theorem~\ref{thm:exact-morse-polynomial}. The degree assertion is the degree conclusion of Lemma~\ref{lem:log-Morse}.
\end{proof}

The unperturbed limit describes variation of the exponent alone. To separate its critical values, we now allow the edge weights to vary on the same first-order scale.

\begin{theorem}[Simultaneous first-order variation of the exponent and edge weights]
\label{thm:tangent-family}
Fix $\eta\in\mathbb R^{\binom n2}$ and set $a(p)=\one+(p-2)\eta$. For all $p$ sufficiently close to $2$ these weights are positive, and
\begin{equation}\label{eq:tangent-family}
 \frac{\mathcal R_{p,a(p)}-n}{p-2}
 \rightarrow\mathcal H_n+\eta\cdot\Psi_2
 \quad\hbox{in }C^1(\mathbb RP^{n-2}).
\end{equation}
\end{theorem}

\begin{proof}
Set $\epsilon=p-2$. For $\epsilon\ne0$, let
$$
 F_\epsilon(u)=\frac{|u|^{2+\epsilon}-u^2}{\epsilon},
 \qquad
 B_\epsilon(x)=\frac{\mathfrak m_{2+\epsilon}(x)-S(x)}{\epsilon}.
$$
On $\Sph(H)$, Lemmas~\ref{lem:scalar-log} and~\ref{lem:centre-derivative} give the exact representation
$$
 \Psi_{2+\epsilon}([x])_{ij}
 =\frac{(x_i-x_j)^2+\epsilon
 F_\epsilon(x_i-x_j)}
 {S(x)+\epsilon B_\epsilon(x)}.
$$
It follows that $\Psi_p\to\Psi_2$ in $C^1$. Equation \eqref{eq:height-identity} now gives
$$
 \frac{\mathcal R_{p,a(p)}-n}{p-2}
 =
 \frac{\mathcal R_{p,\one}-n}{p-2}
 +\eta\cdot\Psi_p,
$$
and Theorem~\ref{thm:log-blowup} proves the claim.
\end{proof}

To separate critical values in this $C^1$ setting, we use the following general definable Maxwell-set theorem.

\begin{theorem}[Definable critical-value Maxwell sets for $C^1$ families]
\label{thm:C1-Maxwell}
Let $S$ be a definable $C^1$ manifold and let $M$ be a compact definable $C^1$ manifold. Suppose
$$
 h:S\times M\rightarrow\mathbb R,
 \qquad
 \Psi:S\times M\rightarrow\mathbb R^N
$$
are definable and continuous. They are $C^1$ in the $M$-variable, and their differentials in that variable depend continuously and definably on $(s,x)$. Assume that every $\Psi_s$ is injective. Set
$$
 f_{s,c}(x)=h_s(x)+c\cdot\Psi_s(x),
 \qquad (s,c)\in S\times\mathbb R^N.
$$
Let $\mathfrak C(h,\Psi)$ be the collision set
$$
 \left\{
 (s,c):
 \begin{array}{l}
 \exists x\ne y\text{ in }M,\quad
 d_Mf_{s,c}(x)=d_Mf_{s,c}(y)=0,\\
 f_{s,c}(x)=f_{s,c}(y)
 \end{array}
 \right\},
$$
and define the critical-value Maxwell set as the relative closure of this collision set,
$$
 \mathfrak M(h,\Psi)
 =\overline{\mathfrak C(h,\Psi)}^{\,S\times\mathbb R^N}.
$$
Then $\mathfrak M(h,\Psi)$ is closed and definable and
$$
 \dim\mathfrak M(h,\Psi)\leq\dim S+N-1.
$$
For every $(s,c)\notin\mathfrak M(h,\Psi)$, the critical set of $f_{s,c}$ is finite and distinct critical points have distinct critical values.
\end{theorem}

\begin{proof}
The collision set $\mathfrak C(h,\Psi)$ is definable. If it had dimension $\dim S+N$, it would contain a non-empty relatively open definable cell $U\subset S\times\mathbb R^N$. Apply definable choice to the incidence set defining $\mathfrak C(h,\Psi)$ and refine $U$ by $C^1$ cell decomposition. After shrinking $U$, there are $C^1$ selections
$$
 x,y:U\rightarrow M,\qquad x(s,c)\ne y(s,c),
$$
lying in fixed coordinate charts and satisfying the two critical-point equations and the equal-value equation. Write the common value as $\lambda(s,c)$. For every coefficient direction $v\in\mathbb R^N$, the chain rule at fixed $s$, together with criticality, gives
$$
 D_c\lambda(s,c)[v]
 =d_Mf_{s,c}(x(s,c))\,D_cx(s,c)[v]
 +v\cdot\Psi_s(x(s,c))
 =v\cdot\Psi_s(x(s,c)),
$$
and the same calculation at $y(s,c)$ gives $D_c\lambda(s,c)[v]=v\cdot\Psi_s(y(s,c))$. Thus $\Psi_s(x(s,c))=\Psi_s(y(s,c))$, contradicting injectivity. Hence
$
 \dim\mathfrak C(h,\Psi)\leq\dim S+N-1.
$
Definable closure preserves dimension, proving the same bound for $\mathfrak M(h,\Psi)$.

Take $(s,c)$ outside this closure. Equal values at two distinct critical points are impossible. If $\Crit(f_{s,c})$ were infinite, it would contain a positive-dimensional definable $C^1$ cell and hence a nonconstant definable $C^1$ path. The chain rule makes $f_{s,c}$ constant along that path, producing a collision. This contradiction proves finiteness.
\end{proof}

\begin{theorem}[Critical-value separation near $p=2$]
\label{thm:tangent-Maxwell}
Set $m=\binom n2$ and
$$
 \mathcal H_{n,\eta}=\mathcal H_n+\eta\cdot\Psi_2.
$$
There exist an open convex ball $W\subset\mathbb R^m$ about the origin and a relatively closed definable set
$$
 \mathfrak M_n^{\mathrm{tan}}\subset W,
 \qquad
 \dim\mathfrak M_n^{\mathrm{tan}}\leq m-1,
$$
such that:
\begin{enumerate}[label=\textup{(\roman*)}]
\item for every $\eta\in W$, the functional $\mathcal H_{n,\eta}$ has a critical point in a fixed neighbourhood $U_C$ of each $b_C$, where $U_C$ is chosen to isolate $b_C$ as a critical point of $\mathcal H_n$;
\item for $\eta\in W\setminus\mathfrak M_n^{\mathrm{tan}}$, its critical set is finite and all critical values are pairwise distinct, so it has at least $R_n-1$ distinct critical values;
\item for every $\eta\in W\setminus\mathfrak M_n^{\mathrm{tan}}$, there is $\delta>0$ such that the complete graph with weights
$$
 a(p)=\one+(p-2)\eta
$$
has at least $R_n$ distinct spectral values whenever $0<|p-2|<\delta$.
\end{enumerate}
In particular, $\eta$ may be chosen rational and arbitrarily small.
\end{theorem}

\begin{proof}
Choose pairwise disjoint isolating neighbourhoods $U_C$ of the barycentres $b_C$. Equation~\eqref{eq:H-local-degree} gives
$$
 \deg(\nabla\mathcal H_n,b_C)=(-1)^{\dim C}\ne0.
$$
The gradient is nonzero on the finite union of their boundaries. Shrink an open convex ball $W$ about zero so that, for every $\eta\in W$, the homotopy
$$
 \mathcal H_n+t\eta\cdot\Psi_2,\qquad0\leq t\leq1,
$$
has no boundary critical point. Homotopy invariance of degree proves \textup{(i)}.

Apply Theorem~\ref{thm:C1-Maxwell} with $S$ a point, $h=\mathcal H_n$, and $\Psi=\Psi_2$. The map $\Psi_2$ is injective, so intersecting the resulting closed Maxwell set with $W$ gives $\mathfrak M_n^{\mathrm{tan}}$ and proves \textup{(ii)}.

Fix $\eta$ outside this set. Its critical set is finite, so every critical point is isolated. In each $U_C$, additivity of degree gives a critical point $y_C$ of nonzero local degree. Choose pairwise disjoint isolating neighbourhoods $V_C\Subset U_C$ around these points. If there is more than one cell, set
$$
 g=\min_{C\ne C'}
 \bigl|\mathcal H_{n,\eta}(y_C)-\mathcal H_{n,\eta}(y_{C'})\bigr|>0;
$$
for the zero-dimensional $n=2$ case set $g=1$. Shrink the $V_C$ so that
$
 \operatorname{osc}_{V_C}\mathcal H_{n,\eta}<g/8.
$

Theorem~\ref{thm:tangent-family} gives
$$
 F_p:=
 \frac{\mathcal R_{p,\one+(p-2)\eta}-n}{p-2}
 \rightarrow\mathcal H_{n,\eta}
 \quad\hbox{in }C^1.
$$
For $p$ sufficiently close to $2$, the local degrees on every $V_C$ persist and $\|F_p-\mathcal H_{n,\eta}\|_{C^0}<g/8$. Each $V_C$ therefore contains a critical point of $F_p$, and values chosen in distinct $V_C$'s lie in disjoint intervals of radius $g/4$. Since $p\ne2$, $F_p$ and the reduced quotient differ by a nonzero affine rescaling and have the same critical points and the same value collisions. Thus the reduced quotient has at least $R_n-1$ distinct nonconstant critical values. The weights are positive for $|p-2|$ small, and the constant eigenline gives the distinct value zero. This proves \textup{(iii)}.

The complement in $W$ of a lower-dimensional relatively closed definable set is open and dense, and contains rational points arbitrarily close to the origin.
\end{proof}

\section{The subquadratic critical-group theorem}
\label{sec:subquadratic}

For fixed $1<p<2$, the quotient fails to be $C^2$ at the relevant vanishing coordinate and edge differences. We therefore replace the Hessian calculation by an anisotropic first-order model from which the local critical groups and degrees can be read. Set $q=p-1$. At an eigenline of the uniformly weighted $K_n$ indexed by $[n]=P\sqcup N\sqcup Z$, with block sizes $(k,\ell,r)$, set
\begin{equation}\label{eq:low-constants}
 \alpha=\ell^{1/q},\qquad
 \beta=k^{1/q},\qquad
 D=\alpha+\beta,\qquad
 \lambda_0=D^q+r,\qquad
 m_0=k\ell D.
\end{equation}
Let $x^0=\alpha\one_P-\beta\one_N$.
Let
$$
 X_P=\left\{\xi\in\mathbb R^P:\sum_P\xi_i=0\right\},
$$
and define $X_N,X_Z$ analogously. For $r\geq1$, the following coordinates give a local chart on $\Pj(V_n)$ at the eigenline:
\begin{equation}\label{eq:low-chart}
 x_i=\alpha+\xi_i,\quad
 x_j=-\beta+\eta_j,\quad
 x_a=t+\zeta_a,
 \qquad
 (\xi,\eta,\zeta,t)\in
 X_P\oplus X_N\oplus X_Z\oplus\mathbb R.
\end{equation}
For $r=0$, omit $(\zeta,t)$. Projective rescaling fixes the difference between the $P$- and $N$-block means at $D$; translation then fixes those means at $\alpha$ and $-\beta$. Set
\begin{align*}
 W_P(\xi)=\sum_{\{i,i'\}\subset P}|\xi_i-\xi_{i'}|^p,\qquad
 W_N(\eta)=\sum_{\{j,j'\}\subset N}|\eta_j-\eta_{j'}|^p,
\end{align*}
and, on $X_Z$,
$$
 E_{r,p}(\zeta)=
 \sum_{\{a,a'\}\subset Z}|\zeta_a-\zeta_{a'}|^p,\qquad
 \mathfrak m_{r,p}(\bar\zeta)=
 \min_c\sum_{a\in Z}|\zeta_a+c|^p.
$$
For $r\geq1$, let $c_{r,p}(\zeta)$ denote the unique minimiser in the definition of $\mathfrak m_{r,p}(\bar\zeta)$.
When $r=0$ or $1$, the two zero-block functionals vanish (with the zero-dimensional variables omitted). The anisotropic dilation and corresponding homogeneous gauge are given by
\begin{align}
 \delta_s(\xi,\eta,\zeta,t)
 &=\bigl(s^{1/p}\xi,s^{1/p}\eta,s^{1/p}\zeta,s^{1/2}t\bigr),\label{eq:low-dilation}\\
 \varrho(\xi,\eta,\zeta,t)
 &=\|\xi\|^p+\|\eta\|^p+\|\zeta\|^p+t^2,
 \qquad \varrho(\delta_su)=s\varrho(u).
 \label{eq:low-gauge}
\end{align}
The $s^{1/p}$-scale resolves variations within the three coordinate blocks, whereas the $s^{1/2}$-scale captures the displacement of the zero block relative to the other two. For $s>0$, write $a_s=s^{1/p}$ and $b_s=s^{1/2}$.
\begin{lemma}[Scaled $p$-centre]
\label{lem:low-centre}
For $r\geq1$, write the optimising translation in the form
\begin{equation}\label{eq:scaled-centre}
 c_p(x^0+\delta_su)=-b_s t+a_s d_s(u).
\end{equation}
Then, uniformly for $u$ in compact subsets of the chart,
\begin{equation}\label{eq:scaled-centre-limit}
 d_s(u)\rightarrow c_{r,p}(\zeta),
 \qquad
 \sum_{a\in Z}\Phi_p\bigl(\zeta_a+c_{r,p}(\zeta)\bigr)=0.
\end{equation}
\end{lemma}

\begin{proof}
The centred coordinates are
\begin{align*}
 y_i&=\alpha+a_s(\xi_i+d_s)-b_s t &&(i\in P),\\
 y_j&=-\beta+a_s(\eta_j+d_s)-b_s t &&(j\in N),\\
 y_a&=a_s(\zeta_a+d_s) &&(a\in Z).
\end{align*}
For a trial value $d$, define
\begin{align}
 \Gamma_s(u,d)={}&a_s^{-q}\left[
 \sum_P\Phi_p\bigl(\alpha+a_s(\xi_i+d)-b_s t\bigr)
 +\sum_N\Phi_p\bigl(-\beta+a_s(\eta_j+d)-b_s t\bigr)\right]
 \\
 &+\sum_Z\Phi_p(\zeta_a+d).
 \label{eq:scaled-centring-equation}
\end{align}
The centring equation is $\Gamma_s(u,d_s(u))=0$. Set $A_0=k\alpha^{p-2}+\ell\beta^{p-2}$. Taylor expansion at the two nonzero levels, the zero-sum conditions, and $k\alpha^q=\ell\beta^q=k\ell$ give, uniformly for $u,d$ in compact sets,
$$
 a_s^{-q}[\cdots]
 =qA_0a_s^{-q}(a_s d-b_s t)
 +O\bigl(a_s^{-q}(a_s+b_s)^2\bigr)=o(1).
$$
Indeed, the relevant powers of $s$ are
$$
 a_s^{1-q}=s^{(2-p)/p},\qquad
 b_s a_s^{-q}=s^{(2-p)/(2p)},\qquad
 b_s^2a_s^{-q}=s^{1/p},
$$
all of which tend to zero.

Fix a compact set $K$ of chart variables and choose $M>1+\sup_{u\in K,\,z\in Z}|\zeta_z|$. At $d=M$ the $Z$-block contribution in \eqref{eq:scaled-centring-equation} is uniformly positive, and at $d=-M$ it is uniformly negative, whereas the $P\cup N$ contribution is $o(1)$ at both values. Strict monotonicity therefore shows that the unique root lies in $(-M,M)$. If $s_j\downarrow0$ and $u_j\to u$ in $K$, every subsequential limit of $d_{s_j}(u_j)$ solves the limiting centring equation at $u$. Uniqueness of its root gives uniform convergence on $K$. For $r=1$ the space $X_Z$ is zero-dimensional and the limit is zero.
\end{proof}

Set
\begin{equation}\label{eq:low-C0}
 C_0=\frac{p(p-1)}2
 \bigl(k\alpha^{p-2}+\ell\beta^{p-2}\bigr).
\end{equation}

\begin{lemma}[Scaled mass and energy expansions]
\label{lem:low-scaled-expansions}
For every compact set $K$ in the chart, as $s\downarrow0$,
\begin{align}
 \frac{\mathfrak m_p(\overline{x^0+\delta_su})-m_0}{s}
 &\rightarrow
 \mathfrak m_{r,p}(\bar\zeta)+C_0t^2,
 \label{eq:low-mass-expansion}\\
 \frac{E_{p,\one}(x^0+\delta_su)-E_{p,\one}(x^0)}{s}
 &\rightarrow
 W_P(\xi)+W_N(\eta)+E_{r,p}(\zeta)+rC_0t^2
 \label{eq:low-energy-expansion}
\end{align}
in $C^1(K)$. When $r=0$, the terms in $(\zeta,t)$ are omitted.
\end{lemma}

\begin{proof}
For convergence of the functions in \eqref{eq:low-mass-expansion}, expand the nonzero centred coordinates from the preceding proof. The linear term vanishes identically:
$$
 p\{k\alpha^q-\ell\beta^q\}(a_s d_s-b_s t)=0.
$$
Since $a_s/b_s=s^{1/p-1/2}\to0$, the quadratic term divided by $s=b_s^2$ converges to $C_0t^2$. After division by $s$, the zero-block contribution is exactly
$$
 \sum_Z|\zeta_a+d_s|^p
 \rightarrow \mathfrak m_{r,p}(\bar\zeta).
$$

For convergence of their differentials, the envelope identity eliminates derivatives of $d_s$ and gives, for $\dot u=(\dot\xi,\dot\eta,\dot\zeta,\dot t)$,
\begin{align}
 D_u\!\left[\mathfrak m_p\!\left(\overline{x^0+\delta_su}\right)
 \right][\dot u]
 =p\{&
 a_s\sum_P\Phi_p(y_i)\dot\xi_i+
 a_s\sum_N\Phi_p(y_j)\dot\eta_j+
 a_s\sum_Z\Phi_p(y_a)\dot\zeta_a\notag\\
 &+b_s\dot t\sum_Z\Phi_p(y_a)\}.
 \label{eq:low-mass-envelope}
\end{align}
After division by $s$, the first two sums tend to zero because $\sum\dot\xi=\sum\dot\eta=0$. The third tends uniformly to
$$
 p\sum_Z\Phi_p(\zeta_a+c_{r,p}(\zeta))\dot\zeta_a
 =D\mathfrak m_{r,p}(\bar\zeta)[\dot\zeta].
$$
The centring equation and Taylor expansion give
$$
 \sum_{P\cup N}\Phi_p(y_i)
 =qA_0(a_s d_s-b_s t)+O((a_s+b_s)^2)
 =-qA_0b_s t+o(b_s).
$$
Since the zero-block sum is its negative, the last term in \eqref{eq:low-mass-envelope}, divided by $s=b_s^2$, converges to $pqA_0t\dot t$, which equals $D(C_0t^2)[\dot t]$. Uniform continuity of $\Phi_p$ on bounded intervals also gives uniform convergence at points where an optimally centred zero-block coordinate vanishes.

For the energy, the within-$P$, within-$N$, and within-$Z$ terms are exactly $sW_P,sW_N,sE_{r,p}$. Every cross-block difference is nonzero at the base point and can be Taylor expanded. The $P$--$N$ linear terms vanish by the two zero-sum conditions. The $P$--$Z$ and $N$--$Z$ linear terms in $t$ are respectively $-pk\ell r b_s t$ and $+pk\ell r b_s t$, so they cancel. The remaining $t^2$ coefficient is
$$
 \frac{p(p-1)}2r
 \bigl(k\alpha^{p-2}+\ell\beta^{p-2}\bigr)=rC_0.
$$
Terms of size $a_s^2$ and $a_s b_s$ are $o(s)$ because $2/p>1$ and $1/p+1/2>1$. Differentiating with respect to the chart variables gives the same estimates and proves \eqref{eq:low-energy-expansion} in $C^1$.

When $r=0$, the minimising translation is smooth near $x^0$, its first variation vanishes on $X_P\oplus X_N$, and
$$
 \mathfrak m_p(\overline{x^0+\delta_su})
 =m_0+O(s^{2/p})=m_0+o(s),
$$
where the $O(s^{2/p})$ estimate holds in $C^1$ on compact subsets of the chart; the energy limit is $W_P+W_N$.
\end{proof}

The base identities are
\begin{equation}\label{eq:low-base-identities}
 m_0=k\alpha^p+\ell\beta^p=k\ell D,\qquad
 E_{p,\one}(x^0)=k\ell D^p+kr\alpha^p+\ell r\beta^p
 =m_0\lambda_0.
\end{equation}
Define
\begin{equation}\label{eq:low-gamma}
 \gamma=(\lambda_0-r)C_0
 =\frac{p(p-1)}2(\lambda_0-r)
 \bigl(k\alpha^{p-2}+\ell\beta^{p-2}\bigr)>0.
\end{equation}

\begin{theorem}[Anisotropic zero-block splitting]
\label{thm:low-splitting}
For every compact $K$ in the chart,
\begin{equation}\label{eq:low-scaled-splitting}
 \frac{\mathcal R_{p,\one}([x^0+\delta_su])-\lambda_0}{s}
 \rightarrow\mathcal P(u)
 \quad\hbox{in }C^1(K),
\end{equation}
where
\begin{equation}\label{eq:low-principal-form}
 \mathcal P(u)=\frac1{m_0}\left[
 W_P(\xi)+W_N(\eta)+E_{r,p}(\zeta)
 -\lambda_0\mathfrak m_{r,p}(\bar\zeta)-\gamma t^2
 \right].
\end{equation}
\end{theorem}

\begin{proof}
Combining \eqref{eq:low-mass-expansion}--\eqref{eq:low-energy-expansion} with \eqref{eq:low-base-identities} gives the asserted $C^1$ limit. The coefficient of $t^2$ is $(r-\lambda_0)C_0=-\gamma$.
\end{proof}

\begin{lemma}[The zero-block Rayleigh quotient]
\label{lem:zero-block-gap}
For $r\geq2$, set
$$
 \mu_{r,p}=\max_{\bar z\ne0}
 \frac{E_{r,p}(z)}{\mathfrak m_{r,p}(\bar z)}.
$$
Then
\begin{equation}\label{eq:zero-block-gap}
 \mu_{r,p}=r-2+2^q,
\end{equation}
and, for $\bar\zeta\ne0$,
\begin{equation}\label{eq:zero-block-coercivity}
 \lambda_0\mathfrak m_{r,p}(\bar\zeta)-E_{r,p}(\zeta)
 \geq\bigl(D^q+2-2^q\bigr)\mathfrak m_{r,p}(\bar\zeta)
 \geq2\,\mathfrak m_{r,p}(\bar\zeta).
\end{equation}
\end{lemma}

\begin{proof}
A maximiser exists on compact projective space and is a critical point. By Proposition~\ref{prop:reduced-correspondence} and Lemma~\ref{lem:uniform-level-collapse}, every nonconstant critical value on $K_r$ has the form
$$
 r-a-b+\bigl(a^{1/q}+b^{1/q}\bigr)^q,
 \qquad a,b\geq1,\quad a+b\leq r.
$$
Set
$
 \Delta_q(a,b)
 =a+b-\bigl(a^{1/q}+b^{1/q}\bigr)^q.
$
For real $a,b>0$,
$$
 \frac{\partial\Delta_q}{\partial a}
 =1-\left(1+(b/a)^{1/q}\right)^{q-1}>0,
$$
and symmetrically in $b$, since $q-1<0$. Thus the minimum of $\Delta_q$ over $a,b\geq1$ occurs at $(1,1)$ and equals $2-2^q$. This proves \eqref{eq:zero-block-gap}.

Since $\lambda_0=D^q+r$,
$
 \lambda_0-\mu_{r,p}=D^q+2-2^q.
$
Here $D=\alpha+\beta\geq2$, so $D^q\geq2^q$, and \eqref{eq:zero-block-coercivity} follows. For $r=0,1$, the two zero-block functionals vanish.
\end{proof}

The principal form is therefore a difference of two positive weighted-homogeneous functions. The following continuation principle converts this splitting into critical groups and local degree.

\begin{lemma}[Critical groups under anisotropic scaling]
\label{lem:anisotropic-Morse}
Let
$$
 \delta_s(u_1,\ldots,u_d)
 =(s^{\omega_1}u_1,\ldots,s^{\omega_d}u_d),
 \qquad0<\omega_i<1,
$$
and let $\varrho$ be a continuous proper gauge, positive away from the origin, such that $\varrho(\delta_su)=s\varrho(u)$. Suppose that $\mathcal P\in C^1(\mathbb R^d)$ satisfies
$$
 \mathcal P(\delta_su)=s\mathcal P(u),
 \qquad D\mathcal P(u)\ne0\quad(u\ne0).
$$
Let $F\in C^1$ near the origin, with $F(0)=0$ and $DF(0)=0$, and assume that
$$
 F_s:=s^{-1}F\circ\delta_s\rightarrow\mathcal P
 \quad\hbox{in }C^1
$$
on compact sets. Then the origin is an isolated critical point of $F$, and, for every coefficient field $\mathbb F$,
$$
 C_*(F,0;\mathbb F)\cong C_*(\mathcal P,0;\mathbb F),
 \qquad
 \deg(\nabla F,0)=\deg(\nabla\mathcal P,0).
$$

Suppose, more specifically, that $\mathcal P(u,v)=A(u)-B(v)$, where $A$ and $B$ are $C^1$, positive away from the origin, and homogeneous of degree one for the induced diagonal dilations on spaces of dimensions $a$ and $b$. Then
\begin{equation}\label{eq:anisotropic-groups-degree}
 C_j(F,0;\mathbb F)\cong
 \begin{cases}\mathbb F,&j=b,\\0,&j\ne b,\end{cases}
 \qquad
 \deg(\nabla F,0)=(-1)^b.
\end{equation}
\end{lemma}

\begin{proof}
Let $S=\{u:\varrho(u)=1\}$. This set is compact, and $D\mathcal P$ has no zero on it, so $\|\nabla\mathcal P\|\geq c>0$ on $S$. Set
$$
 H_{s,\tau}=(1-\tau)\mathcal P+\tau F_s.
$$
Uniform convergence of $F_s$ and $DF_s$ on the compact set $S$ gives $s_0>0$ such that
\begin{equation}\label{eq:anisotropic-sphere-gradient}
 \|\nabla H_{s,\tau}(u)\|\geq c/2
 \qquad
 (u\in S,\ 0<s\leq s_0,\ 0\leq\tau\leq1).
\end{equation}
Fix $s_0$ and consider $H_{s_0,\tau}$ on $U=\{u:\varrho(u)<1\}$. Homogeneity and the semigroup law give
$
 H_{s_0,\tau}\circ\delta_r=rH_{s_0r,\tau}.
$
If $x=\delta_ru$ with $u\in S$ and $0<r\leq1$, differentiation gives
$$
 D\delta_r^{\mathsf T}
 \nabla H_{s_0,\tau}(\delta_ru)
 =r\nabla H_{s_0r,\tau}(u).
$$
Thus a nonzero critical point of $H_{s_0,\tau}$ in $\overline U$ would contradict \eqref{eq:anisotropic-sphere-gradient}. Lemma~\ref{lem:local-critical-group-continuation} therefore gives
$
 C_*(F_{s_0},0;\mathbb F)
 \cong C_*(\mathcal P,0;\mathbb F).
$
Multiplication by the positive scalar $s_0^{-1}$ and the homeomorphism $\delta_{s_0}$ identify the local sublevel pairs of $F_{s_0}$ and $F$. This proves the critical-group assertion and also shows that the origin is isolated for $F$.

For the degree, if $A_s=D\delta_s$, then
$
 \nabla F_s(u)=s^{-1}A_s^{\mathsf T}\nabla F(\delta_su).
$
The domain map $A_s$ and the target map $s^{-1}A_s^{\mathsf T}$ both have positive determinant. Coordinate change and the preceding homotopy therefore give $\deg(\nabla F,0)=\deg(\nabla\mathcal P,0)$.

For the final assertion, let the dilation weights on the $u$-space be $\omega_1^+,\ldots,\omega_a^+$ and those on the $v$-space be $\omega_1^-,\ldots,\omega_b^-$, and define
$$
 A_0(u)=\sum_{i=1}^a|u_i|^{1/\omega_i^+},
 \qquad
 B_0(v)=\sum_{j=1}^b|v_j|^{1/\omega_j^-}.
$$
Use the straight-line homotopies from $A$ to $A_0$ and from $B$ to $B_0$. Every function in these homotopies is positive away from the origin and weighted homogeneous. The weighted Euler identity therefore excludes nonzero critical points throughout, so local continuation applies to their difference. The componentwise homeomorphism
$$
 U_i=\operatorname{sgn}(u_i)|u_i|^{1/(2\omega_i^+)},
 \qquad
 V_j=\operatorname{sgn}(v_j)|v_j|^{1/(2\omega_j^-)}
$$
identifies the local sublevel pair of $A_0-B_0$ with that of $\|U\|^2-\|V\|^2$. Hence the critical group is concentrated in degree $b$. The gradient of $A_0-B_0$ is the Cartesian product of increasing scalar maps in the $u$-variables and decreasing scalar maps in the $b$ $v$-variables, so its local degree is $(-1)^b$.
\end{proof}

\begin{theorem}[Subquadratic critical groups and local degrees]
\label{thm:low-indices}
For $1<p<2$, every nonconstant eigenline of the uniformly weighted $K_n$ is an isolated critical point of the reduced quotient. For the cell $C_{P,N}$, let $z_{P,N}=\Theta_p^{-1}([b_{P,N}])$. If its zero-block size is $r$, then, for every coefficient field $\mathbb F$,
\begin{equation}\label{eq:low-critical-groups}
\begin{gathered}
 C_j(\mathcal R_{p,\one},z_{P,N};\mathbb F)
 \cong
 \begin{cases}
 \mathbb F,&j=r=\codim C_{P,N},\\
 0,&j\ne r,
 \end{cases}\\
 \deg(\nabla\mathcal R_{p,\one},z_{P,N})=(-1)^r.
\end{gathered}
\end{equation}
Consequently,
\begin{align}
 M_n^-(t)
 =\sum_{r=0}^{n-2}\binom nr(2^{n-r-1}-1)t^r
 =\frac{(t+2)^n-2(t+1)^n+t^n}{2}
 =t^{n-2}M_n^+(t^{-1}).
 \label{eq:low-polynomial}
\end{align}
\end{theorem}

\begin{proof}
By Lemma~\ref{lem:zero-block-gap}, the limiting function $\mathcal P$ in \eqref{eq:low-principal-form} has the form
$
 \frac1{m_0}\{A(\xi,\eta)-B(\zeta,t)\},
$
where
$$
 A=W_P+W_N,\qquad
 B=\lambda_0\mathfrak m_{r,p}-E_{r,p}+\gamma t^2
$$
On $X_P\oplus X_N$, the function $A$ is positive away from the origin. For $r\geq2$, Lemma~\ref{lem:zero-block-gap} gives
$$
 B(\zeta,t)
 \geq2\,\mathfrak m_{r,p}(\bar\zeta)+\gamma t^2
$$
so $B$ is positive away from the origin on $X_Z\oplus\mathbb R$. If $r=1$, then $X_Z=\{0\}$ and $B(0,t)=\gamma t^2$; if $r=0$, regard $B$ as the zero function on the zero-dimensional domain of $B$. Thus $A$ and $B$ are positive away from the origin on their respective spaces in every case. Both are homogeneous of weighted degree one under the corresponding restrictions of \eqref{eq:low-dilation}. Their weighted Euler identities imply that their gradients have no nonzero zeros. The domain of $A$ has dimension $(k-1)+(\ell-1)=n-r-2$, and the domain of $B$ has dimension $(r-1)+1=r$. Apply Theorem~\ref{thm:low-splitting} and Lemma~\ref{lem:anisotropic-Morse}.

If $r=0$, the negative factor is absent; the eigenline is a strict local minimum and has degree $+1$. If $r=1$, $X_Z=\{0\}$ and the only term from the negative factor is $-\gamma t^2/m_0$. If $r=n-2$, then $k=\ell=1$, the positive factor is absent, and the eigenline is a strict local maximum on the $(n-2)$-dimensional reduced projective space. When $n=2$, the reduced projective space is a point, and the assertion follows directly.

Counting the cells by their zero-block size gives the first expression in \eqref{eq:low-polynomial}; the binomial theorem gives the second. Replacing $r$ by $n-2-j$ and using the face coefficients in \eqref{eq:index-count-reduced} gives the reciprocal identity.
\end{proof}

The nonzero local degrees persist under small changes of the edge weights and force a critical point to remain near each uniform-weight eigenline; neither nondegeneracy nor local uniqueness is required.

\begin{theorem}[Persistence and generic eigenvalue separation for $1<p<2$]
\label{thm:low-interval}

Let $n\geq2$, let $J\Subset(1,2)$ be a compact interval, and set $m=\binom n2$. There exists a definable open neighbourhood $V_J\subset(0,\infty)^m$ of $\one$ and a relatively closed definable set
$$
 \mathfrak M^-_{n,J}\subset J\times V_J,
 \qquad
 \dim\mathfrak M^-_{n,J}\leq m,
$$
with the following properties.
\begin{enumerate}[label=\textup{(\roman*)}]
\item Every $(p,a)\in J\times V_J$ has at least $R_n$ $p$-eigenlines, including the constant eigenline.
\item If $(p,a)\notin\mathfrak M^-_{n,J}$, the eigenline set is finite and all its eigenvalues are pairwise distinct. In particular, $\#\sigma_p(a)\geq R_n$.
\item For each fixed $p\in J$, there is a relatively closed definable set
$$
 \mathfrak M^-_{n,p}\subset V_J,\qquad
 \dim\mathfrak M^-_{n,p}\leq m-1.
$$
Outside $\mathfrak M^-_{n,p}$, the eigenline set is finite and its eigenvalues are pairwise distinct. Consequently, $V_J\setminus\mathfrak M^-_{n,p}$ is open, dense, and of full Lebesgue measure in $V_J$, and every weight in this set has at least $R_n$ distinct spectral values.
\item There are positive integer edge weights $A$ on $K_n$ such that, for every $p\in J$, the graph has at least $R_n$ eigenlines and, for all but finitely many $p\in J$, it has at least $R_n$ distinct spectral values.
\end{enumerate}
\end{theorem}

\begin{proof}
Choose an open interval $J_0\Subset(1,2)$ containing $J$. For $C=C_{P,N}\in\Cn$, let
$$
 z_C(p)=\Theta_p^{-1}(b_C)=[\,\overline{\Phi_{p'}(b_{P,N})}\,],
 \qquad p'=\frac p{p-1}.
$$
These are pairwise disjoint definable $C^1$ sections over $J_0$, and the uniform critical set is exactly their union. Using the explicit cell charts from this section, with the two nonzero block levels varying in $p$, choose pairwise disjoint definable $C^1$ charts
$$
 \chi_C:J_0\times\overline B_\rho
 \rightarrow\mathbb RP^{n-2},
 \qquad
 \chi_C(p,0)=z_C(p),
$$
for a common sufficiently small $\rho>0$. Their restrictions to $J\times\overline B_\rho$ have compact, pairwise disjoint images, and for each $p$, the image of $\chi_C(p,\cdot)$ contains no other critical point of the uniformly weighted $K_n$.

When $n=2$, use the unique zero-dimensional chart; its local degree is $+1$.

The local degree at $z_C(p)$ is $(-1)^{\codim C}$ by Theorem~\ref{thm:low-indices}. On the compact union of the chart boundaries, the norm of the differential of the quotient for the uniformly weighted $K_n$ with respect to the chart variables is continuous and strictly positive. It therefore has a positive lower bound uniformly in $p\in J$. Lemma~\ref{lem:Psi-C1-family} gives a uniform $C^1$ bound for the coefficient map $\Psi_p$. Since
$
 \mathcal R_{p,a}-\mathcal R_{p,\one}
 =(a-\one)\cdot\Psi_p,
$
one neighbourhood $V_J$ of $\one$ can be chosen so that the straight-line homotopy in the edge weights has no boundary critical point for any $p\in J$. In these chart coordinates, homotopy invariance gives
$$
 \deg\left(
 \nabla_u\bigl(\mathcal R_{p,a}\circ\chi_C(p,\cdot)\bigr),
 B_\rho,0
 \right)
 =(-1)^{\codim C}\ne0.
$$
The image of each chart therefore contains a critical point. These images are disjoint, so this gives $R_n-1$ nonconstant eigenlines; together with the constant eigenline, this gives $R_n$ eigenlines.

Apply Theorem~\ref{thm:C1-Maxwell} with $S=J_0$, $h=0$, and $\Psi(p,x)=\Psi_p(x)$, and define $\mathfrak M^-_{n,J}$ by intersecting the resulting Maxwell set with $J\times V_J$. This set is relatively closed and has dimension at most $m$. Outside it, the reduced critical set is finite and its critical values are pairwise distinct. The constant value is zero and every nonconstant value is positive, proving \textup{(ii)}.

For fixed $p$, the same theorem with $S$ a point gives
$
 \mathfrak M^-_{n,p}=\mathfrak M(0,\Psi_p)\cap V_J,
$
with the dimension bound in \textup{(iii)}.

For \textup{(iv)}, set
$
 E=\{a\in V_J:\dim(\mathfrak M^-_{n,J})_a\geq1\}.
$
Definable fibre dimension gives $\dim E+1\leq\dim\mathfrak M^-_{n,J}\leq m$. Since definable closure preserves dimension, $V_J\setminus\overline E^{\,V_J}$ is a non-empty open definable set. Choose
$
 a\in\mathbb Q^m\cap
 \bigl(V_J\setminus\overline E^{\,V_J}\bigr).
$
The fibre $(\mathfrak M^-_{n,J})_a$ has dimension at most zero and hence, by definability, is finite. Clearing denominators preserves eigenlines and equalities between eigenvalues and scales every eigenvalue by the same positive factor. The resulting graph, with positive integer edge weights, has the asserted properties.
\end{proof}
\begin{corollary}[Reciprocal Morse relations]
\label{cor:reciprocal-Morse}
Let $n\geq2$, set $d=n-2$, and let
$
 P_n(t)=\sum_{j=0}^{d}t^j
$
be the $\mathbb F_2$-Poincar\'e polynomial of $\mathbb RP^d$. There is a polynomial $Q_n^+(t)\in\mathbb Z_{\geq0}[t]$ such that
\begin{equation}\label{eq:positive-Morse-relation}
 M_n^+(t)=P_n(t)+(1+t)Q_n^+(t).
\end{equation}
For $n\geq3$, define
$
 Q_n^-(t)=t^{d-1}Q_n^+(t^{-1});
$
for $n=2$, set $Q_2^-=0$. Then
\begin{equation}\label{eq:negative-Morse-relation}
 M_n^-(t)=P_n(t)+(1+t)Q_n^-(t).
\end{equation}
Moreover,
$$
 2Q_n^\pm(1)=R_n-n,
 \qquad
 M_n^\pm(-1)=\chi(\mathbb RP^{n-2}).
$$
In particular,
$$
 M_n^-(t)=t^{n-2}M_n^+(t^{-1}),
 \qquad
 Q_n^-(t)=t^{n-3}Q_n^+(t^{-1})\quad(n\geq3).
$$
\end{corollary}

\begin{proof}
For $n=2$, one has $M_2^+=M_2^-=P_2=1$ and $Q_2^+=Q_2^-=0$, so the assertion is immediate. Assume henceforth that $n\geq3$.

The polynomial Morse inequalities over $\mathbb F_2$ for the smooth Morse function $\mathcal R_{p,\one}$, $p>2$, give \eqref{eq:positive-Morse-relation}. Use $M_n^-(t)=t^dM_n^+(t^{-1})$ and $t^dP_n(t^{-1})=P_n(t)$ to obtain
$$
 M_n^-(t)
 =P_n(t)+(1+t)t^{d-1}Q_n^+(t^{-1}),
$$
which proves \eqref{eq:negative-Morse-relation}. Evaluating at $t=1$ and $t=-1$, and using $M_n^\pm(1)=R_n-1$ and $P_n(1)=n-1$, gives the remaining assertions.
\end{proof}

Theorem~\ref{thm:low-interval} already gives generic separation. We now give a direct first-order proof along a rational ray, which yields the subquadratic exponential lower bound and prepares the explicit integer perturbation used below.

\begin{theorem}[The exponential lower bound for $1<p<2$]
\label{thm:low-Rn}
For every fixed $1<p<2$ and every $n\geq2$, there is a connected complete graph with positive integer edge weights and at least
$$
 R_n=\frac{3^n-2^{n+1}+3}{2}
$$
distinct spectral values. Hence $U_n(p)\geq R_n$.
\end{theorem}

\begin{proof}
For $n=2$, the unit-weight graph $K_2$ has the two spectral values $0$ and $2^{p-1}$, while $R_2=2$. Assume henceforth that $n\geq3$. Let $z_\rho=[x_\rho^0]$ run over the $R_n-1$ nonconstant eigenlines of the uniformly weighted $K_n$, let $r_\rho$ be the size of its zero block, let $x_\rho^0$ be the representative fixed by the local chart, and let $\lambda_\rho=\mathcal R_{p,\one}(z_\rho)$. Since $\Psi_p$ is injective, choose $v\in\mathbb Q^{\binom n2}$ outside the finitely many hyperplanes
\begin{equation}\label{eq:low-separating-direction}
 v\cdot\bigl(\Psi_p(z_\rho)-\Psi_p(z_\sigma)\bigr)=0,
 \qquad \rho\ne\sigma,
\end{equation}
and set $a_\epsilon=\one+\epsilon v$.

Choose the pairwise disjoint local charts used in the proof of Theorem~\ref{thm:low-interval}. In the chart at $z_\rho$ set
$$
 U_\rho=\{u:\varrho(u)<1\},\qquad
 \widehat B_{\rho,s}
 =\{[x_\rho^0+\delta_su]:u\in U_\rho\}.
$$
Let $\widetilde{\mathcal R}_\rho(w)=\mathcal R_{p,\one}([x_\rho^0+w])$ denote the quotient in this affine chart.
The $C^1$ convergence in Theorem~\ref{thm:low-splitting}, together with compactness of $\partial U_\rho=\{\varrho=1\}$, gives constants $c_\rho,s_\rho>0$ for which the norm of the differential with respect to $u$ of the rescaled quotient is bounded below whenever $0<s\leq s_\rho$. If $L_s=D\delta_s$, then, for $u\in\partial U_\rho$,
$$
 D_u\left[s^{-1}
 \bigl(\widetilde{\mathcal R}_\rho(\delta_su)-\lambda_\rho\bigr)\right]
 =s^{-1}L_s^{\mathsf T}
 \nabla\widetilde{\mathcal R}_\rho(\delta_su).
$$
Since $\|L_s\|\leq s^{1/2}$ for $0<s<1$,
\begin{equation}\label{eq:low-boundary-gradient}
 \|\nabla\widetilde{\mathcal R}_\rho(w)\|\geq c_\rho s^{1/2}
 \quad\bigl(w\in\delta_s(\partial U_\rho)\bigr),
 \qquad0<s\leq s_\rho.
\end{equation}
When $r_\rho=0$, $\|L_s\|=s^{1/p}$ and the sharper lower bound is of order $s^{(p-1)/p}$, which is still at least a constant times $s^{1/2}$.

The coefficient map is $C^1$ for every $p>1$, so $\nabla(v\cdot\Psi_p)$ is bounded on these fixed charts. Choose
$
 s_\epsilon=\epsilon^\theta$ with $1<\theta<2$. 
For $\epsilon>0$ small enough, $s_\epsilon\leq\min_\rho s_\rho$ and $\epsilon=o(s_\epsilon^{1/2})$. Equation \eqref{eq:low-boundary-gradient} shows that the straight-line family $\mathcal R_{p,\one}+t\epsilon\,v\cdot\Psi_p$, $0\leq t\leq1$, has no boundary critical point. Homotopy invariance of Brouwer degree and Lemma~\ref{lem:anisotropic-Morse} therefore show that the perturbed quotient has the same nonzero local degree $(-1)^{r_\rho}$ on $\widehat B_{\rho,s_\epsilon}$. The local degree of a gradient is invariant under changes of coordinates: if $y=h(x)$ is a coordinate transition, then
$
 \nabla_y(f\circ h^{-1})(h(x))
 =Dh(x)^{-\mathsf T}\nabla_xf(x);
$
the domain transition and the inverse-transpose target map have the same orientation sign, whose product is $+1$. The perturbed quotient therefore has a critical point $z_\rho(\epsilon)$ in that shrinking neighbourhood.

The expansion in Theorem~\ref{thm:low-splitting} also gives, uniformly on the shrinking neighbourhood,
$$
 \left|
 \mathcal R_{p,\one}(z_\rho(\epsilon))-\lambda_\rho
 \right|=O(s_\epsilon)=o(\epsilon).
$$
Since $z_\rho(\epsilon)\to z_\rho$,
\begin{equation}\label{eq:low-first-order-values}
 \mathcal R_{p,a_\epsilon}(z_\rho(\epsilon))
 =\lambda_\rho+
 \epsilon\,v\cdot\Psi_p(z_\rho)+o(\epsilon).
\end{equation}
Pairs with different base values remain separated. For pairs with the same base value, the coefficient of $\epsilon$ in their difference is nonzero by \eqref{eq:low-separating-direction}. Since there are finitely many pairs, all selected values are distinct for every sufficiently small positive $\epsilon$. Choose such an $\epsilon$ rational and small enough that every weight is positive. Clearing denominators preserves distinctness and gives positive integer weights. The constant eigenline gives the additional value zero, so the connected complete graph has at least $R_n$ distinct values.
\end{proof}

The cellular model also makes the reversal of local degrees across $p=2$ topologically explicit. For a cell $C_{P,N}$, let
$$
 \widehat C_{P,N}^{\,p}=\Theta_p^{-1}(C_{P,N}),\qquad
 \widehat b_{P,N}^{\,p}=\Theta_p^{-1}([b_{P,N}]).
$$

\begin{proposition}[Cellular incidence and reversal of degrees]
\label{prop:cellular-incidence}
For every $p>1$, the cells $\widehat C_{P,N}^{\,p}$ form a regular CW decomposition of $\Pj(V_n)$. Over $\mathbb F_2$, its cellular boundary is
\begin{equation}\label{eq:cellular-boundary}
 \partial[P\mid N]
 =\sum_{\substack{i\in P\\|P|>1}}[P\setminus\{i\}\mid N]
 +\sum_{\substack{j\in N\\|N|>1}}[P\mid N\setminus\{j\}].
\end{equation}
The integral cellular boundary is as follows. For each unordered label, choose the ordered representative $(P,N)$ for which $\min(P\cup N)\in P$, let
$$
 P=\{p_1<\cdots<p_k\},\qquad N=\{n_1<\cdots<n_\ell\},
$$
and orient $\overline C_{P,N}\cong\Delta^{k-1}\times\Delta^{\ell-1}$ by the product orientation. Transfer this orientation to $\overline{\widehat C_{P,N}^{\,p}}$ through $\Theta_p^{-1}$. For non-empty disjoint $A,B$, let $\operatorname{can}(A,B)$ be $(A,B)$ if $\min(A\cup B)\in A$ and $(B,A)$ otherwise, and set
$$
 \epsilon(A,B)=
 \begin{cases}
 1,&\operatorname{can}(A,B)=(A,B),\\
 (-1)^{(|A|-1)(|B|-1)},&\operatorname{can}(A,B)=(B,A).
 \end{cases}
$$
Then, with brackets denoting the oriented generator belonging to the canonical representative,
\begin{align}
 \partial_{\mathbb Z}[P\mid N]
 ={}&\mathbf1_{\{k>1\}}\sum_{a=1}^k
 (-1)^{a-1}\epsilon(P\setminus\{p_a\},N)
 [\operatorname{can}(P\setminus\{p_a\},N)]\notag\\
 &+\mathbf1_{\{\ell>1\}}\sum_{b=1}^{\ell}
 (-1)^{k+b-2}\epsilon(P,N\setminus\{n_b\})
 [\operatorname{can}(P,N\setminus\{n_b\})].
 \label{eq:integral-cellular-boundary}
\end{align}
For $p>2$, the Morse index of $\widehat b_{P,N}^{\,p}$ is $\dim C_{P,N}$. At $p=2$, the local critical group of the logarithmic model $\mathcal H_n$ at $b_C=[b_{P,N}]$ is concentrated in degree $\dim C_{P,N}$, whereas for $1<p<2$ the local critical group at $\widehat b_{P,N}^{\,p}$ is concentrated in degree $\codim C_{P,N}$. Thus $M_n^+$ is the face polynomial of this decomposition and $M_n^-$ is obtained from it by reversing degrees.
\end{proposition}

\begin{proof}
The global homeomorphism $\Theta_p^{-1}$ carries the regular CW decomposition of Proposition~\ref{prop:sign-cells} to the displayed cells, preserving its face poset and characteristic maps. A facet of $\Delta^{|P|-1}\times\Delta^{|N|-1}$ is obtained by deleting exactly one element of $P$ or $N$, provided the corresponding block remains non-empty. Every such facet occurs once, which gives \eqref{eq:cellular-boundary} over $\mathbb F_2$.

For the integral formula, the usual product-boundary convention gives, before replacing each face by its canonical representative,
$$
 \sum_{a=1}^k(-1)^{a-1}[P\setminus\{p_a\}\mid N]
 +(-1)^{k-1}\sum_{b=1}^{\ell}(-1)^{b-1}[P\mid N\setminus\{n_b\}],
$$
with the first sum omitted when $k=1$ and the second when $\ell=1$. If this replacement swaps the two sign blocks, the characteristic map is obtained by transposing simplex factors, whose orientation sign is $(-1)^{(|A|-1)(|B|-1)}$. This proves \eqref{eq:integral-cellular-boundary}; reduction modulo two recovers \eqref{eq:cellular-boundary}.

The superquadratic index statement is Theorem~\ref{thm:exact-morse-polynomial}, the $p=2$ statement for the logarithmic model is Corollary~\ref{cor:log-indices}, and the subquadratic statement is Theorem~\ref{thm:low-indices}. The corresponding generating functions are therefore the face polynomial and its degree reversal.
\end{proof}

\begin{proof}[Proof of Theorem~\ref{thm:intro-cellular-transition}]
Propositions~\ref{prop:Theta} and~\ref{prop:sign-cells} establish the duality map and cell structure, and Corollary~\ref{cor:uniform-barycentres} identifies the barycentres. The superquadratic indices, face polynomial, and ordering follow from Proposition~\ref{prop:uniform-rays-inertia}, Theorem~\ref{thm:exact-morse-polynomial}, and Proposition~\ref{prop:face-order}.

Theorem~\ref{thm:log-blowup}, Theorem~\ref{thm:H-critical-set}, and Corollary~\ref{cor:log-indices} give the $C^1$ blow-up limit and the critical points, critical groups, and local degrees of the logarithmic model. Theorem~\ref{thm:low-indices} gives the subquadratic critical groups and the degree-reversed polynomial, while Proposition~\ref{prop:cellular-incidence} gives the cellular boundary and reversal of degrees.
\end{proof}

\section{Caustics and critical-value Maxwell sets for complete graphs}
\label{sec:morse-chambers}

Local perturbations give separated critical values near the uniform weights. We now describe the global parameter sets on which critical points become degenerate or distinct critical values collide. For $C^2$ height functions, degeneracy is governed by the critical values of the normal-bundle projection, giving a caustic alongside the Maxwell set. Let $M$ be a compact $C^2$ definable manifold, let $\Psi:M\to\mathbb R^N$ be a $C^2$ definable embedding, and set $f_c(x)=c\cdot\Psi(x)$. Its normal bundle, viewed as an incidence manifold, is
\begin{equation}\label{eq:normal-incidence}
 \mathcal N_\Psi
 =\{(x,c)\in M\times\mathbb R^N:
 c\perp D\Psi_x(T_xM)\}.
\end{equation}

\begin{theorem}[Definable height-function discriminants]
\label{thm:height-discriminants}
There are closed definable sets $\Cau(\Psi),\Maxw(\Psi)\subset\mathbb R^N$, each of dimension at most $N-1$, such that:
\begin{enumerate}[label=\textup{(\roman*)}]
\item $f_c$ is Morse for $c\notin\Cau(\Psi)$;
\item on each connected component of $\mathbb R^N\setminus\Cau(\Psi)$, the number of critical points and their Morse-index multiset are constant;
\item if $c\notin\Cau(\Psi)\cup\Maxw(\Psi)$, the critical values of $f_c$ are pairwise distinct;
\item away from the caustic, the Maxwell set is locally a finite union of embedded $C^1$ hypersurfaces (one for each colliding pair of local critical branches), and globally the union of the two discriminants admits a finite definable Whitney stratification.
\end{enumerate}
Both $\mathbb R^N\setminus\Cau(\Psi)$ and $\mathbb R^N\setminus(\Cau(\Psi)\cup\Maxw(\Psi))$ have finitely many connected components.
\end{theorem}

\begin{proof}
A critical point of $f_c$ is exactly a point of $\mathcal N_\Psi$. The normal-bundle projection
$$
 \pi:\mathcal N_\Psi\rightarrow\mathbb R^N,\qquad(x,c)\mapsto c,
$$
is proper because $M$ is compact. In a local chart on $M$, the normal bundle is described by
$
 (D\Psi(x))^{\mathsf T}c=0.
$
A tangent vector $(\xi,\dot c)$ at $(x,c)$ therefore satisfies
$$
 \operatorname{Hess}_x(f_c)\,\xi
 +(D\Psi(x))^{\mathsf T}\dot c=0.
$$
The kernel of the derivative of the projection $\pi(x,c)=c$ consists of the pairs with $\dot c=0$, and is consequently identified with $\ker\operatorname{Hess}_x(f_c)$. Since the normal bundle and the coefficient space both have dimension $N$, $D\pi$ is invertible if and only if the critical point of $f_c$ is nondegenerate. Let $\Cau(\Psi)$ be the image under $\pi$ of its critical locus. This locus is closed and definable; since $\pi$ is proper, its image is closed. Definable Sard gives $\dim\Cau(\Psi)\leq N-1$. Above each component of the complement, $\pi$ is a proper local diffeomorphism and hence a finite covering. This proves the constancy assertion, including the indices.

Theorem~\ref{thm:C1-Maxwell}, with $S$ a point and $h=0$, gives a closed definable Maxwell set of dimension at most $N-1$, with generic finiteness of the critical set and pairwise distinct critical values. Away from the caustic, label the local critical branches by $x_1(c),\ldots,x_s(c)$. For $i\ne j$ set $\delta_{ij}(c)=f_c(x_i(c))-f_c(x_j(c))$. Differentiation along the critical branches gives
\begin{equation}\label{eq:abstract-HF}
 D\delta_{ij}(c)[v]
 =v\cdot\bigl(\Psi(x_i(c))-\Psi(x_j(c))\bigr).
\end{equation}
The gradient is nonzero because $\Psi$ is injective. Thus every non-empty zero set is a $C^1$ hypersurface. Taking $\Maxw(\Psi)=\mathfrak M(0,\Psi)$, definable Whitney stratification and the finiteness of definably connected components complete the proof.
\end{proof}

\begin{theorem}[Discriminants for definable families of embeddings]
\label{thm:height-family}
Let $I$ be a definable $C^1$ manifold and $\Psi:I\times M\to\mathbb R^N$ be definable. Assume that every $\Psi_s$ is a $C^2$ embedding and that, in local charts, $(s,x)\mapsto D_x\Psi_s(x)$ is $C^1$. There are closed definable caustic and Maxwell sets of codimension at least one in $I\times\mathbb R^N$. Off the caustic the critical-point count and index multiset of $c\cdot\Psi_s$ are constant on each connected component of the caustic complement; off the union of the caustic and the Maxwell set, the critical values are pairwise distinct.
\end{theorem}

\begin{proof}
In local coordinates on $M$, the universal critical incidence is the zero set of
$
 F(s,c,x)=(D_x\Psi_s(x))^{\mathsf T}c.
$
Its derivative in $c$ is surjective because $D_x\Psi_s(x)$ is injective. Thus the incidence is a definable $C^1$ manifold of dimension $\dim I+N$. The projection to $I\times\mathbb R^N$ is proper, since the inverse image of a compact parameter set is a closed subset of the product of that set with the compact manifold $M$. The normal-bundle calculation in Theorem~\ref{thm:height-discriminants} identifies the critical values of this projection with the parameters at which the height function has a degenerate critical point. Definable Sard gives the codimension bound, and over each component of the complement the projection is a finite covering. The argument used in the proof of Theorem~\ref{thm:C1-Maxwell}, with $h=0$, gives the closed Maxwell set and generic distinctness of the critical values. Away from the caustic, the calculation leading to \eqref{eq:abstract-HF} applies with $s$ fixed, and its derivative in the $c$-directions is still \eqref{eq:abstract-HF}.
\end{proof}

\subsection{Discriminants for complete graphs}

Let $m=\binom n2$, $\mathcal A_n=(0,\infty)^m$, and $a^0=(1,\ldots,1)$.

\begin{theorem}[Complete-graph discriminants at a fixed exponent]
\label{thm:global-complete-chamber}
Fix $n\geq2$ and $p>2$. There are closed definable sets
$$
 \Cau_{n,p},\Maxw_{n,p}\subset\mathcal A_n,\qquad
 \dim\Cau_{n,p},\dim\Maxw_{n,p}\leq m-1,
$$
with the following properties, where $\Omega_{n,p}$ denotes the connected component of $\mathcal A_n\setminus\Cau_{n,p}$ containing $a^0$.
\begin{enumerate}[label=\textup{(\roman*)}]
\item If $a\notin\Cau_{n,p}$, every nonconstant eigenline is regular.
\item On each component of $\mathcal A_n\setminus\Cau_{n,p}$, the number of nonconstant eigenlines and their reduced Morse-index multiset are constant.
\item For every $a\in\Omega_{n,p}$, the quotient has $R_n-1$ nonconstant eigenlines and Morse polynomial $M_n^+(t)$. Including the constant eigenline, the corresponding graph has $R_n$ eigenlines.
\item If $a\in\Omega_{n,p}\setminus\Maxw_{n,p}$, all $R_n$ eigenvalues are pairwise distinct.
\end{enumerate}
Away from the caustic, the loci on which two critical values coincide are locally embedded hypersurfaces, and the union of the caustic and Maxwell set has a finite definable stratification.
\end{theorem}

\begin{proof}
Apply Theorem~\ref{thm:height-discriminants} to $\Psi_p$ and use \eqref{eq:height-identity}. Theorem~\ref{thm:exact-morse-polynomial} shows that the uniform weight vector $a^0$ lies outside the caustic. The finite covering over the component containing it makes the critical-point count and Morse polynomial constant there. The constant value is zero, while every nonconstant value is positive, so it cannot coincide with a nonconstant critical value.
\end{proof}

\begin{theorem}[Joint exponent--weight discriminants]
\label{thm:joint-chamber}
There are closed definable caustic and Maxwell sets of codimension at least one in $(2,\infty)\times\mathcal A_n$. The complement of the caustic has finitely many connected components. On each connected component of the caustic complement, the critical-point count and Morse-index multiset are constant. On each component of the complement of the two discriminants, the quotient is Morse and its critical values are pairwise distinct. The parameter curve
$
 \{(p,a^0):p>2\}
$
lies in one component of the caustic complement, denoted $\Omega_n^{\mathrm{unif}}$. For every $(p,a)$ in this component, the corresponding quotient has $R_n-1$ nonconstant eigenlines and reduced Morse polynomial $M_n^+(t)$; away from the joint Maxwell set, all critical values are distinct.
\end{theorem}

\begin{proof}
Apply Theorem~\ref{thm:height-family} to $(p,[\bar x])\mapsto\Psi_p([\bar x])$. Every point of this parameter curve is outside the caustic by Proposition~\ref{prop:uniform-rays-inertia}. The curve is connected and therefore lies in one connected component of the caustic complement.
\end{proof}

\begin{corollary}[Uniform neighbourhood over a compact exponent interval]\label{cor:compact-p-tube}
If $J\Subset(2,\infty)$ is compact, there is a definable neighbourhood $V_J\subset\mathcal A_n$ of $a^0$ such that for every $(p,a)\in J\times V_J$, the reduced quotient has exactly $R_n-1$ critical points; including the constant eigenline, the graph has exactly $R_n$ eigenlines. All nonconstant eigenlines are regular, and the reduced Morse polynomial is $M_n^+(t)$.
\end{corollary}

\begin{proof}
Since $J\times\{a^0\}$ is compact and lies in the open component $\Omega_n^{\mathrm{unif}}$, there is a neighbourhood $V_J$ of $a^0$ such that $J\times V_J\subset\Omega_n^{\mathrm{unif}}$. Shrinking $V_J$ to a definable open ball gives the assertion.
\end{proof}

\begin{corollary}[One integer-weighted graph on an exponent interval]
\label{cor:integer-interval}
For every compact $J\Subset(2,\infty)$ there are positive integer edge weights on $K_n$ for which the graph has exactly $R_n$ $p$-eigenlines, with reduced Morse polynomial $M_n^+$, for every $p\in J$. The weights can be chosen so that the resulting graph has pairwise distinct eigenvalues for all but finitely many $p\in J$.
\end{corollary}

\begin{proof}
The joint Maxwell set in $J\times V_J$ has dimension at most $m$. Let $E\subset V_J$ be the definable set of weights $a$ for which the fibre of the Maxwell set over $a$ has positive dimension. Definable fibre dimension gives $\dim E\leq m-1$, and definable closure preserves dimension. Hence $V_J\setminus\overline E$ is a non-empty open definable set and contains a rational point. Choose $a\in\mathbb Q^m\cap(V_J\setminus\overline E)$. The fibre of the Maxwell set over this $a$ has dimension at most zero and hence, by definability, is finite. Clearing denominators multiplies all weights by the same positive scalar; this preserves eigenlines, indices, and equalities between eigenvalues while scaling the eigenvalues.
\end{proof}

\begin{corollary}[Local chamber near the uniform weights]\label{cor:old-local-from-global}
For every $p>2$, a sufficiently small neighbourhood of $a^0$ lies in $\Omega_{n,p}$. After shrinking this neighbourhood, the local Maxwell set is a finite union of the embedded $C^1$ hypersurfaces determined by pairs of local critical-point branches.
\end{corollary}

\subsection{Generic distinctness of eigenvalues for complete signed graphs}

The reduced unsigned problem uses only the edge coordinates of $\Psi_p$. For complete signed graphs with potentials, the enlarged coefficient map below supplies the corresponding separation mechanism on the full projective space. Fix $p>1$, a signature $\varepsilon=(\varepsilon_{ij})_{i<j}$ on $K_n$, and positive vertex measures $\nu_i$. Recall that
$
 \cM_{p,\nu}(x)=\sum_i\nu_i|x_i|^p
$
and
\begin{equation}\label{eq:signed-coefficient-map}
 \Xi_{p,\varepsilon,\nu}([x])
 =\frac{\bigl((|x_i|^p)_i,\,(|x_i-\varepsilon_{ij}x_j|^p)_{i<j}\bigr)}{\cM_{p,\nu}(x)}.
\end{equation}

\begin{theorem}[Coefficient maps for complete signed graphs]
\label{thm:signed-coefficient-embedding}
The map
$$
 \Xi_{p,\varepsilon,\nu}:\mathbb{R}P^{n-1}
 \rightarrow\mathbb R^{n+m}
$$
is a $C^1$ definable embedding for every $p>1$, and is $C^2$ for $p>2$. Consequently, outside a closed definable subset of codimension at least one in $\mathbb R^n\times\mathbb R^m$, the displayed homogeneous linear-form quotient associated with the complete signed graph
$$
 [x]\mapsto
 \frac{\sum_{i<j}a_{ij}|x_i-\varepsilon_{ij}x_j|^p+\sum_i b_i|x_i|^p}{\cM_{p,\nu}(x)}
$$
has finitely many eigenlines with pairwise distinct eigenvalues. The same conclusion holds on an open dense subset of $\mathbb R^n\times(0,\infty)^m$, corresponding to arbitrary real potentials and positive edge weights. If $p>2$, exclusion of the additional caustic makes the quotient Morse, with pairwise distinct critical values.
\end{theorem}

\begin{proof}
Normalise by $\cM_{p,\nu}(x)=1$. Equality of the two vectors gives all $|x_i|$ and all quantities $|x_i-\varepsilon_{ij}x_j|$. Taking $2/p$ powers,
$$
 2\varepsilon_{ij}x_ix_j
 =x_i^2+x_j^2-|x_i-\varepsilon_{ij}x_j|^2.
$$
Thus the entire rank-one Gram matrix $xx^{\mathsf T}$ is recovered, so $y=\pm x$. This proves injectivity.

For immersion, let $h$ be tangent to the mass sphere and suppose every component differential vanishes. Coordinate components give $h_i=0$ when $x_i\ne0$. If $x_i=0$, choose $j$ with $x_j\ne0$; the edge component for $(i,j)$ gives $h_i-\varepsilon_{ij}h_j=h_i=0$. Hence $h=0$, so the descended map is an immersion. Since $\mathbb RP^{n-1}$ is compact and the map is injective, it is a $C^1$ embedding. Only the $C^2$ conclusion uses $p>2$. Apply Theorem~\ref{thm:C1-Maxwell} with $S$ a point, $h=0$, and coefficient vector $(b,a)$, and intersect the open dense complement of its Maxwell set with $\mathbb R^n\times(0,\infty)^m$. When $p>2$, apply Theorem~\ref{thm:height-discriminants} as well and exclude the caustic.
\end{proof}

Near the uniformly weighted unsigned graph, the generic separation can be realised by an explicit one-parameter family of integer edge weights.

\begin{proposition}[Concrete integer weights near the uniformly weighted $K_n$]
\label{prop:explicit-integer}
Fix $n\geq2$ and $p>2$, and order the $m=\binom n2$ edges as $e=0,\ldots,m-1$. There is an integer $T\geq2$ such that, for every sufficiently large integer $Q$, the weights
\begin{equation}\label{eq:explicit-integer-weights}
 A_e(Q)=Q+T^e
\end{equation}
produce exactly $R_n$ eigenlines. The $R_n-1$ nonconstant eigenlines are regular, have the prescribed index distribution, and carry pairwise distinct eigenvalues.
\end{proposition}

\begin{proof}
At $a=a^0$, for each pair of critical branches whose values collide, define
$$
 g_{\rho\sigma}
 =\nabla_a\lambda_\rho(a^0)-\nabla_a\lambda_\sigma(a^0).
$$
The Hellmann--Feynman formula identifies this vector with $\Psi_p(z_\rho)-\Psi_p(z_\sigma)$, which is nonzero by injectivity. Hence
$
 P_{\rho\sigma}(T)
 =g_{\rho\sigma}\cdot(1,T,\ldots,T^{m-1})
$
is a nonzero polynomial. Choose an integer $T\geq2$ which is not a zero of any of these finitely many polynomials. Along $a(t)=a^0+t(1,T,\ldots,T^{m-1})$, the difference between any pair of branches whose values coincide at $t=0$ has nonzero derivative there, and every pair whose values are distinct at $t=0$ remains separated for small $t>0$. For all sufficiently large $Q$,
$
 (A_e(Q)/Q)_e=a(1/Q)
$
lies sufficiently close to $a^0$ along this curve. Scaling all weights by $Q$ preserves the critical points, regularity, and indices, and multiplies every eigenvalue by $Q$.
\end{proof}

Below $p=2$, the same explicit family works, with local degree replacing Morse nondegeneracy.

\begin{corollary}[Integer weights for $p\ne2$]
\label{cor:integer-both-sides}
Fix $n\geq2$ and $p>1$ with $p\ne2$, and order the $m=\binom n2$ edges as $e=0,\ldots,m-1$. There is an integer $T\geq2$ such that, for every sufficiently large integer $Q$, the complete graph with edge weights
\begin{equation}\label{eq:integer-both-sides}
 A_e(Q)=Q+T^e
\end{equation}
has at least $R_n$ distinct spectral values.

If $p>2$, the graph has exactly $R_n$ eigenlines; every nonconstant eigenline is regular, and the reduced index distribution is $M_n^+$. If $1<p<2$, the conclusion is the stated lower bound on distinct spectral values.
\end{corollary}

\begin{proof}
The case $p>2$ is Proposition~\ref{prop:explicit-integer}. Suppose that $1<p<2$. For $n=2$, the weights consist only of $A_0(Q)=Q+1$, and the two values $0$ and $2^{p-1}(Q+1)$ prove the assertion. Assume $n\geq3$.

Let $z_\rho$ range over the nonconstant eigenlines of the uniformly weighted $K_n$ and set
$
 g_{\rho\sigma}
 =\Psi_p(z_\rho)-\Psi_p(z_\sigma)\in\mathbb R^m\setminus\{0\}.
$
For each pair,
$
 P_{\rho\sigma}(T)
 =g_{\rho\sigma}\cdot(1,T,\ldots,T^{m-1})
$
is a nonzero polynomial. Choose an integer $T\geq2$ outside the finite union of their real root sets. In the proof of Theorem~\ref{thm:low-Rn}, take
$$
 v=(1,T,\ldots,T^{m-1}),\qquad \epsilon=Q^{-1}.
$$
Degree persistence and the uniform $o(\epsilon)$ estimate in \eqref{eq:low-first-order-values} show that the selected values are distinct for every sufficiently small $\epsilon>0$, and hence for every sufficiently large $Q$. Clearing the common denominator changes $1+Q^{-1}T^e$ into \eqref{eq:integer-both-sides} and scales all values by the same positive factor.
\end{proof}

\section{Extremal spectral growth}\label{sec:growth}
For fixed $p>1$, define
$$
 B_n(p)=\max\{\#\sigma_p(d):d\in D_n\}.
$$
Define $U_n(p)$ by restricting this maximum to unsigned, zero-potential, unit-measure weighted graphs. These maxima exist because the possible cardinalities form non-empty bounded subsets of $\mathbb N$. Define
$$
 B_n^{\mathrm{opt}}
 =\max\{\#\sigma_p(d):(d,p)\in\Theta_n\}
 =\max_{p>1}B_n(p).
$$
The uniform bound in Theorem~\ref{thm:intro-main}\ref{part:intro-uniform} shows that the possible cardinalities are bounded, so the last maximum also exists.

Fix $p>2$. For the unsigned $p$-Laplacian on a complete graph with edge weights $a$, zero potential, and unit vertex measures, let
$$
 \mathcal F_{p,a}(x,\lambda)
 =
 \left(\sum_{j\ne i}a_{ij}\Phi_p(x_i-x_j)-\lambda\Phi_p(x_i)
 \right)_{i=1}^n.
$$
At a normalised eigenpair, let
$
 L_{p,a,x,\lambda}=D_x\mathcal F_{p,a}(x,\lambda).
$
Homogeneity gives $L_{p,a,x,\lambda}x=0$. We call a nonconstant eigenline \emph{regular} when
$
 \ker L_{p,a,x,\lambda}=\R x.
$
For a regular eigenline, define its \emph{Morse index on $\mathbb RP^{n-1}$} by
$$
 \indm_{\mathrm{full}}(x,\lambda)
 =
 \indm_{-}
 \left(L_{p,a,x,\lambda}\big|_{x^\perp}
 \right).
$$
This is independent of the normalised representative. Up to the positive factor $p/\cM_{p,\mathbf 1}(x)$, the displayed restriction is the Hessian of the Rayleigh quotient on $\Sph^{n-1}$. When the same eigenline is represented as a critical point of the reduced quotient on $\Pj(V_n)$, Proposition~\ref{prop:index-shift} gives
$$
 \indm_{\mathrm{full}}(x,\lambda)
 =1+\indm_{\mathrm{red}}([\bar x]).
$$

\begin{theorem}[Exact eigenline count near the uniformly weighted $K_n$]
\label{thm:complete-local-unfolding}
Fix $n\geq2$ and $p>2$, and set
$$
 m=\binom n2,\qquad
 a^0=(1)_{1\leq i<j\leq n},\qquad
 R_n=\frac{3^n-2^{n+1}+3}{2}.
$$
There exists a neighbourhood $V\subset(0,\infty)^m$ of $a^0$ and a relatively closed set $\mathcal D\subset V$, which is a finite union of embedded $C^1$ hypersurfaces, with the following properties.

\begin{enumerate}[label=\textup{(\roman*)}]
\item For every $a\in V$, the unsigned $p$-Laplacian on $K_n$ with edge weights $a$, zero potential, and unit vertex measures has exactly $R_n$ eigenlines. The constant eigenline is its only eigenline at eigenvalue zero, and all $R_n-1$ nonconstant eigenlines are regular.

\item For every $j=1,\ldots,n-1$, precisely
\begin{equation}\label{eq:complete-index-distribution}
 \binom n{j+1}(2^j-1)
\end{equation}
of the nonconstant eigenlines have Morse index $j$ on $\mathbb RP^{n-1}$.

\item For every $a\in V\setminus\mathcal D$, all $R_n$ eigenlines carry pairwise distinct eigenvalues. In particular, the full spectrum has cardinality $R_n$. The set $V\setminus\mathcal D$ is open, dense, and of full Lebesgue measure in $V$.
\end{enumerate}

Rational points of $V\setminus\mathcal D$ lie arbitrarily close to $a^0$. Clearing denominators gives positive integer edge weights for which the exact eigenline count, regularity, Morse-index distribution, and distinctness of the eigenvalues remain valid.
\end{theorem}

\begin{proof}
Choose a sufficiently small connected neighbourhood $V\subset\Omega_{n,p}$ of the uniform weight vector $a^0$. Theorem~\ref{thm:global-complete-chamber} then gives exactly $R_n-1$ nonconstant reduced critical points for every $a\in V$, all regular, with reduced Morse polynomial $M_n^+(t)$. The positively weighted complete graph has the constant eigenline as its unique eigenline at eigenvalue zero, so the total eigenline count is $R_n$.

A point of reduced index $q$ has Morse index $q+1$ on $\mathbb RP^{n-1}$ by Proposition~\ref{prop:index-shift}. Taking the coefficient of $t^{j-1}$ in \eqref{eq:index-count-reduced} gives
$$
 \binom n{j+1}(2^j-1),\qquad j=1,\ldots,n-1,
$$
which proves the index formula.

By Corollary~\ref{cor:old-local-from-global}, after shrinking $V$ further, the Maxwell set in $V$ is a finite union of embedded $C^1$ hypersurfaces. Set
$
 \mathcal D=V\cap\Maxw_{n,p}.
$
Theorem~\ref{thm:global-complete-chamber} shows that all critical values are distinct on $V\setminus\mathcal D$. The constant value is zero and every nonconstant value is positive, so no constant and nonconstant values coincide. The complement of this finite hypersurface union is open, dense, and of full Lebesgue measure.

Rational points are dense in $V\setminus\mathcal D$. Multiplying a positive rational weight vector by a common denominator leaves all eigenlines and indices unchanged and multiplies all eigenvalues and linearisations by the same positive scalar. This gives the asserted positive integer weights.
\end{proof}

\begin{corollary}[The two-sided extremal discontinuity at $p=2$]
\label{cor:p2-transition}
For every $n\geq1$,
$$
 U_n(2)=B_n(2)=n.
$$
By contrast, for every $p>1$ with $p\ne2$,
$$
 U_n(p)\leq B_n(p),
 \qquad
 U_n(p)\geq\frac{3^n-2^{n+1}+3}{2}.
$$
Consequently, for every $n\geq3$,
$$
 \liminf_{\substack{p\to2\\p\ne2}}U_n(p)
 \geq\frac{3^n-2^{n+1}+3}{2}
 >n=U_n(2),
$$
and the same conclusion holds with $B_n(p)$ in place of $U_n(p)$. Thus $U_n$ and $B_n$ are discontinuous at $p=2$ for every $n\geq3$.
\end{corollary}

\begin{proof}
The assertion for $n=1$ is immediate. At $p=2$ the graph equation is a symmetric generalised matrix eigenproblem with positive-definite diagonal mass matrix, so it has at most $n$ distinct eigenvalues. Equality for $B_n(2)$ is attained by taking no edges and choosing the ratios $b_i/\nu_i$ pairwise distinct. Equality for $U_n(2)$ is attained by a path, whose ordinary Laplacian has simple spectrum. The assertion for $p>2$ is Theorem~\ref{thm:complete-local-unfolding}, and the assertion for $1<p<2$ is Theorem~\ref{thm:low-Rn}.
\end{proof}

This discontinuity concerns the extremal function, whose optimising graph data may depend on $p$. The mechanism behind this discontinuity is already visible in the uniformly weighted $K_n$. At $p=2$, the three-level vectors considered above lie in the single $(n-1)$-dimensional eigenspace with eigenvalue $n$, and a linear perturbation can produce at most $n$ distinct values. For $p>2$, one has $\Phi_p'(0)=0$, whereas $\Phi_2'(0)=1$; the corresponding three-level eigenlines are regular and can be separated simultaneously by perturbing the edge weights. For $1<p<2$, the anisotropic analysis gives a nonzero local degree at every nonconstant eigenline of the uniformly weighted $K_n$. The logarithmic functional $\mathcal H_n$ is the first-order model for the transition as $p\to2$ from either direction.

The Krasnosel'skii genus construction supplies at most $n$ distinct values, so most of the spectrum produced above lies outside its min--max sequence.

\begin{corollary}[Eigenvalues outside the Krasnosel'skii genus min--max sequence]
\label{cor:nonkrasnoselskii-growth}
For every $n\geq2$ and every fixed $p>1$ with $p\ne2$, there is a connected complete graph with positive integer edge weights having at least
\begin{equation}\label{eq:nonkrasnoselskii-count}
 R_n-n
 =
 \frac{3^n-2^{n+1}+3-2n}{2}
\end{equation}
distinct eigenvalues that do not occur among the Krasnosel'skii genus min--max values. In particular, the number of such values is
$$
 \left(\frac12+o(1)\right)3^n.
$$

More generally, let $J\Subset(1,2)$ or $J\Subset(2,\infty)$ be compact. A complete graph with positive integer edge weights can be chosen so that, for every $p\in J$ outside a finite set, at least $R_n-n$ of its spectral values do not belong to the $n$-term Krasnosel'skii genus min--max sequence. If $J\Subset(2,\infty)$, the graph has exactly $R_n$ eigenlines and reduced Morse polynomial $M_n^+$ for every $p\in J$.
\end{corollary}

\begin{proof}
On the unit $p$-sphere
$$
 S_p=\left\{x\in\mathbb R^n:\sum_i|x_i|^p=1\right\},
$$
let $\gamma$ denote the Krasnosel'skii genus and define
$$
 \lambda_j^{\mathrm K}(a)
 =\inf_{\substack{C\subset S_p\ \text{compact},\ C=-C\\\gamma(C)\geq j}}\ \sup_{x\in C}E_{p,a}(x),
 \qquad j=1,\ldots,n.
$$
These are eigenvalues and form the $n$-term Krasnosel'skii genus min--max sequence, counted with variational multiplicity; see \cite{DeiddaPuttiTudisco2023} and \cite[Section~3.2.1]{DeiddaTudiscoZhang2025}. It therefore contains at most $n$ distinct spectral values.

For fixed $p\ne2$, Corollary~\ref{cor:integer-both-sides} gives a connected complete graph with positive integer edge weights and at least $R_n$ distinct spectral values. At least $R_n-n$ of them do not occur in the Krasnosel'skii genus min--max sequence. If $J\Subset(2,\infty)$, use Corollary~\ref{cor:integer-interval}; if $J\Subset(1,2)$, use Theorem~\ref{thm:low-interval}. In either case, the graph so obtained has at least $R_n$ distinct values for every $p\in J$ outside a finite set. At most $n$ of them belong to the Krasnosel'skii genus min--max sequence, proving the interval assertion.
\end{proof}

\begin{corollary}[Persistence under variation of the exponent]
\label{cor:complete-p-persistence}
Fix $n\geq2$ and $p_0>2$. There are positive integer edge weights $A=(A_{ij})$ on $K_n$ and an open interval $J\subset(2,\infty)$ containing $p_0$ such that, for every $p\in J$, the corresponding weighted complete graph has exactly $R_n$ $p$-eigenlines, whose eigenvalues are pairwise distinct. All nonconstant eigenlines are regular, and for every $j=1,\ldots,n-1$, exactly
$$
 \binom n{j+1}(2^j-1)
$$
of them have Morse index $j$ on $\mathbb RP^{n-1}$.
\end{corollary}

\begin{proof}
Choose a compact interval $J_0\Subset(2,\infty)$ containing $p_0$ in its interior. Corollary~\ref{cor:compact-p-tube} gives a neighbourhood $V_{J_0}$ of $a^0$ on which the exact eigenline count, regularity, and reduced index distribution hold for every $p\in J_0$. At the fixed exponent $p_0$, the Maxwell set is a finite union of hypersurfaces near $a^0$. We may therefore choose a positive rational vector
$
 a^*\in V_{J_0}\setminus\Maxw_{n,p_0}.
$
Corollary~\ref{cor:compact-p-tube} gives $R_n$ eigenlines for every $(p,a)\in J_0\times V_{J_0}$, of which the $R_n-1$ nonconstant eigenlines are regular, while the choice of $a^*$ makes their eigenvalues pairwise distinct at $p_0$. Along $J_0\times\{a^*\}$, the finite covering yields finitely many continuous eigenvalue branches near $p_0$. Their pairwise gaps remain nonzero on a smaller open interval $J$ about $p_0$.

Clearing the denominators of $a^*$ multiplies every eigenvalue and linearisation by the same positive constant for every $p$. It preserves the exact eigenline count, regularity, the stated Morse-index distribution, and distinctness of the eigenvalues, and gives the asserted positive integer weights.
\end{proof}

At $p=4$, the lower construction and the tensor upper bound have the same exponential rate.

\begin{corollary}[Sharp exponential growth constant at $p=4$]
\label{cor:fourth-order-growth}
One has
$$
 \log U_n(4)=n\log3+O(\log n),\qquad
 \log B_n(4)=n\log3+O(\log n),
$$
and hence the two limits in \eqref{eq:intro-fourth-order-rate} are equal to $3$.
\end{corollary}

\begin{proof}
Combine Corollary~\ref{cor:p2-transition} with \eqref{eq:intro-even-tensor} and take logarithms or $n$th roots.
\end{proof}

\begin{corollary}[The exact value of $U_3(p)$]
\label{cor:ordinary-small-growth}
For every $p>1$,
$$
 U_3(p)=
 \begin{cases}
 7,&p\neq2,\\
 3,&p=2.
 \end{cases}
$$
Moreover, $U_4(p)\geq26$ for every $p\ne2$, and
$$
 \lim_{\substack{p\to2\\p\ne2}}U_3(p)=7>3=U_3(2).
$$
\end{corollary}

\begin{proof}
For $p>2$, Theorem~\ref{thm:complete-local-unfolding} gives $U_3(p)\geq R_3=7$. For $1<p<2$, Theorem~\ref{thm:low-Rn} gives the same lower bound. The balanced-triangle bound in Theorem~\ref{thm:balanced-triangle}, proved in Appendix~\ref{app:small-graphs}, gives $U_3(p)\leq7$ whenever $p\ne2$.

At $p=2$ the problem is a symmetric $3\times3$ matrix eigenproblem, so $U_3(2)\leq3$; the unit-weight path has eigenvalues $0,1,3$, proving equality. Theorem~\ref{thm:complete-local-unfolding} for $p>2$ and Theorem~\ref{thm:low-Rn} for $1<p<2$ give $U_4(p)\geq R_4=26$. The asserted jump follows.
\end{proof}

\begin{proof}[Proof of Theorem~\ref{thm:intro-growth-chambers}]
Corollary~\ref{cor:p2-transition} proves the assertion at $p=2$. Theorem~\ref{thm:complete-local-unfolding} and Theorem~\ref{thm:low-Rn} give the lower bound for every $p\ne2$, and Corollaries~\ref{cor:fourth-order-growth} and~\ref{cor:integer-both-sides} give the growth constant at $p=4$ and the integer-weight construction.

Part~\textup{(iv)} follows from Theorem~\ref{thm:joint-chamber} and Corollary~\ref{cor:integer-interval}, and part~\textup{(v)} is Theorem~\ref{thm:low-interval}. Theorems~\ref{thm:tangent-family} and~\ref{thm:tangent-Maxwell} prove part~\textup{(vi)}.
\end{proof}

\section{Consequences and open directions}\label{sec:discussion}

The full spectrum of a finite graph is finite for every $1<p<\infty$, with a bound depending only on the number of vertices. Definable triviality gives a finite definable partition of the coefficient--exponent family on which the spectra and eigenvector sets are topologically trivial.

For unsigned complete graphs the structure is more explicit. The map $\Psi_p$ embeds $\mathbb RP^{n-2}$ so that the reduced Rayleigh quotients are height functions. The projectivised $L^p$-duality map pulls the coordinate-hyperplane arrangement back to a regular CW decomposition indexed by pairs $[P\mid N]$ of disjoint non-empty subsets of $[n]$, modulo interchange of $P$ and $N$. For $p\ne2$, the inverse images of the cell barycentres are precisely the nonconstant eigenlines of the uniformly weighted $K_n$. For $p>2$, their Morse indices equal the cell dimensions; for $1<p<2$, their critical groups are concentrated in the complementary cell codimensions. At $p=2$, the local critical group of the logarithmic model at each barycentre is supported in the degree equal to the corresponding superquadratic Morse index. The two Morse polynomials are therefore
$$
 M_n^+(t)=\frac12\sum_{\substack{r,s\geq1\\r+s\leq n}}
 \binom{n}{r,s,n-r-s}t^{\,r+s-2},
 \qquad
 M_n^-(t)=t^{\,n-2}M_n^+(t^{-1}),
$$
and the integral cellular boundary is explicit.

Consequently,
$$
 R_n-1=M_n^+(1)=M_n^-(1)=\frac{3^n-2^{n+1}+1}{2},
$$
so $U_n(p)\geq R_n$ for every $p\neq2$, whereas $U_n(2)=B_n(2)=n$. At $p=4$,
$$
 \frac{3^n-2^{n+1}+3}{2}
 \leq U_n(4)\leq B_n(4)\leq n3^{\,n-1}.
$$
Thus both $U_n(4)$ and $B_n(4)$ have exponential growth constant $3$; the exact polynomial prefactors remain open.

For the uniformly weighted $K_n$, nonzero spectral values are determined by unordered pairs $(k,\ell)$ of positive block sizes with $k+\ell\leq n$. There are $|\mathcal I_n|=\lfloor n^2/4\rfloor$ such pairs. Their values are distinct outside the finite exceptional set $\mathfrak E_n$, while the number of nonconstant eigenlines is always $R_n-1$ for $p\ne2$. At $p=3$, the value depends only on $k\ell$; for every $n\geq5$, different pairs can therefore give the same value. The resulting cardinality is governed by the multiplication-table problem and is $n^{2-o(1)}$ but $o(n^2)$.

For a graph with $n$ vertices and $m$ positive-weight edges, the real-exponent Pfaffian argument bounds the sum of the rational Betti numbers of all normalised eigenvector sets by $\exp(O((n+m)^2))$. At a fixed rational exponent $p=A/B$, the semialgebraic estimate improves the logarithm of the bound to
$$
 O\bigl((n+m)\log((n+m)D)\bigr),
 \qquad D=2\max\{B,A-B,1\}.
$$
Incorporating graph incidence, sparsity, and fewnomial structure may sharpen both estimates.

The endpoint theorems give Hausdorff limits of the normalised finite-$p$ eigenpair sets and prove finiteness of the Clarke critical-value sets of the limiting quotients on the sphere. Their relation to intrinsic graph $1$- and $\infty$-eigenvalue theories depends on the normalisation and set-valued eigen-equation. For the uniformly weighted $K_n$, Corollary~\ref{cor:uniform-endpoint-limits} gives $\#\Lambda_1(K_n)=\lfloor n/2\rfloor+1$, $\#\Lambda_\infty(K_n)=2$, $\#\mathcal L_{\infty,n}=R_n$, and the displayed exact formula for $\#\mathcal L_{1,n}$; both projective limit sets have cardinality $3^{\,n-o(n)}$. Related endpoint theories include Chang's graph $1$-Laplacian, the $p\downarrow1$ limits of first Dirichlet eigenfunctions studied by Ge, Hua, and Jiang, and the graph $\infty$-eigenpairs of Deidda, Burger, Putti, and Tudisco \cite{Chang2016,GeHuaJiang2021,DeiddaBurgerPuttiTudisco2026}.

Several directions remain.
\begin{enumerate}
\item Determine the exact values of $U_n(p)$ and $B_n(p)$ for $p\ne2$, beginning with the remaining unrestricted small-$n$ cases and the factor-of-order-$n$ gap between the present lower and upper bounds at $p=4$.
\item Determine the exceptional collision set $\mathfrak E_n$---the exponents at which two distinct block-size pairs give the same spectral value for the uniformly weighted $K_n$---together with its cardinality, arithmetic structure, and behaviour as $n\to\infty$.
\item Beyond the balanced three-vertex case treated in Appendix~\ref{app:small-graphs}, describe the bifurcation set and its complementary parameter regions for $1<p<2$. In particular, determine uniqueness of the critical point persisting near each eigenline of the uniformly weighted $K_n$, the total eigenline count after perturbation, and a stratification of the parameters at which local uniqueness or the eigenline count changes.
\item For unweighted simple graphs, graphs of degree bounded independently of $n$, and graphs whose positive integer edge weights are bounded by a constant independent of $n$, determine the exponential limsup and liminf of the extremal spectral cardinality and whether they coincide. Sharpen the Pfaffian and rational-exponent bounds by exploiting graph incidence and fewnomial structure.
\item For the signed weighted problems with potentials considered here, identify precisely when the endpoint Clarke critical-value sets agree with intrinsic graph $1$- and $\infty$-eigenvalue spectra. Specify in each comparison the normalisation and the set-valued endpoint eigen-equation.
\item For $p>2$, construct a Morse--Smale pseudo-gradient whose unstable manifolds are exactly the arrangement cells, and construct a corresponding topological model for $1<p<2$. For such a cell-compatible construction, compare the Thom--Smale differential with the explicit integral cellular boundary and describe the resulting cancellations in the Morse complex.
\end{enumerate}

The appendices prove the parameter-dependent, topological, and endpoint results stated above, extend the effective bounds to matrix and linear-form systems, and complete the low-dimensional analysis.

\appendix
\section{Definable families, spectral branches, and eigenvector sets}
\label{app:families}

Fix $d=(a,\varepsilon,b,\nu)$. The spectrum as $p$ varies is encoded by the definable set
\begin{equation}\label{eq:p-spectral-set}
 \mathcal{S}(d)
 =\{(p,\lambda):p>1,\ \lambda\in\sigma_p(d)\}
\end{equation}
This set is the fibre of $\mathcal{S}_n$ over $d$, and its vertical fibres are finite.

\begin{theorem}[Definable spectral branches]\label{thm:intro-branches}
There exists a finite set $P\subset(1,\infty)$ such that, on every connected component $I$ of $(1,\infty)\setminus P$, one has
$$
 \sigma_p(a,\varepsilon,b,\nu)
 =\{\lambda_{1,I}(p)<\cdots<\lambda_{r_I,I}(p)\},\qquad p\in I,
$$
where the functions $\lambda_{k,I}:I\to\R$ are continuous and definable in $\R_{\exp}$. In particular, the number of distinct eigenvalues is constant on $I$, and it can change at only finitely many exponents.
\end{theorem}

For graph data $d=(a,\varepsilon,b,\nu)$, set
$$
 \operatorname{Bif}_{\#}(d)
 =
 \{p>1:q\mapsto\#\sigma_q(d)\text{ is not locally constant at }p\}.
$$

\begin{theorem}[Uniform bounds for spectral branches]
\label{thm:intro-uniform-branches}
For every $n\geq1$ there are non-effective constants $C_n,L_n<\infty$ such that every admissible fixed datum $d$ on $n$ vertices satisfies
$$
 \#\operatorname{Bif}_{\#}(d)\leq C_n.
$$
Moreover, there is an exceptional set $P_d\subset(1,\infty)$ with $\#P_d\leq L_n$ such that the ordered representation in Theorem~\ref{thm:intro-branches} holds on every component of its complement. The spectral set $\mathcal S(d)$ is a union of at most $L_n$ sets, each the graph of a continuous definable function on a point or an interval.
\end{theorem}

\begin{proof}[Proof of Theorem~\ref{thm:intro-branches}]
Apply Proposition~\ref{prop:ominimal-facts}(iii) to the definable set \eqref{eq:p-spectral-set}. After removing finitely many exponents, the restriction over each remaining component is the disjoint union of the graphs of a fixed finite number of continuous definable functions, which may be ordered increasingly.
\end{proof}

\begin{proof}[Proof of Theorem~\ref{thm:intro-uniform-branches}]
After replacing $\mathfrak B_n^{\mathrm{unif}}$ by its ceiling, define, for $0\leq k\leq\mathfrak B_n^{\mathrm{unif}}$,
$$
 K_k=\{(d,p)\in D_n\times(1,\infty):\#\sigma_p(d)=k\}.
$$
Each $K_k$ is definable: one quantifies $k$ ordered, pairwise distinct values in the fibre of \eqref{eq:total-spectrum}, and asserts that they exhaust that fibre. Let $\partial_pK_k$ denote the set of $(d,p)$ for which $p$ belongs to the relative boundary of the one-dimensional fibre $(K_k)_d$, and set
$$
 \mathcal B_n=
 \bigcup_{k=0}^{\mathfrak B_n^{\mathrm{unif}}}\partial_pK_k.
$$
The condition that $p$ lie in the boundary of $(K_k)_d$ relative to $(1,\infty)$ is first-order definable. Moreover, $(\mathcal B_n)_d=\operatorname{Bif}_{\#}(d)$. Indeed, each $(K_k)_d$ is a finite union of points and intervals, so its relative boundary is finite. Hence $(\mathcal B_n)_d$ is finite, and uniform finiteness applied to $\mathcal B_n$ gives $C_n$.

For the stronger assertion, take a cylindrical cell decomposition compatible with the total spectral set \eqref{eq:total-spectrum}, in the coordinate order
$
 d, p,\lambda.
$
Because every $(d,p)$-fibre is finite, a cell contained in $\mathcal S_n$ cannot contain an interval in the $\lambda$-direction; each such cell is therefore the graph of a definable function of $(d,p)$. The decomposition induces a cylindrical decomposition of the $(d,p)$-base. Let $P\subset D_n\times(1,\infty)$ be the union of the cells that are graphs of single $p$-values over cells of $D_n$. For fixed $d$, the fibre $P_d$ is precisely the collection of point cells separating the open interval cells in the induced one-dimensional decomposition, so the fixed decomposition bounds $\#P_d$ uniformly. On each component of $(1,\infty)\setminus P_d$, the spectral cells restrict to disjoint continuous definable graphs ordered by their $\lambda$ coordinate. Their number is constant there, so $\operatorname{Bif}_{\#}(d)\subset P_d$. The fixed total number of cells also bounds, uniformly in $d$, the number of spectral branches with interval domains. Including the spectral points over the domains $p\in P_d$ and enlarging the bound if necessary gives $L_n$ and the asserted decomposition of $\mathcal S(d)$.
\end{proof}

\begin{corollary}[Eventual constancy of the spectral count]
\label{cor:eventual-count}
For fixed $d$, there exist $1<p_-<p_+<\infty$ such that $\#\sigma_p(d)$ is constant on $(1,p_-)$ and constant on $(p_+,\infty)$. More generally, there is a finite set outside of which the spectral count is locally constant on $(1,\infty)$.
\end{corollary}

\begin{proof}
This follows from Theorem~\ref{thm:intro-branches} and the finiteness of its exceptional set.
\end{proof}

\begin{corollary}[Definably varying graph data]\label{cor:definable-data}
Let $I\subset(1,\infty)$ be an interval, and suppose
$$
 p\mapsto (a(p),\varepsilon(p),b(p),\nu(p))
$$
is definable in $\R_{\exp}$, with admissible coefficients for every $p\in I$. Then the corresponding spectral set is a finite union of continuous definable branches away from finitely many points of $I$.
\end{corollary}

\begin{proof}
Pull back $\mathcal{S}_n$ along the definable parameter curve and use Proposition~\ref{prop:ominimal-facts}(iii).
\end{proof}

\begin{example}[The convention $a_{ij}(p)=\omega_{ij}^p$]
For fixed nonnegative edge weights $\omega_{ij}$ and fixed signature, the map $p\mapsto\omega_{ij}^p$ is definable by \eqref{eq:pow-def}. Therefore Corollary~\ref{cor:definable-data} applies to the weighted graph convention used in \cite{DeiddaTudiscoZhang2025}.
\end{example}

\begin{remark}[Regularity of branches]
By the monotonicity and $C^r$ cell-decomposition theorems, each branch in Theorem~\ref{thm:intro-branches} can be further partitioned into finitely many intervals on which it is $C^r$, for any prescribed finite $r$, and is either constant or strictly monotone. This gives a finite piecewise-$C^r$ description of every branch.
\end{remark}

The preceding branch decomposition controls the spectral values; we next lift it to the topology of the corresponding normalised eigenvector sets.

\subsection{Topology of normalised eigenvector sets}\label{sec:topology}

For admissible $d=(a,\varepsilon,b,\nu)$, define
$$
 E(d,p,\lambda)
 =\{x\in\Sph^{n-1}:(d,p,x,\lambda)\in\mathcal Z_n\}.
$$
This compact definable set may be singular, disconnected, or positive-dimensional. For $\eta\in\{-1,0,1\}^n\setminus\{0\}$, the corresponding sign stratum is defined by $\sgn(x_i)=\eta_i$ for every $i$. The corresponding partition into strong nodal domains is the component partition of the graph with vertex set $\{i:\eta_i\ne0\}$ and edges
\begin{equation}\label{eq:signed-nodal-edges}
 a_{ij}>0,\qquad \eta_i\varepsilon_{ij}\eta_j>0.
\end{equation}
Thus each prescribed sign vector, and each prescribed partition into strong nodal domains, is determined by finitely many definable conditions.

\begin{theorem}[Uniform topological triviality]\label{thm:eigenset-hardt}
For every $n\geq1$ there is a non-effective constant $T_n<\infty$ with the following properties. For each admissible $d$ there is a set $P_d\subset(1,\infty)$, $\#P_d\leq T_n$, such that on every component $I$ of its complement
$$
 \sigma_p(d)=\{\lambda_1(p)<\cdots<\lambda_r(p)\},
$$
with continuous definable functions $\lambda_k$. For every $p_0\in I$ and $1\leq k\leq r$ there is a definable homeomorphism over $I$,
\begin{equation}\label{eq:eigenset-trivialization}
 I\times E(d,p_0,\lambda_k(p_0))
 \rightarrow
 \{(p,x):p\in I,\ x\in E(d,p,\lambda_k(p))\}.
\end{equation}
It may be chosen compatible with every sign stratum and every partition into strong nodal domains. The analogous statement holds for the projective eigenvector sets $\overline E(d,p,\lambda_k(p))$ under the realisation $[x]\mapsto xx^{\mathsf T}\in\mathcal P_n$, compatibly with the images of the two antipodal sign strata indexed by $\eta$ and $-\eta$. Taking the finite disjoint union over $k$ gives a trivialisation of the full normalised eigenpair family over $I$.

After enlarging $T_n$, the spectral set $\mathcal S(d)$ is a union of at most $T_n$ graphs over point or interval domains, and the normalised eigenvector set and its projectivisation are definably trivial over each such domain. Among all choices of $d,p,\lambda$, at most $T_n$ definable homeomorphism types of these sets occur. In particular, their dimensions, numbers of connected components, rational Betti numbers, and Euler characteristics admit bounds depending only on $n$.
\end{theorem}

\begin{proof}
Apply compatible Hardt triviality, Proposition~\ref{prop:ominimal-facts}(iv), to
$$
 \rho:\mathcal Z_n\rightarrow\mathcal S_n,\qquad
 (d,p,x,\lambda)\mapsto(d,p,\lambda),
$$
together with all sign strata and the finitely many subsets specifying every component partition in \eqref{eq:signed-nodal-edges}. Refine the resulting finite partition of $\mathcal S_n$ by a cylindrical cell decomposition in the coordinate order $d,p,\lambda$. Since the $(d,p)$-spectral fibres are finite, no cell in $\mathcal S_n$ contains an interval in the $\lambda$-direction; every cell is the graph of a definable function of $(d,p)$. Restricting these cells and their Hardt trivialisations to fixed $d$ gives \eqref{eq:eigenset-trivialization} after removing the finitely many projected cell boundaries. The number of boundaries and graph pieces is bounded by the fixed global decomposition, uniformly in $d$.

For the quotient, take the definable image of $\mathcal Z_n$ under $(d,p,x,\lambda)\mapsto(d,p,xx^{\mathsf T},\lambda)$ and repeat the argument, compatibly with the images of the unions of the two antipodal sign strata indexed by $\eta$ and $-\eta$. A common refinement gives one $P_d$. The global Hardt partitions have only finitely many cells, so only finitely many fibre homeomorphism types occur. Taking the maximum of these finitely many bounds defines $T_n$.
\end{proof}

\begin{corollary}[Topology for definably varying graph data]
\label{cor:definable-topology}
Let $I\subset(1,\infty)$ be an interval, and suppose that
$$
 p\mapsto d(p)=(a(p),\varepsilon(p),b(p),\nu(p))
$$
is definable in $\R_{\exp}$, with admissible coefficients for every $p\in I$. There is a finite set $P\subset I$ such that, on every component $J$ of $I\setminus P$,
$$
 \sigma_p(d(p))=\{\lambda_1(p)<\cdots<\lambda_r(p)\}
$$
with continuous definable branches. For every $k$, both families
$$
 \{(p,x):p\in J,\ x\in E(d(p),p,\lambda_k(p))\}
$$
and their projective quotients are definably trivial over $J$. The first trivialisation may be chosen compatible with all sign strata and all partitions into strong nodal domains.
\end{corollary}

\begin{proof}
Pull back the map $\rho$ and its projective image along the definable curve $p\mapsto(d(p),p)$. Compatible Hardt triviality, followed by a one-dimensional cell refinement, gives the assertion.
\end{proof}

\begin{remark}[Antipodes, weak nodal domains, and switching]
The realisation $[x]\mapsto xx^{\mathsf T}\in\mathcal P_n$ yields the projective trivialisation. Only finitely many patterns of weak nodal domains, in the sense of \cite[Definition~2.8]{GeLiuZhang2023}, can occur. Applying Hardt triviality compatibly with the definable set on which each pattern occurs gives the same conclusion for these patterns. If $\tau_i\in\{-1,1\}$, replace $\varepsilon_{ij}$ by $\tau_i\varepsilon_{ij}\tau_j$ and $x_i$ by $\tau_i x_i$. The corresponding families of eigenvector sets are definably homeomorphic, while \eqref{eq:signed-nodal-edges} is unchanged. Thus the topological trivialisation is switching invariant on each interval of the decomposition.
\end{remark}

\subsection{Global definable stratification}\label{sec:global-stratification}

The preceding one-parameter decompositions are restrictions of a finite definable partition of the full parameter space.
For $\theta=(d,p)\in\Theta_n$, let $E(\theta,\lambda)=E(d,p,\lambda)$.

\begin{theorem}[Global definable stratification]
\label{thm:global-stratification}
For every $n\geq1$ there is a finite definable cell partition
$
 \Theta_n=\bigsqcup_\alpha C_\alpha
$
such that, for every $\alpha$, there are continuous definable functions
$$
 \lambda_{\alpha,1}<\cdots<\lambda_{\alpha,r_\alpha}
 \quad\hbox{on }C_\alpha
$$
with
$$
 \sigma_p(d)=
 \{\lambda_{\alpha,1}(d,p),\ldots,\lambda_{\alpha,r_\alpha}(d,p)\}
 \qquad ((d,p)\in C_\alpha).
$$
After a finite refinement, for every $k$ and every base point $\theta_\alpha\in C_\alpha$ there is a definable homeomorphism over $C_\alpha$,
$$
 \mathcal T_{\alpha,k}:
 C_\alpha\times\mathbb S^{n-1}
 \rightarrow C_\alpha\times\mathbb S^{n-1},
$$
which carries the family
$
 \{(\theta,x):x\in
 E(\theta,\lambda_{\alpha,k}(\theta))\}
$
onto
$
 C_\alpha\times
 E(\theta_\alpha,\lambda_{\alpha,k}(\theta_\alpha)).
$
The homeomorphism may be chosen compatible with any prescribed finite family of definable subsets, including all sign strata and all partitions into strong nodal domains. The analogous statement holds under the projector realisation $[x]\mapsto xx^{\mathsf T}\in\mathcal P_n$ of $\mathbb{R}P^{n-1}$.
\end{theorem}

\begin{proof}
Apply cylindrical cell decomposition to the total spectral set in the coordinate order $(d,p),\lambda$. After refining the induced partition of $\Theta_n$, each spectral cell is the graph of a continuous definable function over an entire base cell.

Fix a resulting base cell $C_\alpha$. For each spectral branch, set
$$
 A_{\alpha,k}
 =\{(\theta,x)\in C_\alpha\times\mathbb S^{n-1}:x\in E(\theta,\lambda_{\alpha,k}(\theta))\}.
$$
These are definable subsets of the ambient product. Apply compatible definable Hardt triviality to the product projection
$
 \pi_\alpha:C_\alpha\times\mathbb S^{n-1}\rightarrow C_\alpha,
$
with the finite family consisting of all $A_{\alpha,k}$, all sign strata, and all subsets encoding the partitions into strong nodal domains. After a finite refinement of $C_\alpha$, the Hardt homeomorphism trivialises the whole fibre $\mathbb S^{n-1}$ and every one of these subsets simultaneously. It is therefore an ambient homeomorphism of pairs of the form asserted in the theorem. Taking the common refinement over the finitely many base cells proves the sphere statement.

For the projective assertion, apply the same construction to the definable image under $x\mapsto xx^{\mathsf T}$, compatibly with the images of the unions of antipodal sign strata indexed by $\eta$ and $-\eta$. This gives the claimed definable topological trivialisation under the projector realisation.
\end{proof}

\section{Endpoint limits and Clarke critical values}
\label{app:endpoints}

We determine the endpoint limits of the full finite-$p$ spectra and eigenpair sets, and then prove finiteness of the Clarke critical-value sets of the limiting quotients restricted to the sphere.
\begin{corollary}[Endpoint Hausdorff limits of the full spectra]
\label{cor:endpoint-Hausdorff}
Fix graph data $d=(a,\varepsilon,b,\nu)$ on $n$ vertices. Set
$$
 \beta=\min_i\frac{b_i}{\nu_i},
 \qquad
 \widetilde b_i=b_i-\beta\nu_i\geq0,
 \qquad
 \delta_i=\sum_{j\ne i}a_{ij},
$$
and define the set of shifted $p$-th roots of the spectral values by
$$
 \widehat\sigma_p(d)
 =
 \{(\lambda-\beta)^{1/p}:\lambda\in\sigma_p(d)\}.
$$
There are non-empty finite sets
$$
 \Lambda_1(d)\subset\R,\qquad
 \Lambda_\infty(d)\subset[0,2]
$$
such that, for the Hausdorff distance between non-empty compact subsets of $\R$,
$$
 d_H\bigl(\sigma_p(d),\Lambda_1(d)\bigr)\rightarrow0
 \qquad(p\downarrow1),
$$
and
$$
 d_H\bigl(\widehat\sigma_p(d),\Lambda_\infty(d)\bigr)
 \rightarrow0
 \qquad(p\to\infty).
$$
Moreover,
$$
 \#\Lambda_1(d),\ \#\Lambda_\infty(d)
 \leq\widehat{\mathfrak B}_n.
$$

For $x\ne0$, define
$$
 Q_{1,d}(x)=
 \frac{\displaystyle
 \sum_{i<j}a_{ij}|x_i-\varepsilon_{ij}x_j|
 +\sum_i b_i|x_i|}
 {\displaystyle\sum_i\nu_i|x_i|}
$$
and
$$
 R_{\infty,d}(x)=
 \frac{
 \max\!\left(
 \{\,|x_i-\varepsilon_{ij}x_j|:a_{ij}>0\,\}
 \cup
 \{\,|x_i|:\widetilde b_i>0\,\}
 \cup\{0\}\right)}
 {\max_i|x_i|}.
$$
Then
$$
 \min\Lambda_1(d)=\min_{x\ne0}Q_{1,d}(x),
 \qquad
 \max\Lambda_1(d)=\max_{x\ne0}Q_{1,d}(x),
$$
and
$$
 \min\Lambda_\infty(d)=\min_{x\ne0}R_{\infty,d}(x),
 \qquad
 \max\Lambda_\infty(d)=\max_{x\ne0}R_{\infty,d}(x).
$$
In particular,
$$
 \Lambda_1(d)\subseteq
 \left[
 \beta,\
 \beta+\max_i\frac{\delta_i+\widetilde b_i}{\nu_i}
 \right],
$$
while
$$
 \max\Lambda_\infty(d)=
 \begin{cases}
 2,&a_{ij}>0\text{ for some }i<j,\\
 1,&a\equiv0\text{ and }\widetilde b\not\equiv0,\\
 0,&a\equiv0\text{ and }\widetilde b\equiv0.
 \end{cases}
$$
\end{corollary}

\begin{proof}
For $p$ sufficiently close to $1$, the branch theorem and Corollary~\ref{cor:eventual-count} represent the spectrum by a fixed number of ordered continuous definable functions
$$
 \lambda_1(p)<\cdots<\lambda_r(p).
$$
For every eigenvalue, the Rayleigh quotient and $|u-v|^p\leq2^{p-1}(|u|^p+|v|^p)$ give
\begin{equation}\label{eq:endpoint-rayleigh-bound}
 \beta\leq\lambda_j(p)
 \leq
 \beta+\max_i
 \frac{2^{p-1}\delta_i+\widetilde b_i}{\nu_i}.
\end{equation}
Each branch is bounded and definable as $p\downarrow1$, and therefore has a finite one-sided limit. Taking the set of these branch limits gives $\Lambda_1(d)$ and Hausdorff convergence. Collisions at the endpoint can only reduce cardinality, so $\#\Lambda_1(d)\leq r\leq\widehat{\mathfrak B}_n$.

For $p$ sufficiently large, invoke the branch theorem separately and let
$$
 \sigma_p(d)=
 \{\lambda^\infty_1(p)<\cdots<\lambda^\infty_s(p)\}.
$$
For these branches, define
$
 \rho^\infty_j(p)=(\lambda^\infty_j(p)-\beta)^{1/p}.
$
These functions are definable and, by \eqref{eq:endpoint-rayleigh-bound}, satisfy
$$
 0\leq\rho^\infty_j(p)
 \leq
 \left(
 \max_i\frac{2^{p-1}\delta_i+\widetilde b_i}{\nu_i}
 \right)^{1/p}.
$$
They are bounded and hence have finite limits as $p\to\infty$. Their limit set is $\Lambda_\infty(d)$, and the same branch argument proves Hausdorff convergence and the cardinality bound.

It remains to identify the extreme points. On $\Sph^{n-1}$ one has uniform convergence
$$
 \cQ_{p,a,\varepsilon,b,\nu}\rightarrow Q_{1,d}
 \qquad(p\downarrow1).
$$
The minimum and maximum of the finite-$p$ Rayleigh quotient are eigenvalues, so uniform convergence gives the two assertions concerning $\Lambda_1(d)$.

After the spectral shift, set
\begin{equation}\label{eq:shifted-rayleigh}
 \widetilde Q_p(x)=
 \frac{\displaystyle
 \sum_{i<j}a_{ij}|x_i-\varepsilon_{ij}x_j|^p
 +\sum_i\widetilde b_i|x_i|^p}
 {\displaystyle\sum_i\nu_i|x_i|^p}.
\end{equation}
The elementary convergence of finite weighted $p$-norms to the maximum norm is uniform on $\Sph^{n-1}$ and gives
$$
 \widetilde Q_p(x)^{1/p}\rightarrow R_{\infty,d}(x)
 \qquad(p\to\infty)
$$
uniformly in $x$. Since the extreme values of $\widetilde Q_p$ are spectral values, this proves the corresponding assertions for $\Lambda_\infty(d)$. Moreover, $|x_i-\varepsilon_{ij}x_j|\leq2\max_h|x_h|$, and equality can be realised on any positive-weight edge by taking $x_i=1$ and $x_j=-\varepsilon_{ij}$. This proves the displayed formula for $\max\Lambda_\infty(d)$.
\end{proof}

\begin{remark}
Corollary~\ref{cor:endpoint-Hausdorff} describes the Hausdorff limits of the entire finite-$p$ spectrum. Comparison with a set-valued endpoint spectrum depends on the chosen nonsmooth eigenproblem. Relevant frameworks include the graph $1$-Laplacian of \cite{Chang2016}, the $p\downarrow1$ limits of first Dirichlet eigenfunctions in \cite{GeHuaJiang2021}, and the graph $\infty$-eigenpairs of \cite{DeiddaBurgerPuttiTudisco2026}.
\end{remark}

To pass from spectral values to full eigenpair sets, we use a compactness principle for definable families.

\begin{lemma}[Definable Hausdorff limits]
\label{lem:definable-Hausdorff}
Let $(K,d)$ be a compact definable metric space with definable metric, let $I=(a,b)$, where the endpoints may be extended real numbers, and let $X\subset I\times K$ be definable. If every fibre $X_t$ is non-empty and compact, then there are non-empty compact definable sets $X_{a+},X_{b-}\subset K$ such that
$$
 X_t\rightarrow X_{a+}\quad(t\downarrow a),\qquad
 X_t\rightarrow X_{b-}\quad(t\uparrow b)
$$
in Hausdorff distance.
\end{lemma}

\begin{proof}
For $z\in K$ set $d_t(z)=\dist(z,X_t)$. The function $(t,z)\mapsto d_t(z)$ is definable, and the family is uniformly $1$-Lipschitz. For each fixed $z$, the function $t\mapsto d_t(z)$ is bounded and definable, and therefore has a unique one-sided limit $d_*(z)$. The limit is definable and $1$-Lipschitz. If the convergence were not uniform, there would be $t_j\downarrow a$ and $z_j\to z$ for which $|d_{t_j}(z_j)-d_*(z_j)|$ stays bounded away from zero. The $1$-Lipschitz bounds and pointwise convergence at $z$ give a contradiction. Thus $d_t\to d_*$ uniformly.

The hyperspace of non-empty compact subsets of $K$ is compact, and $A\mapsto d(\,\cdot\,,A)$ is an isometric embedding for the Hausdorff metric. The uniform limit is therefore the distance function of a unique non-empty compact set. That set is definable from the definable limit function. The proof at the other endpoint is the same.
\end{proof}

\begin{theorem}[Hausdorff limits of the full eigenpair sets]
\label{thm:eigenpair-endpoints}
Fix graph data $d=(a,\varepsilon,b,\nu)$ and set
$$
 \beta=\min_i\frac{b_i}{\nu_i},\qquad
 \widetilde b_i=b_i-\beta\nu_i,\qquad
 \delta_i=\sum_{j\ne i}a_{ij}.
$$
Let $\widetilde Q_p$ be the shifted quotient in \eqref{eq:shifted-rayleigh}. Define
$$
 \mathcal E_p(d)=
 \{(x,\lambda)\in\mathbb S^{n-1}\times\mathbb R:(x,\lambda)\text{ is a }p\text{-eigenpair}\}
$$
and
$$
 \widehat{\mathcal E}_p(d)=
 \{(x,\rho):(x,\beta+\rho^p)\in\mathcal E_p(d),\ \rho\geq0\}.
$$
There are non-empty compact definable sets
$$
 \mathcal E^{\rm lim}_1(d)
 \subset\mathbb S^{n-1}\times\mathbb R,\qquad
 \mathcal E^{\rm lim}_\infty(d)
 \subset\mathbb S^{n-1}\times[0,2]
$$
such that
$$
 \mathcal E_p(d)\rightarrow\mathcal E^{\rm lim}_1(d)
 \quad(p\downarrow1),\qquad
 \widehat{\mathcal E}_p(d)\rightarrow
 \mathcal E^{\rm lim}_\infty(d)\quad(p\to\infty)
$$
in Hausdorff distance. The same assertions hold after projectivising the eigenvectors. Projection to the value coordinate gives the spectral limit sets in Corollary~\ref{cor:endpoint-Hausdorff}.
\end{theorem}

\begin{proof}
Every fibre $\mathcal E_p(d)$ is non-empty: the minimum and maximum of the $C^1$ Rayleigh quotient on $\mathbb S^{n-1}$ are eigenvalues. It is closed, and the value coordinate equals the Rayleigh quotient, so the fibre is compact.

Fix $p_0>1$. For $1<p\leq p_0$, \eqref{eq:endpoint-rayleigh-bound} gives
$$
 \beta\leq\lambda\leq
 \beta+\max_i
 \frac{2^{p_0-1}\delta_i+\widetilde b_i}{\nu_i}.
$$
The normalised eigenpair family is definable and lies in the resulting fixed compact cylinder. Lemma~\ref{lem:definable-Hausdorff} therefore gives the limit as $p\downarrow1$.

After the spectral shift, every transformed eigenpair satisfies
\begin{equation}\label{eq:root-eigenpair-identity}
 \rho=(\lambda-\beta)^{1/p}
 =\widetilde Q_p(x)^{1/p}.
\end{equation}
Set
$$
 C=\max\left\{1,\max_i\frac{\delta_i+\widetilde b_i}{\nu_i}\right\}.
$$
For $p\geq p_0$, \eqref{eq:endpoint-rayleigh-bound} implies
$$
 0\leq\rho\leq
 2^{1-1/p}C^{1/p}
 \leq2\max\{1,C^{1/p_0}\}.
$$
Thus the definable family $\widehat{\mathcal E}_p(d)$ also lies in a fixed compact cylinder, and the lemma gives a Hausdorff limit as $p\to\infty$. The uniform convergence $\widetilde Q_p^{\,1/p}\to R_{\infty,d}$, together with \eqref{eq:root-eigenpair-identity}, shows that every limiting point $(x,\rho)$ satisfies
$
 \rho=R_{\infty,d}(x)\in[0,2].
$
Hence $\mathcal E^{\rm lim}_\infty(d)\subset\mathbb S^{n-1}\times[0,2]$, as claimed.

Projectivisation of the sphere and projection to the value coordinate are continuous on the compact ambient cylinders. They therefore preserve Hausdorff convergence and give the remaining assertions.
\end{proof}

\begin{corollary}[Exact endpoint limits for the uniformly weighted $K_n$]
\label{cor:uniform-endpoint-limits}
Let $n\geq2$, and write $\mathbf1_S$ for the indicator of $S\subset[n]$. Let $\mathcal Z_p^{\rm proj}(K_n)$ be the set of all projective eigenpairs $([x],\lambda)$ of the uniformly weighted $K_n$, and set
$$
 \widehat{\mathcal Z}_p^{\rm proj}(K_n)
 =\{([x],\lambda^{1/p}):([x],\lambda)\in\mathcal Z_p^{\rm proj}(K_n)\}.
$$
Then, in the Hausdorff metric on the indicated compact projective cylinders,
$$
 \mathcal Z_p^{\rm proj}(K_n)
 \rightarrow\mathcal L_{1,n}
 \qquad(p\downarrow1),
$$
where, in each of the two families indexed by $(P,N)$ below, $(P,N)$ and $(N,P)$ determine the same element,
\begin{align*}
 \mathcal L_{1,n}
 ={}&\{([\one],0)\}\\
 &\cup\{([\mathbf1_S],n-|S|):1\leq|S|<n/2\}\\
 &\cup\{([\mathbf1_P-\mathbf1_N],n-|P|):P\cap N=\varnothing,\ |P|=|N|\geq1\},
\end{align*}
and
$$
 \widehat{\mathcal Z}_p^{\rm proj}(K_n)
 \rightarrow\mathcal L_{\infty,n}
 \qquad(p\to\infty),
$$
where
$$
 \mathcal L_{\infty,n}
 =\{([\one],0)\}
 \cup\{([\mathbf1_P-\mathbf1_N],2):P,N\ne\varnothing,\ P\cap N=\varnothing\}.
$$
Consequently the spectral limit sets in Corollary~\ref{cor:endpoint-Hausdorff} are
\begin{equation}\label{eq:uniform-endpoint-spectra}
 \Lambda_1(K_n)
 =\{0\}\cup\{n-j:1\leq j\leq\lfloor n/2\rfloor\},
 \quad
 \Lambda_\infty(K_n)=\{0,2\}.
\end{equation}
In particular,
$$
 \#\Lambda_1(K_n)=\lfloor n/2\rfloor+1,
 \qquad
 \#\Lambda_\infty(K_n)=2,
$$
while
\begin{align*}
 \#\mathcal L_{1,n}
 &=1+\sum_{1\leq j<n/2}\binom nj+\frac12\sum_{j=1}^{\lfloor n/2\rfloor}\binom nj\binom{n-j}{j},\\
 \#\mathcal L_{\infty,n}&=R_n.
\end{align*}
More precisely, for $1\leq j\leq\lfloor n/2\rfloor$,
\begin{equation}\label{eq:uniform-endpoint-fibres}
 \#\{([x],\lambda)\in\mathcal L_{1,n}:\lambda=n-j\}
 =\mathbf1_{\{2j<n\}}\binom nj
 +\frac12\binom nj\binom{n-j}{j},
\end{equation}
and
$$
 \#\{([x],\lambda)\in\mathcal L_{\infty,n}:\lambda=2\}
 =R_n-1.
$$
Moreover,
$$
 \#\mathcal L_{1,n}
 \sim\frac{\sqrt3}{4\sqrt{\pi n}}\,3^n,
 \qquad
 \#\mathcal L_{\infty,n}
 \sim\frac12\,3^n.
$$
Thus the uniform endpoint spectra have only $O(n)$ values, while their projective eigenpair limit sets have cardinality $3^{\,n-o(n)}$.
\end{corollary}

\begin{proof}
For a label $(P,N,Z)$, set $k=|P|$, $\ell=|N|$, $r=|Z|$, and $q=p-1$. A representative of the corresponding eigenline of the uniformly weighted $K_n$ is
$
 x_{P,N}(p)=\ell^{1/q}\mathbf1_P-k^{1/q}\mathbf1_N.
$
After interchanging $P$ and $N$, assume $k\leq\ell$. If $k<\ell$, division by $\ell^{1/q}$ gives
$$
 [x_{P,N}(p)]\rightarrow[\mathbf1_P]
 \qquad(p\downarrow1).
$$
If $k=\ell$, the eigenline is already $[\mathbf1_P-\mathbf1_N]$. Proposition~\ref{prop:uniform-entropy} gives the limiting value $n-k$ in both cases. This produces exactly $\mathcal L_{1,n}$.

As $p\to\infty$, both $k^{1/q}$ and $\ell^{1/q}$ tend to one, so
$
 [x_{P,N}(p)]\rightarrow[\mathbf1_P-\mathbf1_N].
$
Moreover,
$
 2^{1-p}\lambda_{k,\ell,r}(p)\rightarrow\sqrt{k\ell}>0,
$
and therefore $\lambda_{k,\ell,r}(p)^{1/p}\to2$. The constant eigenline has value zero. This proves the asserted limit set at infinity.

There are only finitely many labels $[P\mid N]$, so convergence of every labelled branch is equivalent to Hausdorff convergence of the full finite sets. Projecting to the value coordinate gives \eqref{eq:uniform-endpoint-spectra}. The count at infinity is $\#\mathcal L_{\infty,n}=1+\#\mathcal C_n=R_n$, with the additional point given by the constant eigenline. At the lower endpoint, the first sum counts indicator eigenlines and the second counts unordered pairs of disjoint equal-size sets. The same decomposition at a fixed value $n-j$ gives \eqref{eq:uniform-endpoint-fibres}; every nonconstant point of $\mathcal L_{\infty,n}$ has value $2$.

For the asymptotics, set
$$
 \tau_n=[z^0](1+z+z^{-1})^n
 =\sum_{j=0}^{\lfloor n/2\rfloor}\binom nj\binom{n-j}{j}.
$$
The balanced part of $\mathcal L_{1,n}$ is $(\tau_n-1)/2$, while the indicator part is at most $2^n$. Fourier inversion gives
$$
 \tau_n=\frac1{2\pi}\int_{-\pi}^{\pi}(1+2\cos t)^n\,dt.
$$
Since $1+2\cos t=3-t^2+O(t^4)$ at the unique point of maximal modulus, Laplace's method yields
$$
 \tau_n\sim\frac{\sqrt3}{2\sqrt{\pi n}}\,3^n.
$$
This proves the first asymptotic; the second follows from $R_n=(3^n-2^{n+1}+3)/2$.
\end{proof}

The Hausdorff limits describe limits of finite-$p$ eigenpairs but do not select an intrinsic nonsmooth endpoint eigenproblem. We therefore turn to the Clarke critical values of the limiting quotients. If $M$ is a $C^1$ manifold, $x=\phi(u)$ in a $C^1$ chart $\phi:U\to M$, and $f$ is locally Lipschitz, define
$$
 \partial_M^Cf(x)
 =(D\phi_u^*)^{-1}\partial^C(f\circ\phi)(u)
 \subset T_x^*M.
$$
The Clarke chain rule for $C^1$ changes of coordinates \cite[Chapter~2]{Clarke1990} makes this definition independent of the chart. In particular, $0\in\partial_M^Cf(x)$ if and only if $0\in\partial^C(f\circ\phi)(u)$.

\begin{theorem}[Finiteness of Clarke critical values for definable functions]
\label{thm:clarke-values}
Let $M$ be a compact definable $C^1$ manifold and let $f:M\to\mathbb R$ be definable and locally Lipschitz. Its set of Clarke critical values on $M$
$$
 \operatorname{CV}_C(f)
 =\{f(x):0\in\partial^C_Mf(x)\}
$$
is finite. More generally, in a definable family of locally Lipschitz functions on $M$, the cardinalities of these sets are uniformly bounded.
\end{theorem}

\begin{proof}
Choose finitely many definable $C^1$ charts $\phi_j:U_j\to M$ and definable open sets $V_j\Subset U_j$ whose images cover $M$. The function $f\circ\phi_j$ is definable and locally Lipschitz on $U_j$. Restrict it to $\overline V_j$ and extend it by $+\infty$ outside $\overline V_j$. The resulting function on Euclidean space is definable and lower semicontinuous because $\overline V_j\Subset U_j$ and $f\circ\phi_j$ is continuous on $\overline V_j$. At every point of $V_j$ it agrees locally with $f\circ\phi_j$, so its Clarke subdifferential agrees there with that of $f\circ\phi_j$. By \cite[Corollary~9(ii)]{BolteDaniilidisLewisShiota2007}, only finitely many Clarke critical values are attained at points of $V_j$. The sets $\phi_j(V_j)$ cover $M$, so their finite union contains $\operatorname{CV}_C(f)$.

Now let $F:T\times M\to\mathbb R$ be a definable family whose fibres $F_t$ are locally Lipschitz. In charts, the limiting gradients in the fibre over $t$ are obtained by closing, with $t$ fixed, the definable graph of $\nabla_xF$ over the set where $F$ is differentiable with respect to $x$ \cite[Chapter~2]{Clarke1990}. The graph of $\partial_x^CF_t(x)$ is the fibrewise convex hull of this closure. By Carath\'eodory's theorem, the condition $0\in\partial_x^CF_t(x)$ can be expressed using at most $\dim M+1$ limiting gradients and therefore has a finite first-order description. Thus the graph of the fibrewise Clarke subdifferential is definable. Consequently
$$
 \mathcal V
 =\{(t,F(t,x)):x\in M,\ 0\in\partial_M^CF_t(x)\}
$$
is definable. Every fibre $\mathcal V_t$ is finite by the first part, so o-minimal uniform finiteness gives a common bound for all $t$.
\end{proof}

\begin{corollary}[Endpoint Clarke critical values]\label{cor:endpoint-clarke}
The Clarke critical-value sets of the restrictions of $Q_{1,d}$ and $R_{\infty,d}$ to $\mathbb S^{n-1}$ are finite, uniformly in the graph data in fixed dimension.
\end{corollary}

\begin{proof}
The two endpoint quotients form definable families in the graph data. For each fixed datum they are locally Lipschitz on the Euclidean sphere: their numerators are finite sums or finite maxima of absolute values of linear forms, while
$$
 \sum_i\nu_i|x_i|\geq\min_i\nu_i>0,
 \qquad
 \max_i|x_i|\geq n^{-1/2}
 \quad(x\in\mathbb S^{n-1}).
$$
The family assertion in Theorem~\ref{thm:clarke-values} therefore gives finiteness and a uniform bound depending only on the dimension.
\end{proof}

\begin{proof}[Proof of Corollary~\ref{cor:intro-global-consequences}]
The global stratification assertion is Theorem~\ref{thm:global-stratification}. The endpoint conclusions are Theorem~\ref{thm:eigenpair-endpoints} and Corollary~\ref{cor:endpoint-clarke}.
\end{proof}

\section{Linear-form and matrix \texorpdfstring{$p$}{p}-eigenvalue problems}
\label{app:matrix}

The same effective argument applies to Rayleigh quotients whose numerator and denominator are $p$-homogeneous sums of absolute values of linear forms. Let $A\in\R^{m\times n}$ and $C\in\R^{\ell\times n}$, let $\mu\in\R^m$ and $\nu\in(0,\infty)^\ell$, and suppose that $C$ is injective. Define
\begin{align}
 F_{A,\mu,p}(x)
 &=\sum_{r=1}^m\mu_r|(Ax)_r|^p,\label{eq:matrix-numerator}\\
 G_{C,\nu,p}(x)
 &=\sum_{s=1}^{\ell}\nu_s|(Cx)_s|^p.\label{eq:matrix-mass}
\end{align}
The corresponding generalised $p$-eigenvalue problem for the matrix pair $(A,C)$ is
\begin{equation}\label{eq:matrix-eigen}
 A^{\mathsf T}\bigl(\mu\od\Phi_p(Ax)\bigr)
 =\lambda C^{\mathsf T}\bigl(\nu\od\Phi_p(Cx)\bigr),
 \qquad x\neq0.
\end{equation}
The numerator coefficients may have either sign. Injectivity of $C$ makes the denominator in the associated Rayleigh quotient positive away from the origin.

\begin{theorem}[Effective bounds for matrix $p$-eigenvalue problems]
\label{thm:matrix-finite}
For every $A,C,\mu,\nu$ as above and every $p>1$, let $\Sigma_{A,C,\mu,\nu,p}$ be the set of real $\lambda$ for which \eqref{eq:matrix-eigen} has a nonzero solution. Then
\begin{equation}\label{eq:matrix-explicit-bound}
 \#\Sigma_{A,C,\mu,\nu,p}
 \leq\mathfrak B(m+\ell)
 \leq3^{m+\ell}2^{2(m+\ell)^2+m+\ell+2}\bigl(4(m+\ell)+3\bigr)^{2(m+\ell)}.
\end{equation}
In particular, the logarithm of the bound is $O((m+\ell)^2)$.

If
$$
 \begin{gathered}
 E_{A,C}(p,\lambda)
 =\{x\in\Sph^{n-1}:\eqref{eq:matrix-eigen}\text{ holds}\},\\
 \overline E_{A,C}(p,\lambda)
 =\pi_{\rm proj}\bigl(E_{A,C}(p,\lambda)\bigr),
 \end{gathered}
$$
then
$$
 \sum_{\lambda\in\Sigma_{A,C,\mu,\nu,p}}b_0\bigl(\overline E_{A,C}(p,\lambda)\bigr)
 \leq
 \sum_{\lambda\in\Sigma_{A,C,\mu,\nu,p}}b_0\bigl(E_{A,C}(p,\lambda)\bigr)
 \leq\mathfrak B(m+\ell).
$$
\end{theorem}

\begin{proof}
Apply Theorem~\ref{thm:linear-form-effective} to the $m$ row functionals of $A$ and the $\ell$ row functionals of $C$, assigning numerator coefficient $\mu_r$ and denominator coefficient zero to each row of $A$, and numerator coefficient zero and denominator coefficient $\nu_s$ to each row of $C$. The denominator is positive because $C$ is injective. This gives the first inequality in \eqref{eq:matrix-explicit-bound} and the component estimate. The displayed closed-form estimate follows from $\sum_{q=1}^{m+\ell}\binom{m+\ell}{q}2^q\leq3^{m+\ell}$ exactly as in Theorem~\ref{thm:intro-main}\ref{part:intro-effective}.
\end{proof}

\begin{remark}[A bound using only the active row functionals]
Discarding zero rows and rows of $A$ with $\mu_r=0$ replaces $m+\ell$ by the number of active nonzero row functionals. Repeated forms may be retained or combined, and the rows of $A$ may be linearly dependent.
\end{remark}

\begin{remark}[Graph eigenproblems as linear-form systems]
Taking the forms $x_i$ and $x_i-\varepsilon_{ij}x_j$ recovers the numerator of the signed weighted graph Rayleigh quotient. The numerator coefficients are $b_i$ and $a_{ij}$, while the denominator coefficients are $\nu_i$ on the coordinate forms and zero on the edge forms. Thus Theorem~\ref{thm:linear-form-effective} contains arbitrary real graph potentials without a preliminary spectral shift.
\end{remark}

\begin{corollary}[Spectral branches for matrix $p$-eigenvalue problems]
\label{cor:matrix-branches}
For fixed $A,C,\mu,\nu$, the full spectrum of \eqref{eq:matrix-eigen}, as $p$ varies in $(1,\infty)$, admits the finite-branch decomposition of Theorem~\ref{thm:intro-branches}. The same conclusion holds for fixed matrix dimensions when $A(p),C(p),\mu(p),\nu(p)$ vary definably on an interval $I\subset(1,\infty)$, provided $\nu(p)\in(0,\infty)^\ell$ and $C(p)$ remains injective.
\end{corollary}

\begin{proof}
The total eigenpair and spectral sets are definable in $\R_{\exp}$. Their spectral fibres are finite by Theorem~\ref{thm:linear-form-effective}, so Proposition~\ref{prop:ominimal-facts}(iii) applies, also after pullback along a definable parameter curve.
\end{proof}

\section{Small graphs}
\label{app:small-graphs}

For two vertices the sharp bound is elementary and much smaller than the uniform estimate \eqref{eq:explicit-Bn}.

\begin{theorem}[The sharp two-vertex bound]\label{thm:two-vertex}
Let $n=2$, $a=a_{12}$, and $\varepsilon=\varepsilon_{12}$. For every $p>1$ and arbitrary admissible $b$ and $\nu$,
$
 \#\sigma_p(a,\varepsilon,b,\nu)\leq2.
$
If $a=0$, then
$$
 \sigma_p(a,\varepsilon,b,\nu)
 =\left\{\frac{b_1}{\nu_1},\frac{b_2}{\nu_2}\right\}.
$$
If $a>0$, there are exactly two distinct eigenvalues, each with one eigenline. Consequently $B_2^{\mathrm{opt}}=2$.
\end{theorem}

\begin{proof}
The assertion for $a=0$ follows from the two diagonal equations. Suppose $a>0$. Switching allows us to take $\varepsilon=1$. Neither coordinate of an eigenvector can vanish: if $x_1=0$, for example, the first equation gives $a\Phi_p(-x_2)=0$.

Set
$$
 t=\frac{x_2}{x_1},\qquad q=p-1,\qquad
 \rho=\frac{\nu_2}{\nu_1},\qquad
 \delta=\frac{b_2-\rho b_1}{a}.
$$
Eliminating $\lambda$ from the two eigenvalue equations gives
\begin{equation}\label{eq:two-vertex-H}
 H(t)=\delta,\qquad
 H(t)=\Phi_p(1-t)\left(\frac1{\Phi_p(t)}+\rho\right),\qquad t\ne0.
\end{equation}
For $t=-s<0$,
$$
 H(-s)=\rho(1+s)^q-(1+s^{-1})^q.
$$
This is strictly increasing in $s$ from $-\infty$ to $+\infty$; hence $H$ is a strictly decreasing bijection on $(-\infty,0)$. On the positive half-line,
$$
 H(t)=
 \begin{cases}
 ((1-t)/t)^q+\rho(1-t)^q,&0<t\leq1,\\
 -((t-1)/t)^q-\rho(t-1)^q,&t\geq1.
 \end{cases}
$$
It decreases from $+\infty$ to zero and then to $-\infty$. Thus \eqref{eq:two-vertex-H} has precisely two roots $t_-<0<t_+$. Conversely, each root gives an eigenpair with
$$
 \lambda(t)=\frac{a\Phi_p(1-t)+b_1}{\nu_1}.
$$
This function is strictly decreasing, so the two eigenvalues are distinct. The two eigenlines in the original signed coordinates are represented by $(1,\varepsilon t_-)$ and $(1,\varepsilon t_+)$. The choice $a=1$, $b=0$, and $\nu=(1,1)$ gives the spectrum $\{0,2^{p-1}\}$, proving sharpness.
\end{proof}

The forest theorem of \cite[Theorem~3.7]{DeiddaPuttiTudisco2023}, together with switching, already implies that a signed three-vertex graph whose support graph is acyclic has at most three distinct eigenvalues. The following sharp bound applies to balanced triangles satisfying $b_i/\nu_i=\beta$, independent of $i$, and is attained for every $p\ne2$.

Recall that a signature on a triangle is balanced if $\varepsilon_{12}\varepsilon_{23}\varepsilon_{31}=1$, or, equivalently, if it can be switched to the unsigned signature.

\begin{theorem}[The sharp bound for balanced three-vertex graphs with proportional potential]
\label{thm:balanced-triangle}
Let $p>1$ and let $n=3$. Suppose the signature is balanced and
$$
 b_i=\beta\nu_i,\qquad i=1,2,3,
$$
for some $\beta\in\R$. Then
$$
 \#\sigma_p(a,\varepsilon,b,\nu)\leq7.
$$
If $p=2$, the stronger bound
$$
 \#\sigma_2(a,\varepsilon,b,\nu)\leq3
$$
holds. If $p\neq2$ and all three edge weights are positive, there are at most six nonconstant eigenlines. The bound seven is attained in the unsigned class for every $p\ne2$; at $p=2$, the sharp unsigned bound is three. See Corollary~\ref{cor:ordinary-small-growth}.
\end{theorem}

The proof uses the following elementary root count. Set
$$
 F_{A,B,C}(s)=B(1+s)^q-A\Phi_{q+1}(s)-C,
 \qquad -1<s<\infty,
$$
where $A,B,C>0$ and $q>0$. At the right endpoint set $r=s/(1+s)$ and multiply the equation by $(1-r)^q$. The compactified equation is
\begin{equation}\label{eq:cyclic-compactification}
 B-Ar^q-C(1-r)^q=0,
 \qquad 0\leq r\leq1.
\end{equation}
Thus the left endpoint is a zero precisely when $A=C$, and the right endpoint is a zero precisely when $A=B$.

\begin{lemma}[Cyclic root lemma]\label{lem:cyclic-root}
Let $q>0$, $q\neq1$, and $X,Y,Z>0$. Consider
$$
 F_{X,Y,Z},\qquad F_{Y,Z,X},\qquad F_{Z,X,Y}
$$
on three copies of $(-1,\infty)$. Compactify as in \eqref{eq:cyclic-compactification}, and glue the right endpoint of each copy to the left endpoint of the next copy in the displayed cyclic order. If a glued endpoint is counted only once, the three functions have at most six distinct zeros.
\end{lemma}

\begin{proof}
On the negative half write $s=-u$, $0<u<1$. Then
\begin{equation}\label{eq:cyclic-negative}
 F_{A,B,C}(-u)=B(1-u)^q+Au^q-C.
\end{equation}
This is strictly convex when $q>1$ and strictly concave when $0<q<1$. On the positive half,
\begin{equation}\label{eq:cyclic-derivative}
 F'_{A,B,C}(s)
 =q\{B(1+s)^{q-1}-As^{q-1}\},
\end{equation}
so the derivative has at most one zero. Together with the endpoint signs, these convexity properties and the fact that the derivative has at most one zero give the root counts below.

The collection of three functions is invariant under cyclic permutation of $X,Y,Z$, and reversal is governed by the identity
$$
 (1+s)^qF_{A,C,B}\left(-\frac{s}{1+s}\right)
 =-F_{A,B,C}(s).
$$
We may therefore arrange $X\geq Y\geq Z$.

Suppose first that $X>Y>Z$. If $q>1$, the numbers of roots on the two open half-intervals are
$$
\begin{array}{ccc}
 &(-1,0)&(0,\infty)\\[2pt]
 F_{X,Y,Z}&m&1\\
 F_{Y,Z,X}&0&m\\
 F_{Z,X,Y}&1&0
\end{array}
\qquad (m\in\{0,1,2\}).
$$
The two occurrences of $m$ are equal. For $0<u<1$,
\begin{equation}\label{eq:cyclic-pair-convex}
\begin{aligned}
 F_{Y,Z,X}\left(\frac{u}{1-u}\right)=0
 &\quad\Longleftrightarrow\quad Z=Yu^q+X(1-u)^q\\
 &\quad\Longleftrightarrow\quad F_{X,Y,Z}(-(1-u))=0.
\end{aligned}
\end{equation}
The total is $2+2m\leq6$.

For $0<q<1$ the corresponding table is
$$
\begin{array}{ccc}
 &(-1,0)&(0,\infty)\\[2pt]
 F_{X,Y,Z}&0&1\\
 F_{Y,Z,X}&m&0\\
 F_{Z,X,Y}&1&m
\end{array}
\qquad (m\in\{0,1,2\}),
$$
and the possible additional roots are paired by
\begin{equation}\label{eq:cyclic-pair-concave}
\begin{aligned}
 F_{Z,X,Y}\left(\frac{u}{1-u}\right)=0
 &\quad\Longleftrightarrow\quad X=Zu^q+Y(1-u)^q\\
 &\quad\Longleftrightarrow\quad F_{Y,Z,X}(-(1-u))=0.
\end{aligned}
\end{equation}
Again the total is at most six.

It remains to include ties and compactified endpoints. In the following table, ``interior'' means the whole interval $(-1,\infty)$, including $s=0$, and $e$ is the number of distinct glued endpoint zeros:
$$
\begin{array}{ccccc}
\text{ordering}&q&
 \multicolumn{1}{c|}{\text{interior roots in cyclic order}}&e&\text{total}
 \\[2pt]
X>Y=Z&q>1&(3,1,1)&1&6\\
X>Y=Z&0<q<1&(1,m,m)&1&2+2m\leq6\\
X=Y>Z&q>1&(m,m,1)&1&2+2m\leq6\\
X=Y>Z&0<q<1&(1,1,3)&1&6
\end{array}
$$
Here again $m\in\{0,1,2\}$, and the paired entries follow from \eqref{eq:cyclic-pair-convex}--\eqref{eq:cyclic-pair-concave}. We verify the remaining entries directly.

Suppose $X>Y=Z$. The function $F_{X,Y,Y}$ always vanishes at $s=0$. When $q>1$, its restriction to $(-1,0)$ is strictly convex, initially decreases from zero, and ends at $X-Y>0$, while its restriction to $(0,\infty)$ initially increases from zero and has negative sign at the compactified right endpoint. It therefore has one further zero on each open half-interval. The function $F_{Y,Y,X}$ has one positive zero and $F_{Y,X,Y}$ has one negative zero. Their common compactified endpoint is counted once, giving the row $(3,1,1)$ and $e=1$. When $0<q<1$, concavity reverses the first conclusion: $F_{X,Y,Y}$ has only its zero at $s=0$. The $m$ roots of $F_{Y,Y,X}$ lie on the negative half-interval, whereas those of $F_{Y,X,Y}$ lie on the positive half-interval; they are paired by \eqref{eq:cyclic-pair-concave}. This gives $(1,m,m)$ and the same single glued endpoint.

Now suppose $X=Y>Z$. For $q>1$, the $m$ roots of $F_{X,X,Z}$ lie on the negative half-interval, whereas those of $F_{X,Z,X}$ lie on the positive half-interval; the reversal identity pairs them. The function $F_{Z,X,X}$ has only its zero at $s=0$, and the shared endpoint of the first two functions contributes $e=1$. Hence the row is $(m,m,1)$. When $0<q<1$, the function $F_{X,X,Z}$ has exactly one positive zero and $F_{X,Z,X}$ exactly one negative zero. The function $F_{Z,X,X}$ vanishes at $s=0$. On the negative half,
$$
 F_{Z,X,X}(-u)
 =X(1-u)^q+Zu^q-X
 =Zu^q-qXu+O(u^2)>0
$$
for small $u>0$, whereas its left-endpoint sign is $Z-X<0$. On the positive half,
$$
 F_{Z,X,X}(s)
 =X(1+s)^q-Zs^q-X
 =qXs-Zs^q+O(s^2)<0
$$
for small $s>0$, whereas its compactified right-endpoint sign is $X-Z>0$. Since the derivative has at most one zero on each half-interval, there is exactly one additional zero on each. Thus the last row is $(1,1,3)$ with $e=1$.

If $X=Y=Z$, each interval has the one interior zero $s=0$, and all three glued endpoints are zeros. There are no others, since
$$
 (1-u)^q+u^q<1,\qquad (1+s)^q-s^q>1\qquad(q>1),
$$
for $0<u<1$ and $s>0$, with both inequalities reversed when $0<q<1$. Thus the total is again six.
\end{proof}

\begin{proof}[Proof of Theorem~\ref{thm:balanced-triangle}]
Switching makes the graph unsigned, and Proposition~\ref{prop:potential-shift} reduces the proof to $b=0$; the original spectrum is obtained by adding $\beta$. If an edge is absent, the support graph is a forest, and the result follows from \cite[Theorem~3.7]{DeiddaPuttiTudisco2023}. We may therefore suppose all three weights are positive.

Set $q=p-1$. The constant vectors give the zero eigenvalue. Every nonconstant eigenvalue is positive. Summing the three eigenvalue equations therefore gives
\begin{equation}\label{eq:balanced-mean-zero}
 \nu_1\Phi_p(x_1)+\nu_2\Phi_p(x_2)+\nu_3\Phi_p(x_3)=0.
\end{equation}
Suppose first that no coordinate vanishes. Up to the antipodal action there is a unique negative coordinate. In the sign region $x_i<0<x_j,x_k$, scale so that $x_k=1$ and let
$$
 t=x_j>0,\qquad c=-x_i>0.
$$
Equation \eqref{eq:balanced-mean-zero} becomes
\begin{equation}\label{eq:balanced-c}
 \nu_i c^q=\nu_j t^q+\nu_k.
\end{equation}
Set
$$
 U=a_{ij}(c+t)^q,\qquad V=a_{ik}(c+1)^q,\qquad
 W=a_{jk}\Phi_p(t-1).
$$
The equations at $i,j,k$, respectively, are
\begin{equation}\label{eq:balanced-UVW}
 U+V=\lambda\nu_i c^q,\qquad
 U+W=\lambda\nu_jt^q,\qquad
 V-W=\lambda\nu_k.
\end{equation}
Eliminating $\lambda$ between the last two equations and using \eqref{eq:balanced-c} gives
\begin{equation}\label{eq:balanced-scalar}
 Y(1+s)^q-X\Phi_p(s)-Z=0,\qquad -1<s<\infty,
\end{equation}
where
$$
 s=\frac{c(t-1)}{c+t},\qquad
 X=a_{jk}\nu_i,\quad Y=a_{ik}\nu_j,\quad Z=a_{ij}\nu_k,
$$
and $1+s=t(c+1)/(c+t)$.

The change of variables is bijective. Equation~\eqref{eq:balanced-c} defines an increasing function $c=c(t)$, and differentiation gives $tc'(t)<c(t)$. Moreover,
$$
 s'(t)(c+t)^2=c(c+1)+t(t-1)c'(t).
$$
For $t\geq1$, the right-hand side is positive. If $0<t<1$, then
$$
 t(1-t)c'(t)<c(1-t)<c(c+1),
$$
so it is again positive. The endpoint limits are $-1$ and $+\infty$, so $t\mapsto s$ is a bijection from $(0,\infty)$ to $(-1,\infty)$. Conversely, a root of \eqref{eq:balanced-scalar} determines $t$ and then $c$. Define $\lambda$ by the equation at $k$ in \eqref{eq:balanced-UVW}; \eqref{eq:balanced-scalar} says exactly that the equation at $j$ has the same $\lambda$. The remaining equation follows because the sum of the three coordinates of the unsigned graph Laplacian is zero, while the right sides sum to zero by \eqref{eq:balanced-mean-zero}.

Choose a cyclic orientation and set
$$
 X=a_{23}\nu_1,\qquad Y=a_{31}\nu_2,\qquad Z=a_{12}\nu_3.
$$
The three open sign regions are parametrised by the zeros of
$$
 F_{X,Y,Z},\qquad F_{Y,Z,X},\qquad F_{Z,X,Y}.
$$
Their compactified endpoints are precisely the vectors with one zero coordinate. Direct substitution in the original equations gives the endpoint conditions preceding Lemma~\ref{lem:cyclic-root}; the common endpoint of two adjacent sign regions represents one eigenline, not two. A vector with two zero coordinates cannot satisfy \eqref{eq:balanced-mean-zero}. Thus the three cyclically glued intervals contain every nonconstant eigenline.

If $p\neq2$, Lemma~\ref{lem:cyclic-root} gives at most six nonconstant eigenlines, hence at most seven spectral values after adding the constant eigenline. If $p=2$, the problem is a generalised symmetric $3\times3$ matrix eigenvalue problem and has at most three distinct eigenvalues. These bounds are sharp within the stated class: Theorem~\ref{thm:complete-local-unfolding} gives seven values for $p>2$, Theorem~\ref{thm:low-Rn} gives seven for $1<p<2$, and at $p=2$ the unit-weight path has the three eigenvalues $0,1,3$.
\end{proof}

The same scalar reduction also describes the stable parameter regions. To separate the corresponding critical values, we need local continuation of the simple roots.

\begin{lemma}[Local continuation of simple triangle eigenpairs]
\label{lem:triangle-regular-continuation}
Fix $p>1$. Suppose that, near a solution $(s_0,\eta_0)$, a local scalar parametrisation of the projective eigenpair equation is given by $C^1$ maps $H$ and $\chi$, through
$$
 H(s,\eta)=0,\qquad [x]=\chi(s,\eta).
$$
Assume
$$
 \partial_sH(s_0,\eta_0)\ne0,\qquad
 D_s\chi(s_0,\eta_0)\ne0,
$$
and assume that the corresponding representative has no zero coordinate and that $x_i-\varepsilon_{ij}x_j\ne0$ on every edge. Then this solution continues as a locally unique $C^1$ projective eigenpair branch. Along an edge-weight variation,
\begin{equation}\label{eq:triangle-HF}
 \frac{\partial\lambda}{\partial a_{ij}}
 =\frac{|x_i-\varepsilon_{ij}x_j|^p}{\sum_k\nu_k|x_k|^p}.
\end{equation}
\end{lemma}

\begin{proof}
The scalar implicit-function theorem gives a unique $C^1$ root $s=s(\eta)$. Since $D_s\chi(s_0,\eta_0)\ne0$, this gives a valid local projective parametrisation. After shrinking the parameter neighbourhood, these conditions persist and $\Phi_p$ is $C^1$ there even when $1<p<2$. Thus the corresponding root is the locally unique $C^1$ eigenpair branch. Differentiating the Rayleigh quotient along the branch, the term containing the derivative of the eigenvector vanishes by criticality. This gives \eqref{eq:triangle-HF}.
\end{proof}

\begin{proposition}[Eigenline counts for balanced triangles]
\label{prop:triangle-regions}
Let $p>1$, $p\ne2$, and consider a balanced three-vertex datum satisfying $b_i=\beta\nu_i$, with all three edge weights positive. After switching and relabelling the vertices, set
$$
 q=p-1,\qquad
 X=a_{23}\nu_1,\qquad Y=a_{31}\nu_2,\qquad Z=a_{12}\nu_3,
$$
and assume $X>Y>Z>0$. If $q>1$, set
$$
 h_{\min}=
 \left(X^{-1/(q-1)}+Y^{-1/(q-1)}\right)^{1-q}.
$$
The numbers of nonconstant eigenlines are respectively
$$
 6,\ 4,\ 2
 \quad\hbox{according as}\quad
 Z>h_{\min},\ Z=h_{\min},\ Z<h_{\min}.
$$
If $0<q<1$, set
$$
 h_{\max}=
 \left(Z^{1/(1-q)}+Y^{1/(1-q)}\right)^{1-q}.
$$
The corresponding numbers are
$$
 6,\ 4,\ 2
 \quad\hbox{according as}\quad
 X<h_{\max},\ X=h_{\max},\ X>h_{\max}.
$$
Adding the constant eigenline gives total eigenline counts $7$, $5$, or $3$. In either open region separated by the relevant bifurcation hypersurface, all scalar roots are simple. Within either region, a generic sufficiently small perturbation of the edge weights, with $\nu$, $\beta$, and the balanced signature fixed, separates their critical values.
\end{proposition}

\begin{proof}
Under the scalar reduction in the proof of Theorem~\ref{thm:balanced-triangle}, for $q>1$ the paired roots counted by $m$ in Lemma~\ref{lem:cyclic-root} solve
$$
 Z=Yu^q+X(1-u)^q,\qquad0<u<1.
$$
The right side is strictly convex, has endpoint values $X$ and $Y$, and its minimum is $h_{\min}$. Since $Z<Y<X$, it has two, one, or no roots as $Z$ lies above, on, or below the minimum. The cyclic count is $2+2m$.

For $0<q<1$, the roots counted by $m$ solve
$$
 X=Zu^q+Y(1-u)^q.
$$
The right side is strictly concave, has endpoint values $Y$ and $Z$, and maximum $h_{\max}$. Since $X>Y>Z$, the same trichotomy holds with the displayed inequalities. Away from the unique extremizer, strict convexity or concavity makes the roots simple. In either open region, strict ordering keeps these roots away from the compactified endpoints and from configurations with a zero coordinate or a zero edge difference. Lemma~\ref{lem:triangle-regular-continuation} therefore gives locally unique branches. Along variations of the edge weights, \eqref{eq:triangle-HF} identifies the gradient of each critical value with the vector of normalised $p$-th powers of the absolute differences of the eigenfunction values across the three labelled edges. These gradients are distinct for distinct eigenlines. Indeed, normalise representatives $x$ and $y$ so that $\sum_i\nu_i|x_i|^p=\sum_i\nu_i|y_i|^p=1$. Equality of the gradients gives $|x_i-x_j|=|y_i-y_j|$ on all three edges, so the two labelled triples are related by an isometry of $\mathbb R$: $y=\pm x+c\one$. Both satisfy \eqref{eq:balanced-mean-zero}, whose left-hand side is strictly increasing in $c$; hence $c=0$. Thus the gradients of the finitely many critical-value branches are distinct, and a generic small edge perturbation separates their values.
\end{proof}

\begin{remark}
The proportionality $b_i=\beta\nu_i$ yields, after spectral shift, the constraint \eqref{eq:balanced-mean-zero}. The example in Appendix~\ref{app:nine-triangle} lies outside both this class and the balanced-signature class and has nine spectral values.
\end{remark}

\begin{remark}[Relation to tensor eigenpair counts]
For an orthogonally decomposable symmetric tensor whose decomposition vectors span the ambient space and whose coefficients are nonzero, Robeva's formula \cite[Theorem~2.3]{Robeva2016} explicitly enumerates, by support, the complex projective eigenvectors of $Tx^{d-1}=\lambda x$, the Euclidean-gradient tensor equation underlying E/Z-eigenpairs. At $p=4$ with unit measures, the graph equation is instead the H-eigenvalue equation $Ax^3=\lambda x^{[3]}$; general positive measures give a generalised eigenvalue problem for a pair of symmetric tensors. The H-eigenvalue terminology goes back to \cite{Qi2005}. These tensor counts concern fixed tensor equations, whereas Theorem~\ref{thm:intro-cellular-transition} concerns distinct real critical values under variation of positive graph edge weights.
\end{remark}

\section{An exact nine-value triangle}
\label{app:nine-triangle}

We verify the following nine-value example by exact algebraic computation.

\begin{proposition}[A nine-value triangle]
\label{prop:nine-triangle}
Let $p=4$, let every edge of the triangle have negative signature, and take
$$
 (a_{12},a_{13},a_{23})=(8739145,3903820,8576922),
$$
$$
\begin{gathered}
 b=(-3969860,-8425563,20915687727748),\\
 \nu=(12411,482,20915691304489).
\end{gathered}
$$
This graph has exactly nine eigenlines and nine distinct spectral values. Consequently
$
 B_3^{\mathrm{opt}}\geq9.
$
\end{proposition}

\begin{proof}
In the projective chart $x_3=1$, write $(x_1,x_2,x_3)=(x,y,1)$, and define
$$
 L_i=b_i x_i^3+\sum_{j\ne i}a_{ij}(x_i+x_j)^3.
$$
Eliminating $\lambda$ from the first and third, and from the second and third, coordinate equations gives
$$
 F=\nu_3L_1-\nu_1x^3L_3=0,
 \qquad
 G=\nu_3L_2-\nu_2y^3L_3=0.
$$
The exact resultant
$
 R(x)=\operatorname{Res}_y(F,G)\in\mathbb Z[x]
$
has degree $27$ and is square-free. Its Sturm chain has $18$ sign variations at $-\infty$ and $9$ at $+\infty$, so $R$ has exactly nine real roots. The penultimate subresultant is $A(x)y+C(x)$ and $\gcd(R,A)=1$. Hence $A(x_0)\ne0$ at every root $x_0$ of $R$. The subresultant specialisation theorem implies that the first nonzero specialised subresultant is $A(x_0)y+C(x_0)$, the greatest common divisor of $F(x_0,y)$ and $G(x_0,y)$. It has degree one. Thus every real root $x_0$ lifts to the unique real common zero $y=-C(x_0)/A(x_0)$.

The exact certificate in Remark~\ref{rem:nine-certificate} supplies nine disjoint rational $x$-intervals, each with a strict resultant sign change. Since the Sturm count gives nine real roots, these intervals exhaust them. Exact rational interval evaluation of
$$
 \lambda=\frac{b_3+a_{13}(x+1)^3+a_{23}(y+1)^3}{\nu_3}
$$
places the corresponding values, one each, in
$$
\begin{gathered}
 (-0.000003,-0.000002),\quad (0.067,0.068),\quad (0.083,0.084),\\
 (0.955,0.956),\quad (0.973,0.974),\quad (1,1.000000001),\\
 (1.00000002,1.00000003),\quad (1.000001,1.000002),
 (43735,43736).
\end{gathered}
$$
The intervals are pairwise disjoint. On the projective line $x_3=0$, the exact gcd of the third coordinate equation and the equation obtained by eliminating $\lambda$ from the first two coordinate equations is constant in the chart $x_2=1$. The remaining point $[1:0:0]$ fails the two required equations. Hence there are no further eigenlines on $x_3=0$.

Exact rational interval evaluation on the certified isolating boxes also proves that the nine eigenlines are nondegenerate: four are minima, four are saddles, and one is a maximum on $\mathbb RP^2$. The spectral shift $b\mapsto b+17481\nu$ gives the nonnegative integer potential
$
 (212986831,279,365648115381499957)
$
and translates the nine spectral values.
\end{proof}

\begin{remark}[Exact certificate]\label{rem:nine-certificate}
The machine-readable certificate and verification script can be found at \url{https://github.com/MColbrook/finite-graph-p-spectrum-certificate}. The script performs the resultant root count, the unique lifts, the interval comparisons, the exclusion of projective roots at infinity, and the Hessian signature calculations using exact integer and rational arithmetic.
\end{remark}

The exact triangle of Proposition~\ref{prop:nine-triangle} is nondegenerate, so its nine eigenlines persist under variation of the exponent.

\begin{corollary}[Persistence of the nine-value triangle]
\label{cor:nine-open-interval}
For the fixed datum in Proposition~\ref{prop:nine-triangle}, there is an open interval $J\Subset(2,\infty)$ containing $4$ such that, for every $p\in J$, the graph has exactly nine eigenlines and nine distinct spectral values. Four eigenlines have Morse index $0$ on $\mathbb RP^2$, four have index $1$, and one has index $2$. The spectral shift in Proposition~\ref{prop:nine-triangle} produces a nonnegative integer potential and preserves the nine eigenlines, their Morse indices, and the distinctness of their spectral values. In particular,
$$
 B_3(p)\geq9\qquad(p\in J).
$$
\end{corollary}

\begin{proof}
On a compact subinterval of $(2,\infty)$, the projective Rayleigh quotient depends continuously in the $C^2$ topology on $p$. The exact rational interval calculations show that all nine critical points at $p=4$ are nondegenerate, with the stated indices. The implicit-function theorem continues them uniquely and preserves their indices.

Choose pairwise disjoint projective neighbourhoods of the nine points. The gradient of the $p=4$ quotient is bounded away from zero on the compact complement. $C^1$ continuity in $p$ excludes additional critical points there when $p$ is sufficiently close to $4$. The finitely many critical-value gaps remain nonzero. Adding $s\nu$ to the potential adds the constant $s$ to the quotient and changes neither critical points, Hessians, nor value gaps.
\end{proof}

\begin{corollary}\label{cor:small-bounds}
The optimal bounds in dimensions at most four satisfy
$$
 B_1^{\mathrm{opt}}=1,\qquad B_2^{\mathrm{opt}}=2,\qquad
 9\leq B_3^{\mathrm{opt}}\leq\widehat{\mathfrak B}_3,\qquad
 26\leq B_4^{\mathrm{opt}}\leq\widehat{\mathfrak B}_4.
$$
Within the class of balanced three-vertex graphs satisfying $b_i=\beta\nu_i$, the sharp bound is seven.
\end{corollary}

\begin{proof}
For $n=1$, the eigen-equation is $b_1\Phi_p(x_1)=\lambda\nu_1\Phi_p(x_1)$. Since $x_1\ne0$, one has $\sigma_p(d)=\{b_1/\nu_1\}$ and hence $B_1^{\mathrm{opt}}=1$.

The upper bounds follow from Theorem~\ref{thm:intro-main}\ref{part:intro-effective}. The remaining assertions follow from Theorem~\ref{thm:two-vertex}, Theorem~\ref{thm:balanced-triangle}, Corollary~\ref{cor:ordinary-small-growth}, Proposition~\ref{prop:nine-triangle}, and Theorem~\ref{thm:intro-growth-chambers}(ii).
\end{proof}

\bibliographystyle{amsplain}
\bibliography{finite_graph_p_spectrum}

\end{document}